\documentclass[11 pt,reqno]{amsart}
\usepackage{amsmath,amsthm,amsfonts,amssymb,mathrsfs,bm,graphicx,stmaryrd,dsfont}
\usepackage[usenames,dvipsnames]{color}
\usepackage{pkgfile}
\usepackage{longtable}
\usepackage[colorlinks=true,linkcolor=blue]{hyperref}
\newcommand\norm[1]{\left\lVert#1\right\rVert}
\usepackage[letterpaper,hmargin=1.0in,vmargin=1.0in]{geometry}
\parskip 	\smallskipamount

\newtheorem{theorem}{Theorem}[section]

\newtheorem{proposition}[theorem]{Proposition}

\newtheorem{maintheorem}{Theorem}

\newcommand{\Ba}{{\mathbf{Barrier}}}
\newcommand{\Ub}{{\mathbf{Tile}}}

\newcommand{\Di}{{\mathbf{Dilation}}}
\newcommand{\An}{{\mathbf{Corridor}}}
\newcommand{\Bb}{{\mathbf{Box}}}
\newcommand{\Lb}{{\mathbf{L-Box}}}
\newcommand{\PT}{{\mathbf{PT}}}
\newcommand{\Geo}{{\mathbf{Geo}}}

\newcommand{\bw}{{\mathbf{w}}}
\newcommand{\Tt}{{\mathbf{T}}}

\newcommand{\SST}{{\mathbf{Stable}}}

\newcommand{\US}{{\mathbf{Unstable}}}
\newcommand{\Boo}{{\mathbf{Boosting}}}
\newcommand{\Fav}{{\mathbf{Fav}}}

\newcommand{\Gr}{{\mathbf{Grid}}}
\newcommand{\Pro}{{\mathbf{Proj}}}
\newcommand{\Bas}{{\mathbf{Base-event}}}
\newcommand{\Sc}{{\mathbf{S}}}

\newcommand{\fn}{{\mathfrak{n}}}
\newcommand{\fm}{{\mathfrak{m}}}
\newcommand{\brj}{{\llbracket{1,2^j}\rrbracket^2}}

\newcommand{\fb}{\mathfrak{B}}
\newcommand{\nin}{\noindent}
\newcommand{\e}{\varepsilon}
\begin{document}
\title[Large Deviation in FPP]{Upper Tail Large Deviations in First Passage Percolation}
\author{Riddhipratim Basu}
\address{Riddhipratim Basu, International Centre for Theoretical Sciences, Tata Institute of Fundamental Research, Bangalore, India}
\email{rbasu@icts.res.in}
\author{Shirshendu Ganguly}
\address{Shirshendu Ganguly, Department of Statistics, UC Berkeley, Berkeley, CA, USA}
\email{sganguly@berkeley.edu}
\author{Allan Sly}
\address{Allan Sly, Department of Mathematics, Princeton University, Princeton, NJ, USA}
\email{allansly@princeton.edu}
\date{}
\maketitle
\begin{abstract} For first passage percolation on $\Z^2$ with i.i.d.\ bounded edge weights, we consider the upper tail large deviation event; i.e., the rare situation where the first passage time between two points at distance $n$, is macroscopically larger than typical. It was shown by Kesten \cite{Kes86} that the probability of this event decays as $\exp (-\Theta(n^2))$. However the question of existence of the rate function i.e., whether the log-probability normalized by $n^2$ tends to a limit, had remained open. We show that under some additional mild regularity assumption on the passage time distribution, the rate function for upper tail large deviation indeed exists.  Our proof can be generalized to work in higher dimensions and for the corresponding problem in last passage percolation as well. The  key intuition behind the proof is that a limiting metric structure which is atypical causes the upper tail large deviation event.  The formal argument then relies on an approximate version of the above  which allows us to dilate the large deviation environment to compare the upper tail probabilities for various values of $n.$ 
\end{abstract}
\tableofcontents

\section{Introduction and main result}
First passage percolation is a popular model of fluid flow through inhomogeneous random media, where one puts random weights on the edges of a graph and considers the first passage time between two vertices, which is obtained by minimizing the total weight among all paths between the two vertices. First passage percolation on Euclidean lattices was introduced by Hammersley and Welsh \cite{HW65} in 1965 and has been studied extensively both in statistical physics and probability literature ever since. This model served as one of the motivations of developing the theory of subadditive stochastic processes and the early progresses using subadditivity was made by Hammersley-Richardson-Kingman \cite{HW65,K73,R73} and culminated in the proof of the celebrated Cox-Durrett shape theorem \cite{CD81} establishing the first order law of large number behaviour for passage times between far away points. Further progress was made into the 80s and 90s through efforts of Kesten \cite{Kes86,Kes87,Kes93} and Talagrand \cite{Tal94} establishing concentration inequalities for passage times; Newman and others \cite{New95} on more geometric aspects of the model. Much progress has been made since \cite{BKS04,H08} including a flurry of results in the last five years \cite{Cha11, AD14, DH14, DH17}. Despite this impressive progress, most of the fundamental questions still remain major mathematical challenges, see the survey \cite{ADH15} for a comprehensive history as well as an extensive list of the major open problems in this field. 

One other reason planar first passage percolation came into prominence is that this model is believed to be in the KPZ universality class that was introduced by Kardar, Parisi and Zhang \cite{KPZ86} in 1986. Using non rigorous renormalization group techniques, KPZ predicted universal scaling exponents for many (1+1)-dimensional growth models including first and last passage percolation under very general conditions on the passage time distribution (precise definitions later). An explosion of rigorous results in the last 18 years starting with the seminal work of Baik, Deift and Johansson \cite{BDJ99} has now verified the KPZ prediction for a handful of models including last passage percolation with Exponential, Geometric or Bernoulli passage times. However, this progress has been mostly restricted to the so-called exactly solvable (or, integrable) models where exact formulae are available using deep connections to algebraic combinatorics, representation theory and random matrix theory; and extremely detailed information has been obtained about such models by analyzing those formulae. Although the same results are qualitatively expected to hold for a much larger class of models, these methods rely very crucially on the exact formulae, and moving beyond the exactly solvable models remains a major challenge. 

Our focus in this paper is such a problem in the non-integrable setting of first passage percolation in the large deviation regime. The question first arose in the work of Kesten \cite{Kes86} who considered the probability of large deviation events in first passage percolation. Postponing the precise definitions momentarily, let us first describe informally the set-up. Consider the passage time $\T_n$ from $(0,0)$ to $(n,0)$. The shape theorem dictates that under some regularity conditions $\frac{\Tt_n}{n}\to \mu$ almost surely for some $\mu \in (0,\infty)$. The study of large deviations is concerned with the unlikely events $\{\Tt_n \geq (\mu+\e)n\}$ (upper tail) and $\{\Tt_n \leq (\mu-\e)n\}$ (lower tail). In the classical theory of large deviations, log of such probabilities suitably scaled (by the so-called speed of large deviations) converges to a function of $\e$, known as the rate function. For first passage percolation, Kesten \cite{Kes86} showed the large deviation speed of $n$ and existence of the rate function for the lower tail using a subadditive argument. For the upper tail Kesten showed a large deviation speed on $n^2$ for bounded edge weight distribution, however the existence of rate function remained open (see Open Question 18, in \cite{ADH15}). Our main result in this paper (see Theorem \ref{t:ldp} below) answers this question establishing the existence of rate function for the upper tail, thereby establishing first such result beyond the exactly solvable models.

\subsection{Model definitions and statement of result}\label{mods}
We start with formal definitions of standard first passage percolation on $\Z^d$, $d\geq 2$. Let $E(\Z^d)$ denote the set of all nearest neighbour edges  in $\Z^d$. Let $\nu$ be a probability measure supported on the non-negative real line. Let $\Pi=\{X_e: e \in E(\Z^d)\}$ denote a field of i.i.d.\ random variables where each $X_e$ (called the passage time of the edge $e$) has distribution $\nu$. For a sequence $\gamma=e_1e_2\cdots e_{k}$ of neighbouring edges (called a path), the passage time of the path, denoted by $\ell(\gamma)$,\footnote{{For brevity of notation we shall often denote $\ell(\gamma)$ by $|\gamma|$}.} is defined as 
$$\ell(\gamma)=\sum_{i=1}^{k} X_{e_i}.$$
For any two vertices $u$ and $v$, the first passage time between $u$ and $v$, denoted $\PT(u,v)$ is defined as the infimum of $\ell(\gamma)$ where $\gamma$ varies over all paths starting at $u$ and ending at $v$. Let $\mathbf{0}$ denote the origin. Under very mild conditions on $\nu$, it is a fundamental fact that for all $v\in \Z^d$, there exists $\mu(d,\nu,v)\geq 0$ such that 
$$\lim_{n\to \infty}\frac{\PT(\mathbf{0},nv)}{n}=\mu(d,\nu,v)$$ 
almost surely. For the special case when $v=(1,0,\ldots, 0)$ denote the unit vector along the first co-ordinate, we denote the limiting constant by just $\mu$, also known as the time constant in the literature. 
For the rest of this paper we shall focus on the planar case ($d=2$) of the above model. Although our main result extends to higher dimensions with little to no change, we choose to work in two dimensions to avoid additional notational overhead. From now on, we shall be in the setting of standard first passage percolation on $\Z^2$ unless otherwise mentioned. Let $\mathbf{n}:=(n,0)$ and let us denote the passage time $\PT(\mathbf{0},\mathbf{n})$ by $\Tt_n$. As mentioned above we are concerned with the probability of the upper tail large deviation event: 
\begin{equation}\label{utdef}
\sU_{\zeta}(n):=\{\Tt_{n}\ge (\mu+\zeta) n\}.
\end{equation}
for some $\zeta>0$.  Throughout the paper we will assume the rather general condition that the passage time distribution has a continuous density on a compact interval $[0,b]$.  Even though we believe our proof methods can be used to extend our result beyond this assumption, the former will help make some of the proofs cleaner. For future reference we record this assumption below.
\begin{defn}
\label{d:pt}
For $b>0$, let $\cP(b)$ denote the set of all probability measures with support $[0,b]$ and a continuous density. 
\end{defn}

It is well known that if $\nu\in \cP(b)$ for any $b>0$, then we have $0<\mu<b$ (e.g.\ see \cite{HW65}). Also observe that for $\nu\in \cP(b)$, we have deterministically that $\Tt_{n}\leq bn$. So while considering the large deviation event $\sU_{\zeta}$ in the above scenario it suffices to consider $\zeta \in (0,b-\mu)$. Our main theorem shows that the large deviation rate function exists in the above setting. 

\begin{maintheorem}
\label{t:ldp}
Consider standard first passage percolation on $\Z^2$ with passage time distribution $\nu\in \cP(b)$ for some $b>0$. Then for $\zeta\in (0,b-\mu)$ there exists $r=r(\nu,\zeta)\in (0,\infty)$ such that 
$$\lim_{n\to \infty} -\frac{\log \P (\sU_{\zeta}(n))}{n^2}=r.$$
\end{maintheorem}

A couple of remarks are in order. First, there is nothing special about the direction $(1,0)$; the same result holds for any unit vector $v$ with different rate function $r$, with minor adjustments in the proof. Also, a variant of this result holds in higher dimensions as well where the speed of the large deviation is $n^{d}$ rather than $n^2$ (See e.g., \eqref{e:ut} ). The same argument proving Theorem \ref{t:ldp} can be used to prove the higher dimensional analogue. However, in this paper we shall only concentrate on proving Theorem \ref{t:ldp}.  

Observe that the condition in Theorem \ref{t:ldp} is not optimal and we have not made an attempt to make it the weakest possible. It is however important to observe that some condition is needed to ensure even the $n^2$ speed of the large deviation. Together with the standard assumptions that the mass at $0$ is less than the critical bond percolation probability on $\Z^2$ and that the edge distribution is not degenerate at a single point, Kesten assumed boundedness. It is easy to see that the boundedness assumption cannot be completely removed. For example, if the passage times are exponentially distributed, just increasing all the passage times around the origin by $(\mu+ \zeta) n$, would force the large deviation event, while its probability being only exponentially small in $n$. One can however prove Kesten's result for passage times with sufficiently fast decaying tails, and one believes that the rate function will exist in such a case too possibly under some additional assumptions. However, in this paper we have not pursued those directions, and instead focussed on proving the result in the simplest possible case that is still sufficiently general to be of interest.

\subsection{Background and Related Works}
First passage percolation can be thought of as putting a random metric on $\Z^d$, where the distance between two vertices is given by the first passage time between them. As alluded to above, the most fundamental result about first passage percolation says that under suitable rescaling these metrics converge almost surely to a deterministic metric on $\R^d$ in a pointed Gromov-Hausdorff sense. More precisely we have the following. Suppose $\nu\in \cP(b)$ for some $b\in (0,\infty)$ (actually the result is valid more generally, one only needs some moment condition and that the mass of any atom at $0$ is sufficiently small), and let $\tilde{B}(t)$ denote the set of all vertices that are within distance $t$ of $\mathbf{0}$ in the FPP metric, and let $B(t)=\tilde{B}(t)+[-\frac{1}{2}, \frac{1}{2}]^{d}$. Then there exists a non-random compact convex set $\cB=\cB_{\nu}$ with obvious symmetries such that for each $\e>0$
\begin{equation}
\label{e:shape}
\P\left((1-\e)\cB_\nu \subset \frac{B(t)}{t} \subset (1+\e)\cB_{\nu} \text{ for all large }t \right)=1.
\end{equation}
The set $\cB$ is called the limit shape for this model. Recall the limiting constant $\mu(d,\nu,v)$ in direction $v$. It is not hard to see that $\mu(d,\nu,\cdot)$ can be extended to a norm in $\R^d$ and $\cB$ is the unit ball corresponding to this norm. The shape theorem implies that at large scales, the distance function in the FPP metric in a fixed direction grows approximately linearly, and the convexity of the limit shape is then just a consequence of triangle inequality. 
 
The shape theorem is a law of large number result, and the natural next question of obtaining fluctuations has been extensively investigated. The moderate deviation estimates are interesting, as in $d=2$, KPZ scaling predicts a fluctuation exponent of $1/3$, however the best known fluctuation and concentration bounds (for $\Tt_n$) have so far been proved at $n^{1/2+o(1)}$ scale \cite{Kes93,Tal94, BKS04}. In this paper, we are looking at the large deviation regime, i.e., where we consider a linear deviation of $\Tt_n$ from its long term value. Although we recall standard results only for $\Tt_n$; qualitatively same results hold in all directions. Also we are assuming throughout that the passage time distribution is in $\cP(b)$ for some $b$ although many of these results hold under weaker assumptions.

Kesten \cite{Kes86} considered both upper and lower tail large deviations for first passage percolation. Let $\sL_{\zeta}(n):=\{\Tt_{n}\le (\mu-\zeta) n\}$ (throughout this section for brevity we will use $\Tt_n$ to denote the passage time between $(0,0,\ldots,0)$ and $(n,0,\ldots,0)$ in $\Z^d$ although it was initially defined only for $\Z^2$) denote the lower tail large deviation event. Using a subadditive argument, Kesten showed that for $\zeta\in (0,\mu)$, 
\begin{equation}
\label{e:lt}
\lim_{n\to \infty} -\frac{\log \P(\sL_{\zeta}(n))}{n}=r_{\ell}(\zeta) \in (0,\infty). 
\end{equation}

For the upper tail large deviations, Kesten showed that 
\begin{equation}
\label{e:ut}
 0 < \liminf_{n\to \infty} - \frac{\log \P(\sU_{\zeta}(n))}{n^d} \leq \limsup_{n\to \infty} - \frac{\log \P(\sU_{\zeta}(n))}{n^d}<\infty.
\end{equation}
The existence of the limit was left open and this open question was re-iterated in \cite{ADH15} (See Question 18), which we answer in our Theorem \ref{t:ldp}.

Observe that the speed of large deviations is different in upper and lower tails. This is not unexpected and can be intuitively explained as follows. For $\Tt_n$ to be much smaller than $\mu n$, one needs only one path that is atypically small; however it is much more unlikely for $\Tt_n$ to be atypically large, since typically one can find $n^{d-1}$ many `parallel' short paths between the origin and $(n,0,0,\ldots,0)$ which are disjoint except at the beginning and the end. Thus to attain the upper tail event all such paths need to be large, each of which costs $e^{-\Theta(n)}$ and hence the total cost is at least $(e^{-\Theta(n)})^{n^{d-1}}$. Indeed this feature is quite common in many growth models, e.g.\ last passage percolation, parabolic Anderson model and deviation of the spectrum of GUE (see \cite{cranston} and the references therein).

As a matter of fact, among the only cases of growth models where the existence of rate function is known for both tails are the so-called exactly solvable models of last passage percolation. As an illustration, we only describe the result for the case of exponential directed last passage percolation in $\Z^2$ \cite{Jo99}; however the same qualitative result is known in the case of Poissonian directed last passage percolation in $\R^2$ \cite{DZ1} and last passage percolation on $\Z^2$ with geometric edge weights \cite{Jo99}. Consider the following last passage percolation model on $\Z^2$ where each vertex is equipped with an i.i.d.\ sample of $\mbox{Exp}(1)$ random variable. As before, the weight of any path is the sum of weights on it. The difference from the first passage percolation model is that we only consider up/right directed paths and the last passage time between two vertices is calculated by maximizing the weight over all such paths between the two vertices. This is one of the first exactly solvable models rigorously shown to be in the KPZ universality class by Johansson \cite{Jo99} using exact determinantal formulae.  Let $L_n$ denote the last passage time from $(0,0)$ and $(n,n)$. It is well known \cite{Ro81} that $\frac{L_n}{n}\to 4$ almost surely as $n\to \infty$. Johansson proved large and moderate deviation estimates for $L_n$. In particular he proved that  
$$\lim_{n\to \infty} \frac{\log \P(L_{n}\geq (4+\zeta)n)}{n}=-I_{u}(\zeta);\quad \zeta>0~\text{and}$$
$$\lim_{n\to \infty} \frac{\log \P(L_{n}\leq (4-\zeta)n)}{n^2}= -I_{\ell}(\zeta);\quad \zeta \in (0,4).$$
The functions $I_{\ell}$ and $I_{u}$ could in principle be explicitly evaluated there. Observe that for last passage percolation, as expected, the role of upper tail and lower tail is reversed but qualitatively there is no other difference from the FPP case.
We list below a  few other results worth mentioning: a similar result as above in the  context of Poissonian LPP by Dueschel and Zeitouni in 
\cite{DZ1}.
Still within KPZ universality class, but in the framework of particles systems, functional large deviation principle for Totally Asymmetric Simple Exclusion Process (TASEP), which is closely connected to Exponential LPP, was obtained, for the $n$-speed tail by Varadhan and Jensen \cite{varadhan1, Jensen00} and for the $n^2$-speed tail recently by Olla and Tsai in \cite{OT17}.

However the above results concerning LDP at speed $n^2$, use some form of integrability and the proofs rely heavily on the nature of the passage time distributions which are intimately connected to the integrable features in these models.  Although the large deviation behaviour is expected to be universal, the existence of the rate function was not even known for any other non-integrable model of last passage percolation. It is left to the reader to check that all our arguments will remain valid, in fact become simpler, for a general last passage percolation model (with bounded edge weights, say).

Although as far as we are aware, our result is the first one proving the existence of a large deviation rate function for the $n^2$-speed tail for point to point passage times in a non-integrable setting, one variant of such a result was proved by Chow-Zhang \cite{CZ} in the case of line-to-line first passage time in standard first passage percolation where the open problem addressed by Theorem \ref{t:ldp} was also mentioned.
Formally  Chow-Zhang considers the minimum passage time over all paths with one endpoint in $A=\{(0,i):i\in \{0,1,\ldots, n\}\}$ and the other endpoint in $B=\{(n,i):i\in \{0,1,\ldots, n\}\}$ and moreover they consider the geodesic restricted to lie in the square $[0,n]^2.$ Let us denote the passage time by $\Tt^*_n$. It is a standard result \cite{Kes86} that $\frac{\Tt^*_n}{n}\to \mu$ almost surely as $n\to \infty$. In \cite{CZ}, Chow and Zhang showed that for $\zeta>0$
$$\lim_{n\to \infty} -\frac{\log \P(\sU_{\zeta}(n))}{n^2}$$
exists and is nontrivial. The appropriate variant of their result holds in all dimensions. Even though the  specific geometric setting considered in \cite{CZ} causes significant simplification, and in particular rules out backtracks of the geodesic and does not create a necessity for the metric space dilation approach in this paper,  it is worth mentioning that  the argument in \cite{CZ} is a multi sub-additive argument, which bears resemblance with our  approach at least at a high level (see Section \ref{outline} for more details).

Finally we end this section with a brief discussion about a related line of work concerning geometric consequences of large deviation events in first/last passage percolation.  Formally one considers the measure obtained by conditioning on the large deviation events, and investigates how does the geometry of the random field of weights change? These questions were considered in the setting of exactly solvable Poissonian last passage percolation for the upper tail (i.e., the tail with large deviation speed $n$) by Deuschel and Zeitouni who, in \cite{DZ1}, showed that under the upper tail large deviation event, the maximizing paths between two far away points is with high probability localized around the straight line segment joining the two endpoints. For the harder lower tail case, in a recent paper \cite{BGS17A} we showed that forcing the large deviation event makes the path delocalized with high probability. Although we choose to work in the setting of last passage percolation in the latter, our argument goes beyond the integrable setting under certain distributional assumptions. (see remarks in \cite{BGS17A} for more details.)

\subsection{A brief outline of the paper}\label{outline}
The argument of proving Theorem \ref{t:ldp} is quite involved and has many pieces going into the proof. The purpose of this section is to provide a broad overview of the steps of the argument. At a very high level, our argument intuitively is predicated on the existence of  a limiting metric structure  as in \eqref{e:shape} even in the upper tail large deviation regime, which roughly implies that conditional on the large deviation event, the distances in a fixed direction grow linearly at large scales, and as the direction is varied the gradient changes in a reasonably regular way.
The reason to expect this is intimately tied to the reason behind the $n^2$ speed of large deviation, which causes the edge distributions of $\Theta(n^2)$ many edges to change. 

Although we believe the above statement to be true, for the purposes of the proof it suffices to have sub-sequential limits. In fact the exact statement that we prove in much less refined. 
(see Proposition \ref{t:stable}). 

For the remainder of the paper let $b>0$ and $\nu\in \cP(b)$ be fixed. Recall that $\mu$ denotes the time constant  in the $x$-direction for the standard first passage percolation on $\Z^2$ with $\nu$-distributed edge weights. Let $\zeta\in (0, b-\mu)$ be fixed. For $n\in \N$, let $a_{n}=a_n(\zeta)$ be defined by 
$$ a_{n}=\log \P(\sU_{\zeta}(n)).$$
Theorem \ref{t:ldp} will follow easily from the following multi-subadditive result. 

\begin{proposition}
\label{p:subadditive}
For each $\e>0$, there exists $N_0>0$ such that the following holds. For all $n\in \N$ with $n>N_0$ there exists $M_0=M_0(n)$ such that for all $m>M_0$ we have 
$$ \frac{a_{m}}{m^2} \geq \frac{a_{n}}{n^2} -\e.$$ 
\end{proposition}

Most of this paper is devoted to proving Proposition \ref{p:subadditive}, but before we outline its proof let us quickly finish the proof of  Theorem \ref{t:ldp} assuming the above. 

\begin{proof}[Proof of Theorem \ref{t:ldp}]
Let 
\begin{align}\label{limsupinf}
\kappa &= \limsup_{n\to \infty} \frac{a_n}{n^2};~\text{and}\\
\kappa' &= \liminf_{n\to \infty} \frac{a_n}{n^2}.
\end{align}
By Kesten's result \eqref{e:ut}  we know that $-\infty < \kappa' \leq \kappa <0$ and hence it suffices to prove that for all $\e>0,$ we have $\kappa'\geq \kappa -2\e$. Fix $\e>0$ and let $N_0$ be such that the conclusion of Proposition \ref{p:subadditive} holds. Pick $N_1>N_0$ such that $\frac{a_{N_1}}{N_1^2}\geq \kappa -\e/2$, and pick $N_2> M_0(N_1)$ as in Proposition \ref{p:subadditive} such that  
$\frac{a_{N_2}}{N_2^2}\leq \kappa' +\e/2$. Proposition \ref{p:subadditive} now implies that $\kappa'\geq \kappa-2\e$, as required. This completes the proof of the theorem.
\end{proof}

The rest of this paper proves Proposition \ref{p:subadditive}. Observe that to prove the proposition, we need to obtain a lower bound to $\P(\sU_{\zeta}(m))$ in terms of $\P(\sU_{\zeta}(n))$ for $m\gg n\gg 1$. First (and the most important) step is to construct an event with probability at least $\P(\sU_{\zeta}(n))^{m^2/n^2}$ (upto an error of $e^{-o(m^2)}$) on which we shall have $\{\Tt_{m} \geq (\mu+\zeta') m\}$ for $\zeta'$ smaller but arbitrarily close to $\zeta$. 

\textbf{Throughout the article, for notational brevity we will be omitting the floor signs and ignoring any rounding issue since they will not have any effect on the nature of the arguments.}

Formally we have the following proposition. 

\begin{proposition}
\label{p:pathconstruct} 
For each $\e'\in (0,\zeta)$ and $\e>0$, there exists $N_0$ and $H_0$ such that for all $n>N_0$ and $m>nH_0$ we have 
$$ \log \P(\sU_{\zeta-\e'}(m)) \geq \frac{m^2}{n^2}\log \P(\sU_{\zeta}(n))-\e m^2.$$
\end{proposition}

Once we have Proposition \ref{p:pathconstruct} at our disposal, all we need to prove Proposition \ref{p:subadditive} is a way to compare $\P(\sU_{\zeta}(n))$ and $\P(\sU_{\zeta'}(n))$ when $\zeta$ and $\zeta'$ are close. To this end we have the following proposition which essentially says that if the rate function exists it must be continuous in $\zeta$. 

\begin{proposition}
\label{p:cont}
For each $\e>0$, there exists $\e'>0$ such that for all $n$ sufficiently large we have 
$$ \frac{\log \P(\sU_{\zeta-\e'}(n))}{n^2} \leq \frac{\log \P(\sU_{\zeta}(n))}{n^2}+\e.$$
\end{proposition} 
Our assumption of the edge distribution possessing a continuous density (see Definition \ref{d:pt}) is essentially only used in the proof of the above. Although this result can be proven much more generally we have not made such an attempt in this paper.
It is easy to complete the proof of Proposition \ref{p:subadditive} using Propositions \ref{p:pathconstruct} and \ref{p:cont}. 

\begin{proof}[Proof of Proposition \ref{p:subadditive}]
The proof follows immediately by noticing that $$\frac{a_{m}}{m^2} \geq  \frac{\log \P(\sU_{\zeta-\e'}(m))}{m^2}-\e \ge  \frac{a_{n}}{n^2} -2\e,$$ where the first inequality is the content of Proposition \ref{p:cont} and the second inequality is the content of Proposition \ref{p:pathconstruct}.
\end{proof}

The rest of this paper deals with proving Propositions \ref{p:pathconstruct} and \ref{p:cont}. Proof of Proposition \ref{p:cont} is easier. Essentially one shows that to change the passage time $\Tt_n$ by $\e' n$ it suffices to increase the passage times of all the edges inside a box of size $O(n)$ by $O(\e')$.  The cost of such a change can be made as small as possible in the exponential scale by choosing $\e'$ small enough and using the continuity of the density of $\nu$. The only subtle point is that since the variables are supported on $[0,b]$, one cannot increase the values of the edges  that already have values close to $b.$ However by choosing the parameters carefully we ensure that there are not too many edges of the latter kind and that the geodesic necessarily passes through many edges whose values are away from $b$ for which the perturbation strategy works.  The formal proof appears in Section \ref{s:cont}.  The remainder of this section presents an outline of the proof of Proposition \ref{p:pathconstruct}, which is really the heart of this paper.   
\begin{figure}[h]
\centering
\includegraphics[scale=.5]{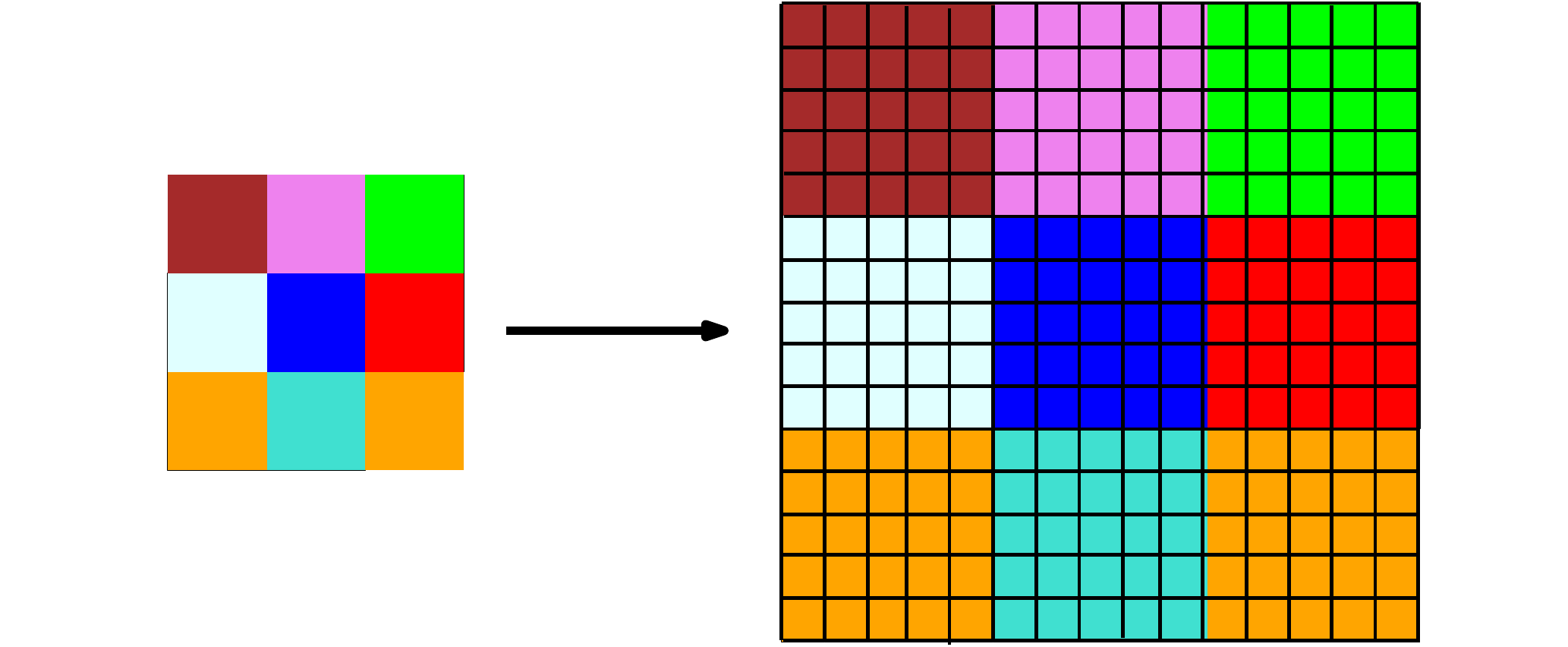}
\caption{As outlined below, the main idea behind the proof of Theorem \ref{t:ldp} involves dilating the limiting metric space structure. The figure illustrates a situation when the dilation factor is $5$.}
\label{fig1}
\end{figure}

For the purpose of facilitating illustration, we shall only outline the proof in the special case $m=2n$. Also we shall pretend, for the time being that the event $\{\Tt_n \geq (\mu+\zeta)n\}$ only depends on the edges weights in the box $B=\llbracket 0, n\rrbracket \times \llbracket-\frac{n}{2}, \frac{n}{2}\rrbracket$ where $\llbracket a,b \rrbracket:= [a,b]\cap \Z$. Observe that this is not deterministically true because the paths are allowed to backtrack. However, we pretend this for the moment for the sake of exposition. In fact the above is true with high probability if one replaces $B$ by a box of side length being a large ($\nu$ dependent) constant times $n$ and centered at the origin. This is what we will do throughout the rest of the paper.

Let $\e$ be an arbitrary small positive number. Suppose that $\P(\Tt_{n}\geq (\mu+\zeta)n)=p$. So our task is to create an environment on $B_1=\llbracket 0, 2n\rrbracket \times \llbracket-n, n\rrbracket$ with probability  at least $p^4$ (upto an error $e^{-o(n^2)}$ on which we shall have $\{\Tt_{2n}\geq (\mu+\zeta-\e)2n\}$. The basic idea of such a construction is as follows. We condition on the large deviation event $\{\Tt_{n}\geq (\mu+\zeta)n\}$ and look at an environment $\omega$ in $B$. We show that with high probability $\omega$ is such that $B$ can be tiled by sub-boxes of size $k\times k$  which we will call `tiles' (see Figure \ref{fig2}), most of the tiles are \textbf{stable}. We describe below roughly the notion of stability which makes precise the notion of a limiting metric space structure as alluded to at the beginning of Section \ref{outline}.
 
\begin{itemize}
\item Consider a tile and for each $\bz$ in the tile and any $\theta \in \bS^1$, let $\bz_1$ and $\bz_2$ be points such that $\bz,\bz_1,\bz_2$ lie in a straight line making angle $\theta$ with the $x-$axis and $\|\bz-\bz_1\|_2=\|\bz_1-\bz_2\|_2=k$ where $\|\cdot\|_2$ denotes the Euclidean norm.   Then the box is said to stable if 
\begin{equation}\label{stab908}
\PT(\bz,\bz_1)=(1+o(1))\PT(\bz_1,\bz_2)=(1+o(1))\frac{\PT(\bz,\bz_2)}{2},
\end{equation}
\end{itemize}
(see Section \ref{gradstabsec} for formal definitions). Thus the above says that for any $\bz$ in the tile, the passage time from $\bz$ acts approximately like a linear function in every direction $\theta$ at scale $k.$  Note that this linear function a priori depends on the environment $\omega.$ However it can be shown that with significant probability the environments $\omega$ approximately yield the same linear function.
Even though for the actual proof we will need a result stronger in many senses, below we illustrate how to exploit the above property.
\begin{figure}[h]
\centering
\includegraphics[scale=.5]{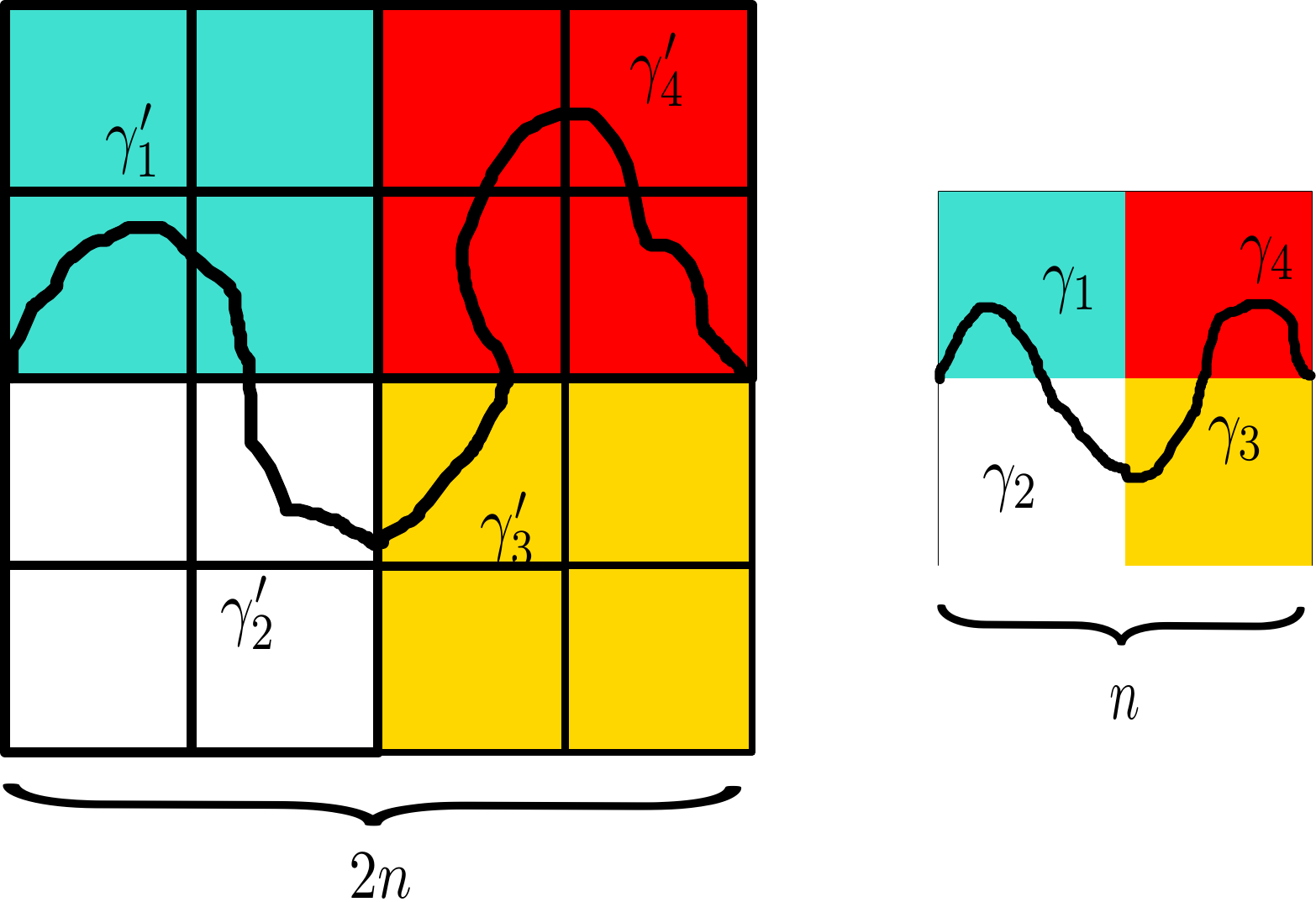}
\caption{Figure illustrating the proof sketch below where for every path $\gamma'$ in the dilated environment $\omega'$ there exists a path $\gamma$ obtained by scaling down the endpoints of the excursions. However note that a priori $\gamma_{i}$ need not be excursions even though $\gamma'_i$ are by definition. The former are just taken to be the shortest path in the environment $\omega$ between the end points of $\gamma'_i$ divided by $2.$  Here $k=\frac{n}{2}.$}
\label{fig2}
\end{figure}

Given  a tiling of $B$ in to stable $k\times k$ tiles we construct an environment on $B_1$ by independently sampling environments $\omega_1,\omega_2,\omega_3,\omega_4$ on $B$ with the same law as $\omega$. Using the latter, we now tile $B_1$ using tiles of size $2k\times 2k$ where each such tile is formed from  $4$ tiles of size $k \times k$ (one from each $\omega_i$)  as illustrated in Figure \ref{fig2}. Let us call the constructed environment $\omega'.$
Given such a construction, we would be done once we establish the following two properties:

\begin{enumerate}
\item The constructed event has probability comparable to $p^4$ which follows quite easily since we picked four independent copies of environments in $\sU_{\zeta}(n)$  to obtain the environment on $B_1$.
\item To show that any path $\gamma'$ in $\omega'$ between $\bo{0}$ and $\bo{2n}$ has length at least $(\mu+\zeta-\e)2n$ for some small $\e.$
\end{enumerate}
To show the latter we decompose $\gamma'$ into excursions $\gamma'_1,\gamma'_2,\ldots$ where each $\gamma'_i$ resides in a tile of size $2k\times 2k$ and $\gamma'_i$ and $\gamma'_{i+1}$ reside in separate tiles. 
Thus $|\gamma'|=\sum_{i}|\gamma'_i|.$
Now the key is to observe that for such a $\gamma'$ one can create a path $\gamma$ in the environment $\omega$ between $\bo{0}$ and $\bo{n}$ such that $\gamma$ is a concatenation of paths $\gamma_1,\gamma_2,\ldots$. This is done by just taking $\gamma_i$ to be the shortest path between points which are the endpoints of $\gamma'_i$ scaled down by a factor $2.$ (see Figure \ref{fig2}).  However note that  $\gamma_{i}$ need not be excursions even though $\gamma'_i$'s  are by definition. The former are just taken to be the shortest path in the environment $\omega$ between the points obtained by dividing the end points of $\gamma'_i$ by $2.$

The stability of the tiles now imply that $|\gamma_{i}|=(1+o(1))\frac{|\gamma'_i|}{2}$. Thus it follows that 
\begin{equation}\label{dila453}
|\gamma'|=\sum_{i}|\gamma'_i|=2(1+o(1))\sum_i|\gamma_i|=2(1+o(1))|\gamma|\ge 2(1+o(1))(\mu+\zeta)n
\end{equation}
where the last inequality follows by definition as $\gamma$ is a path between $\bo{0}$ and $\bo{n}$ in the environment $\omega$ which is in $\sU_{\zeta}(n).$

There are a few obstacles in making this outline rigorous and some work is needed to circumvent them as we briefly outline below.

\noindent 
(1)
The most important step is to prove that $B$ can be divided into such stable tiles. In fact we prove that there exists a tiling of $B$ where most tiles are stable, i.e., the total number of points in unstable tiles is $o(n^2)$. This essentially is a property of a general metric structure on $B$ which is bi-Lipschitz with respect to the Euclidean metric (the fact that the FPP metric has this property is a consequence of the shape theorem in \eqref{e:shape}. We record this observation in Lemma \ref{lb90}).  The formal stability result is Proposition \ref{t:stable} in this paper and the proof is provided in Section \ref{s:sproof} where a detailed outline of the proof and an elaborate explanation of  the key ideas can be found. 

 Intuitively the result says that any sub-sequential limiting metric structure due to its bi-Lipschitz nature should have a reasonably smooth gradient function.  Thus the size of the tiles capture the scale at which an approximate smoothness is witnessed. However formally we show (see Proposition \ref{t:stable}) that all but at most a small fraction of tiles are stable and the unstable tiles can be handled  by replacing all the edge values in those by values close to $b$ (recall that $\nu$ is supported on $[0,b]$). This operation only can increase the passage time and hence makes the upper tail event more likely and on the other hand it only costs $e^{-o(n^2)}$ in probability and hence does not change any of the conclusions.  

\noindent
(2)
Finally we describe briefly another point which we have swept under the carpet so far. All the discussion above describes how to construct a $2n\times 2n$ environment out of an $n\times n$ environment preserving (upto an error) the upper tail large deviation event. However observe that in order to prove Proposition \ref{p:pathconstruct}, we need to be able to dilate the original environment by factor $h=\frac{m}{n}$ which could be arbitrarily large. To ensure that the  error  term $(1+o(1))$ in \eqref{dila453} does not blow up we  will in fact modify the notion of stable tiles which allows dilation by an arbitrary factor $h$. To ensure this we prove that stable tiles have a couple of additional properties:
\begin{itemize}
\item  First of all we need to ensure stability at most locations at many consecutive length scales rather than just two as in \eqref{stab908}.

\item More importantly, we show that as the direction vector is varied at a given location, the gradient field has approximate  convexity properties. This result should be thought of as a weak analogue of the convexity of the limiting shape in \eqref{e:shape} in the upper tail large deviation regime and this will enable us to compare the distance function between the $k\times k$ box and the $kh\times kh$ box.
The formal convexity statement is stated as Proposition \ref{conc1} and the proof is presented in Section \ref{pconc1}.
\end{itemize}

\subsection{Organization}
We finish off this introduction by describing the organization of the remainder of this paper. In Sections \ref{s:grad} and \ref{prelim} we set up the notation and make a precise statement of the stabilization result Proposition \ref{t:stable}. We also make precise definition and statement of the regularity results of the gradient field. The proofs of these results are postponed until later. In Section \ref{s:construct} we use these results to prove Proposition \ref{p:pathconstruct}.  In Sections \ref{s:cont} and \ref{pconc1} we provide the proofs of the continuity of rate function (Proposition \ref{p:cont}) and approximate convexity of the distance function (Proposition \ref{conc1}) respectively. Finally in Section \ref{s:sproof} we  prove the stability result Proposition \ref{t:stable} to complete the argument.
For easy reference, below we summarize some of the notations and the parameters (already defined or to be defined later), that will be used frequently throughout the article.
\begin{longtable}{@{}lll@{}}
  {\bf Notation}
  & {\bf \quad\quad Defined in}
  & {\quad\quad\bf Short Informal Description}\\
\hline  
  $\sU_{\zeta}(n)\, (\sL_{\zeta}(n))$
  & \quad\quad\quad See \eqref{utdef} 
  &\quad  \quad Upper-tail (lower tail) events \\ \quad\quad\quad\quad\quad & \quad\quad\quad\quad\quad 
  & \quad\quad at scale $n$.\\
  $\Pi=(X_e: e\in E(\Z^2))$
  & \quad\quad\quad Section \ref{s:grad}\quad\quad 
  &\quad  \quad Typical noise space.\\
  $ \Pi^{\sU}=(X^{\sU}_{e}: e\in E(\Z^2))$
  & \quad\quad\quad Section \ref{s:grad} 
  &\quad  \quad Noise space conditioned on \\ \quad\quad\quad\quad\quad & \quad\quad\quad\quad\quad 
  & \quad\quad $\sU_{\zeta}(n)$.\\

    $\sC$
  & \quad\quad\quad Lemma \ref{apriori23} \quad\quad 
  &\quad  \quad We will restrict ourselves to a box \\ \quad\quad\quad\quad\quad & \quad\quad\quad\quad\quad 
  & \quad\quad size $\sC n$ outside which the geodesic  \\ 
   \quad\quad\quad\quad\quad & \quad\quad\quad\quad\quad 
  & \quad\quad to $\bo{n}$ does not escape w.h.p..\\
  $\sU^*_{\zeta}(n)\, $
  & \quad\quad\quad See \eqref{e:comparison} \quad\quad 
  &\quad  \quad Upper-tail event restricted \\ \quad\quad\quad\quad\quad & \quad\quad\quad\quad\quad 
  & \quad\quad  inside box of size $4\sC n$.\\
 %  $\Geo_{\sU}(\cdot,\cdot)$  & \quad\quad\quad Section \ref{s:grad} \quad\quad 
 % &\quad  \quad Geodesics in $\Pi^{\sU}$.\\ 

 $\Bb(r)$ ($\Lb(r)$)
  & \quad\quad\quad  Section \ref{s:grad}  \quad\quad 
  &\quad  \quad The continuous box (lattice box)\\ \quad\quad\quad\quad\quad & \quad\quad\quad\quad\quad 
  & \quad\quad  $[-r,r]^2$ ($\llbracket -r,r\rrbracket^2$).\\

 $\bS^1(\eta)$
  & \quad\quad\quad See \eqref{discirc}\quad\quad 
  &\quad  \quad Discretized unit circle:\\ \quad\quad\quad\quad\quad & \quad\quad\quad\quad\quad 
  & \quad\quad $[0,\eta,2\eta,\ldots, 2\pi]$.\\

 $\sE$
  & \quad\quad\quad Lemma \ref{lb90} \quad\quad 
  &\quad  \quad Lower bound on passage times.\\

$\sS(\bz,\theta,\ell,k)$ & \quad\quad\quad See \eqref{disseg}\quad\quad 
  &\quad  \quad Discrete segment.\\ 
  $\PT(\cdot,\cdot)$ & \quad\quad\quad Section \ref{mods} \quad\quad 
  &\quad  \quad Passage time.\\ 
  $\bz \text{ is }(\delta,\theta,\ell,k)-\SST$ & \quad\quad\quad See \eqref{stabnot32}\quad\quad 
  &\quad  \quad Distance function grows linearly.\\
  \quad\quad\quad\quad\quad & \quad\quad\quad\quad\quad 
  & \quad\quad in direction $\theta$. \\

  $\Ub_n(j,v)$ & \quad\quad\quad Definition \ref{deftile543} \quad\quad 
  &\quad  \quad Tile of size $\frac{n}{2^j}$ of $\Bb(n)$ in  \\ \quad\quad\quad\quad\quad & \quad\quad\quad\quad\quad 
  & \quad\quad corresponding to $v\in\brj$.\\

  $\grad(\bz,\theta,\ell)$ & \quad\quad\quad See \eqref{grad87}\quad\quad 
  &\quad  \quad Gradient function. \\

$\Gr_n(j)$ & \quad\quad\quad Section \ref{prelim} \quad\quad 
  &\quad  \quad Points with spacing $\frac{n}{2^j}$ in $\Bb(n).$ \\ 

$\Gr_n(\ell;j)$ & \quad\quad\quad Section \ref{prelim}\quad\quad 
  &\quad  \quad Points with spacing $\ell$ on \\ \quad\quad\quad\quad\quad & \quad\quad\quad\quad\quad 
  & \quad\quad the edges of $\Gr_n(j).$  \\

$  \overset{\eta_1,\ell_1,j_1}{\Pro}_n(\bo{z},\bo{w})$ & \quad\quad\quad See \eqref{projectedfunction}\quad\quad 
  &\quad  \quad Projected distance for pairs\\ \quad\quad\quad\quad\quad & \quad\quad\quad\quad\quad 
  & \quad\quad of points in a grid at scale $\eta_1.$  \\   
  
%  $\sP\sV_{\eta_1,\ell_1,j_1}$  & \quad\quad\quad See \eqref{imageset} \quad\quad 
%  &\quad  \quad All possible images for the \\ \quad\quad\quad\quad\quad & \quad\quad\quad\quad\quad 
%  & \quad\quad projected pairwise distance.\\
\hline
\caption{Table of glossaries}
\end{longtable}

\subsection*{Acknowledgements}
Research of RB is partially supported by a Simons Junior Faculty Fellowship and a Ramanujan Fellowship from Govt.\ of India. SG is supported by a Miller Research Fellowship. 

 % \section{Stabilization and Regularity of the Gradient}
    \section{Formal definitions and notations}
\label{s:grad}

Throughout the remainder of this paper we shall fix a passage time distribution $\nu$ that satisfies the hypothesis of Theorem \ref{t:ldp}, i.e., it is supported on $[0,b]$ with a continuous density function. This in particular implies that passage times are not-concentrated on one point and there is no mass at $0$, which in turn implies that the shape theorem \eqref{e:shape} holds. For this passage time distribution and a direction vector $\bo{v}\in \bS^1$, we shall denote by $\mu_{\bo{v}}$ the time constant in direction $\bo{v}$ (as introduced in the previous section, for $\bo{v}=(1,0)$ we shall drop the subscript). Under these condition one can prove the following basic concentration estimate (see e.g.\ \cite{Kes86}) for each $\e>0$, $\bo{v}\in \bS^1$, some $c>0$ and all $n$ sufficiently large we have ($\lfloor n\bo{v} \rfloor$ is the vertex in $\Z^2$ obtained by taking co-ordinate wise integer parts of $n\bo{v}$):
\begin{equation}
\label{e:concentration1}
\P(|\PT({\mathbf{0}, \lfloor n\bo{v} \rfloor})-\mu_{\bo{v}} n|\geq \e n)\leq e^{-cn}.
\end{equation}
We shall use \eqref{e:concentration1} many times and often implicitly without referring to it. Notice that we are concerned with the large deviation regime whereas \eqref{e:concentration1} is for typical environments. To use it in the large deviation regime we need a tool to compare the environment in the large deviation regime with the typical environment. This is provided by the FKG inequality. Let $\Pi=(X_e: e\in E(\Z^2))$ and $\Pi^{\sU}=(X^{\sU}_{e}: e\in E(\Z^2))$ be the typical and conditional (on $\sU_{\zeta}(n)$) edge weight environments respectively. Let $\Geo_{\sU}(\cdot,\cdot)$ denote the geodesics in the environment $\Pi^{\sU}$. The following lemma is a well known consequence of the FKG inequality (see for e.g. Strassen's Theorem).  

\begin{lem}
\label{FKG}
There exists a coupling $(\Pi,\Pi^{\sU})$ such that almost surely, for each edge $e$ we have $X_{e}\le X^{\sU}_e.$
\end{lem}

There are two main consequences of Lemma \ref{FKG} that will be useful for us. First, this will provide lower bounds on the FPP metric conditional on $\sU_{\zeta}$, and second, it will enable us to restrict our attention to finite boxes. Before we proceed with the relevant statements, we extend the function $\PT$ from $\Z^2\times \Z^2$ to $\R^2\times \R^2$; this will reduce notational complexities significantly. There is not one canonical way to this, we choose the following extension for concreteness. For every $x,y\in \R^2$ define $\PT(x,y):=\PT(\hat x, \hat y)$ where $\hat x$ and $\hat y$ are the nearest lattice points to $x,y$ respectively, (in case of a tie, we choose the one which is smallest in the usual lexicographic order on $\Z^2.$). We introduce some more useful notations. Throughout we will use $\Bb(r)$ (resp.\ $\Lb(r)$) to denote the box  $[-r,r]^2\subseteq \R^2$ (resp.\ the box $\llbracket -r,r \rrbracket^2 \subseteq \Z^2$).

The next lemma shows that, geodesics do not wander too much even in the large deviation regime. Let $\mu_{\min}=\min_{v\in \bS^1} \mu_v$. It a consequence of \eqref{e:shape} that $\mu_{\min}>0$. Let us fix $\sC= \frac{4b}{\mu_{\min}}$. This $\sC$ will be important for us and will be fixed throughout the paper. 

\begin{lem}
\label{apriori23} 
For all $\zeta \in (0, b-\mu)$There exists $c>0$ such that for all $n$ sufficiently large we have,  
$$\P(\Geo_{\sU}(\bo{0}, \bo{n})\subset \Bb(\sC n))\geq 1-e^{-cn}.$$ 
\end{lem}

\begin{proof} 
Observe that \eqref{e:concentration1} together with an union bound over the lattice points on the boundary of $\Bb(\sC n)$ implies that  $\PT(\mathbf{0}, \Z^2\setminus \Bb(\sC n))\ge 2 bn$ with exponentially small failure probability in the typical environment. Lemma \ref{FKG} implies that the same is true for the environment $\Pi^{\sU}$.
\end{proof}

The next lemma shows that the in the environment $\Pi^{\sU}$, with high probability the FPP metric within $\Bb(\sC n)$ is lower bounded by a constant multiple of the Euclidean metric. 

\begin{lem}
\label{lb90}
There exists $\alpha, c>0$ such that for all sufficiently large $n$, with conditional ( on $\sU_{\zeta}$) probability at least $1-e^{-c\sqrt{n}}$ the following holds: for any two points $\bz, \bw \in \Bb(4\sC n)$ with $|\bz-\bw|\ge \sqrt{n},$ the passage time between $\bz$ and $\bw$ restricted inside $\Bb(4\sC n)$ is at least  $\alpha |z-w|.$ 
\end{lem}

This high probability event described above will be useful for us, and we denoted by $\sE$ for future reference.

\begin{proof}
The lemma follows by taking a union bound over all pairs of points in $\Lb(4 \sC n)$ with mutual distance at least $\sqrt{n}-3$, and using Lemma \ref{FKG} together with \eqref{e:concentration1} (take $\alpha= \frac{1}{2}\mu_{\min}$ for example). 
\end{proof}

Thus from the above lemmas we can restrict ourselves to  $\Bb(4\sC n)$ by defining the event $\sU^*_{\zeta}=\sU_{\zeta}\cap \sE$ where we take $\alpha=\frac{1}{2}\mu_{\min}.$ 
Notice that $\sU^*_{\zeta}$ is just a function of the edges in $\Bb(4\sC n).$
Now by Lemma \ref{lb90}
\begin{equation}
\label{e:comparison}
(1-e^{-c\sqrt n})\P(\sU^*_{\zeta})\leq \P(\sU_{\zeta}) 
\end{equation}
This allows us to work with $\sU^*_{\zeta}$ instead of $\sU_\zeta$ and this is what we will do throughout the article. 

\subsection{Gradients and  stability}\label{gradstabsec}
To precisely state the stabilization that we have alluded to, we need to develop some more notation. For our purposes, we shall be comparing distance functions for fixed directions, so we introduce the following notation. For $\bz \in \R^2,$ and $\theta \in \mathbb{S}^1$ (the unit circle), let $\L_{\theta, \bz}=\{\bz+\lambda\theta: \lambda>0\}$, i.e., in the standard parametrization of $\mathbb{S}^1$, $\L_{\theta, \bz}$ denote the the ray starting from $\bz$ in the direction $\theta$. 
 We shall consider a sequence of equally spaced points along $\L_{\theta, \bz}$ defined as follows. For $\bz\in \R^2, \theta\in \mathbb{S}^1, k\in \N$ and $\ell>0$, let us define the discrete segment 
\begin{equation}\label{disseg}
\sS(\bz,\theta,\ell,k)=[\bz_0,\bz_1,\ldots,\bz_k]
\end{equation}
where $\bz_0=\bz$ and $\bz_{i+1}=\bz_i+\ell \theta$, see Figure \ref{fig3}.
\begin{figure}[h]
\centering
\includegraphics[scale=.5]{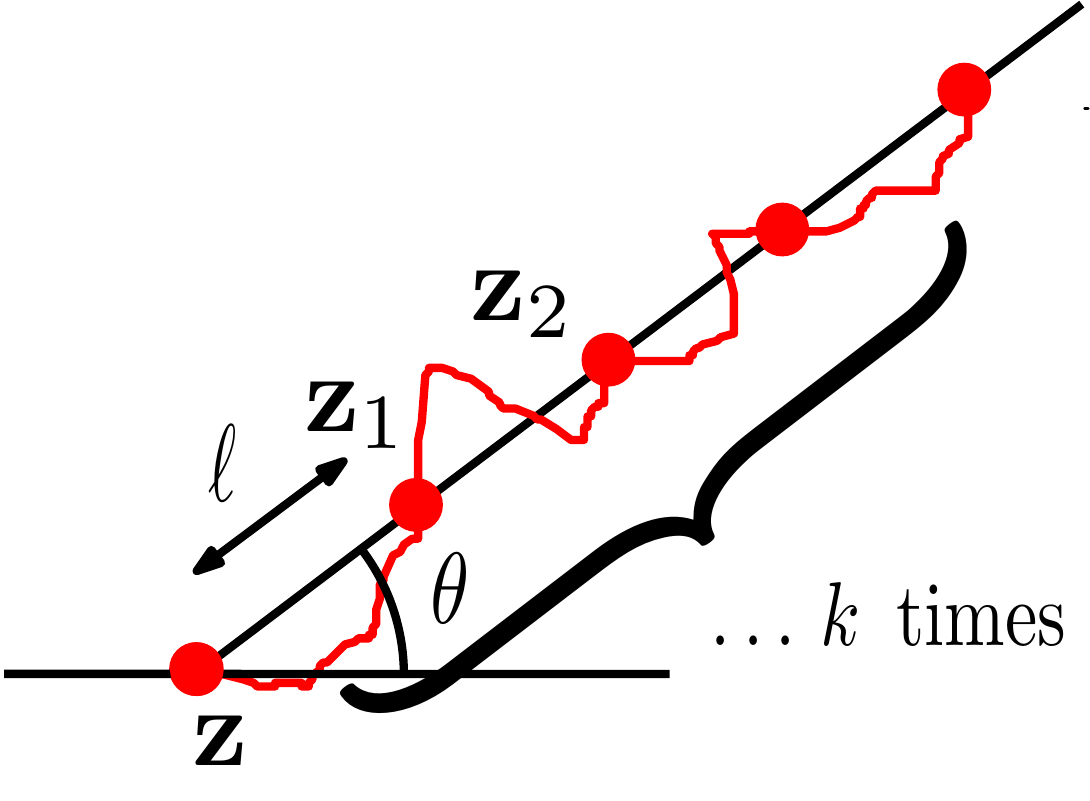}
\caption{$k$ points spaced at distance $\ell$ along a line making angle $\theta$ with the $x-$axis forming $\sS(\bz,\theta,\ell,k)$.}
\label{fig3}
\end{figure}

We define the passage time for the segment $\sS$ by 
\begin{equation}\label{pt12}
\PT(\bz,\theta,\ell,k):=\sum_{i=0}^{k-1} \PT(\bz_i,\bz_{i+1}).
\end{equation}

Now note that the starting point and ending points of $\sS({\bz,\theta, \frac{\ell}{2}, 2k})$  and $\sS({\bz,\theta, \ell, k})$ are the same  and the former is obtained from the latter by subdividing subintervals of the latter in to equal halves.

As an easy consequence of the triangle inequality we have the following straightforward lemma. 
\begin{lem}\label{monotone}
$\PT(\bz,\theta,\frac{\ell}{2}, 2k)\ge \PT(\bz,\theta,\ell,k).$  
\end{lem}
The main arguments in this paper rely on a notion of stability of the passage time at a point $\bz.$
 Fix a tolerance parameter $\delta>0.$  For $k\in \N$, $\ell>0$ and $\theta\in \mathbb{S}^1$, we say that $\bz\in \R^2$ is $(\delta, \theta, \ell, k)-\SST$ (with respect to any edge weight configuration $\Pi$) if  for $1\le k'\le k,$
\begin{align}\label{stabnot32}
\frac{k'\PT(\bz,\theta,\ell,1)}{(1+\delta)}\le \PT(\bz,\theta,\ell k', 1)\le(1+\delta)k' \PT(\bz,\theta,\ell,1).
\end{align}

In words, $\bz\in \R^2$ is $(\delta, \theta, \ell, k)-\SST$ if the passage time from $z$  to $z+(\theta, \ell k')$ can be approximated up to a $(1+\delta)$ multiplicative error by $k'$ times the passage time from $z$  to $z+(\theta, \ell)$ for all $1\le k'\le k$. This captures the linear growth of the distance function.

In the following for convenience we would work with a discretized version of $\bS^1$. For any $\eta>0,$ let 
\begin{equation}\label{discirc}
\bS^1(\eta)=\{0,\eta,2\eta,\ldots 2\pi-\eta\}.
\end{equation}
 ($\eta$ is assumed to have the required properties to avoid rounding issues. Also throughout the article, we will use $\theta$ interchangeably to denote an angle or a unit vector making the corresponding angle with the $x-$axis. The usage will be clear from context.)
In the sequel we will say that $\bz$ is $(\delta, \bS^1(\eta),\ell,k)-\SST$  if $\bz$ is $(\delta,\theta,\ell,k)-\SST$ for each $\theta \in \bS^1(\eta)$ and similarly we will say that $\bz$ is $(\delta,\ell,k)-\SST$ if $\bz$ is $(\delta,\theta,\ell,k)-\SST$ for each $\theta \in \bS^1$.

With this preparation, we can now state  an initial version of our stabilization result. 

\begin{ppn}
\label{t:stable}
Fix $\delta,\e,\eta >0,$ and $k\in \N$ and $J_1\in \N$. There exists $J_2\in \N$  such that for all large enough $n$, conditioned on $\sU^*_{\zeta}(n)$ the following holds: there exists $J_1\leq j\leq J_2$  (random depending on $\Pi \in \sU^*_{\zeta}(n)$ ) such that 
$$\#\{\bz\in \Lb(\sC n): \bz~\text{is not}~(\delta, \bS^1(\eta),\frac{\sC n}{2^j},k)-\SST\} \leq \e n^2 .$$
\end{ppn}

Proof of Proposition \ref{t:stable} is rather technical and is postponed until Section \ref{s:sproof}. This is one of the three main ingredients of our proof, and we state this result in terms of $\Lb(\sC n)$ so that it can directly be fed into many of the later arguments. However, for the next few definitions and results it will be notationally convenient to work with boxes of size $n$.

We next define the gradient function for $\SST$ points naturally in the following way:  For $\theta \in \bS^1,$ and $\ell \in \N,$ let
\begin{equation}\label{grad87}
\grad(\bz,\theta,\ell)=\frac{\PT(\bz, \bz+ (\theta,\ell))}{\ell}.
\end{equation}
An easy consequence of the notion of stability is that the gradient function stays almost constant over a range of values of $\ell.$

\begin{lem}
\label{eascons} 
Fix $j\in \N$. On the event $\sE$ (see Lemma \ref{lb90}), for all sufficiently large $n$, for any $\ell\ge \frac{n}{2^j}$, and for any $(\delta,\bS^1(\eta), \ell,k)-\SST$  point $\bz,$ for any $\frac{k}{4}\ell \le \ell', \ell''\le k\ell,$ and for any $\theta \in \bS^1$
$$\grad(\bz,\theta,\ell')=\left(1+O\bigl(\eta+\delta +\frac{1}{k}\bigr)\right)\grad(\bz,\theta,\ell'').$$
\end{lem}

\begin{proof}
The above lemma without the $O(\eta)$ term in the multiplicative factor follows immediately from definition of  stability for all $\theta$ in $\bS^1(\eta)$. However we need to extend this to all $\theta\in \bS^{1}$, and a further approximation is necessary. For any $\theta\in \bS^1$ let $\hat \theta$ be the closest point in $\bS^1(\eta).$ Then by triangle inequality for any $\ell$, it follows that 
\begin{align}\label{tria1}
\PT(\bz,\bz+(\theta,\ell))&\le \PT(\bz,\bz+(\hat \theta,\ell))+b\eta \ell,\\
\nonumber
\PT(\bz,\bz+(\hat \theta,\ell))&\le \PT(\bz,\bz+( \theta,\ell))+b\eta \ell;
\end{align}
since the edge variables are bounded by $b$. This completes the proof of the lemma with the addition of the $O(\eta)$ term in the multiplicative error.
\end{proof}

Note that Proposition \ref{t:stable} claims that most points in $\Bb(\sC n)$ are stable. 
We now prove a  stable point implies  stability in a neighbourhood with slightly worse parameters.

\begin{lem}
\label{stab23} 
For $k>C>m>0$, $\ell>0$ and for any $\bz$ which is $(\delta,\theta,\ell,k)-\SST$  and any $\bz'$ 
such that $|\bz-\bz'|\le \ell m$ we have $\bz'$ is
$(\delta',\theta,C \ell, \frac{k}{C})-\SST,$  where $\delta'= \delta +O(\frac{m}{C}).$
\end{lem}

\begin{proof} The proof follows by another application of  triangle inequality where we observe the following, analogous to \eqref{tria1}: For any $\ell',$
\begin{align}\label{tria2}
\PT(\bz,\bz+(\theta,\ell'))&\le \PT(\bz',\bz'+(\theta,\ell'))+b \ell m,\\
\nonumber
\PT(\bz',\bz'+(\theta,\ell'))&\le \PT(\bz,\bz+(\theta,\ell'))+b \ell m.
\end{align}
Hence for any $\ell'\ge C\ell$  it follows that,
$\PT(\bz,\bz+(\theta,\ell'))=(1+O(\frac{m}{C}))\PT(\bz',\bz'+(\theta,\ell')),$ (see Figure \ref{fig4} for an illustration).
\end{proof}

An immediate but important corollary is the following smoothness of the gradient field which we state without proof.
\begin{cor}
\label{smoothgrad} Given $\delta$ and $\delta'$ as in Lemma \ref{stab23}, for all large $k$ and all large $n$ and any $\bz, \bz'$ satisfying the hypothesis of that lemma, and for all $\theta \in \bS^{1},$
$$
\frac{1}{1+\delta'}\le \frac{\grad(\bz,\theta,\ell)}{\grad(\bz',\theta, C\ell)}< 1+\delta'. 
$$
\end{cor}
\begin{figure}[h]
\centering
\includegraphics[scale=.4]{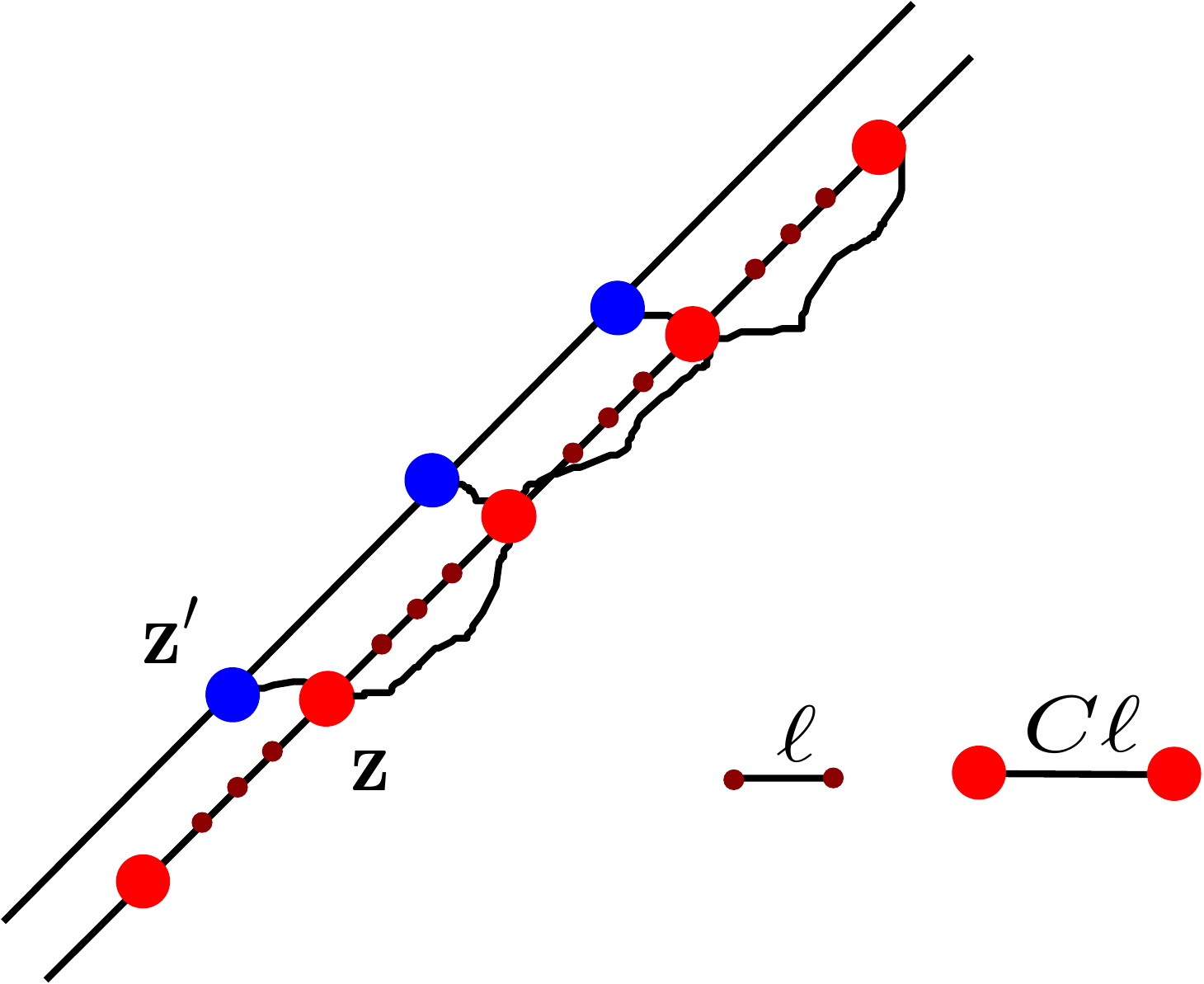}
\caption{Stability for the discrete segment formed by the red points implies the stability for the nearby segment formed by the blue points.}
\label{fig4}
\end{figure}
\subsection{Stability of Tiles} In this subsection we introduce the notion of stability of tiles parallel to the notion of stability for points, which will be convenient for the proofs. The section contains a few lemmas which even though  quite similar to the ones already stated, have various associated quantifiers which could make it a little hard to read and the reader can choose  to skip the straightforward proofs in this section. This will not affect readability of the future sections. 

Given a lattice box $\Lb(n)$ we will often think of it as made up of boxes of a particular scale $j$, i.e. think of the box as being naturally tiled using boxes of size $n/2^j$. Note that one can define a natural bijection between the set of tiles and the set  $\llbracket 1, 2^j\rrbracket^2$. We will use this bijection to denote the tile corresponding to $v\in \llbracket 1, 2^j\rrbracket^2$ by $\Ub_n(j,v)$ (see Figure \ref{fig5}).

\begin{defn}\label{deftile543}For any $v\in \llbracket 1, 2^j \rrbracket^2$, a tile $\Ub_n(j,v)$ is said to be $(\delta, \bS^1(\eta), \ell, k, \e)-\SST$ if at least $1-\e$ fraction of the lattice points in $\Ub_n(j,v)$ are $(\delta, \bS^{1}(\eta),\ell,k)-\SST.$  
\end{defn}
\textbf{In the sequel we will choose $\ell=\frac{n}{2^{j+m}}$ and $k=2^{2m}$ for some $j,m\ll \log_2(n)$, where the choice of $j$ and $m$ will vary through the paper and will depend on some other parameters relevant for specific applications.}

\begin{figure}[h]
\centering
\includegraphics[scale=.6]{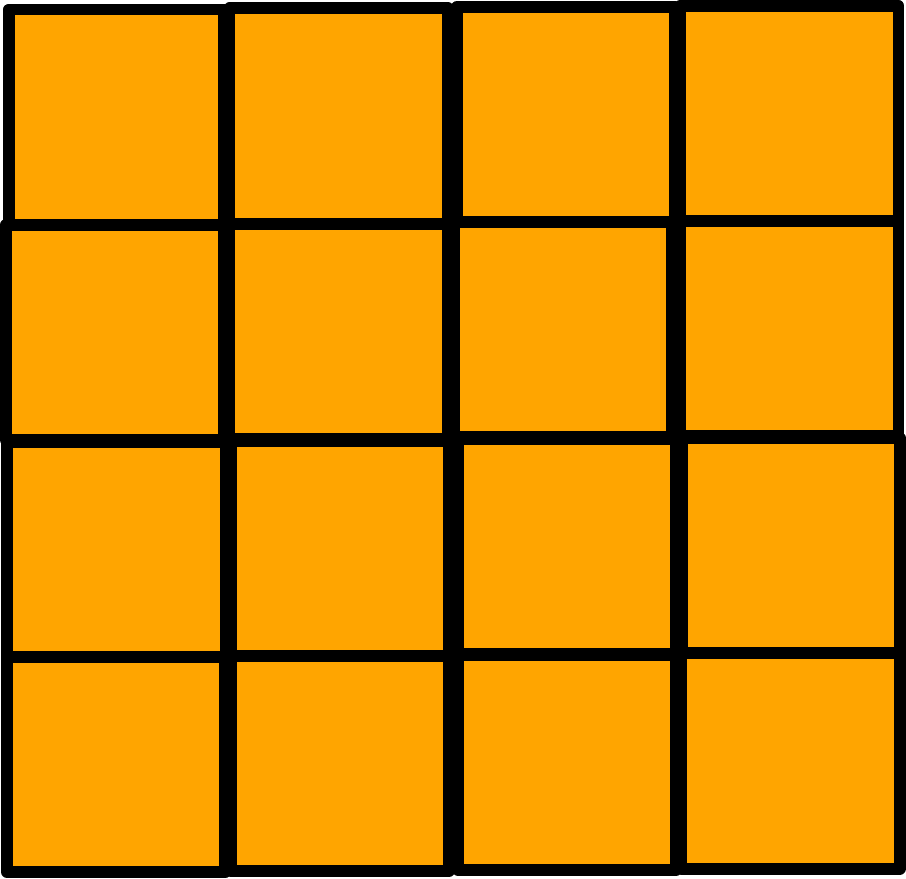}
\caption{The first figure illustrates the tiling an $n\times n$ box in to tiles of size $\frac{n}{4}.$ Thus the set of tiles has a natural bijection with  $\llbracket 1, 4\rrbracket^2$.}
\label{fig5}
\end{figure}

Using Lemma \ref{stab23}, we now prove that if at least $(1-\e)$ fraction of the lattice points in a $\Ub_{n}(j,v)$ are  stable for some values of the parameters, then all the points are  stable for a slightly different range of parameters. 

\begin{lem}
\label{ssttile} 
Let $j,m \in \N$, and $\ell=\frac{n}{2^{j+m}}, k=2^{2m}$.    Fix $\delta>0$. There exists $C>0$ sufficiently large such that for all sufficiently small $\e>0$, on $\sE,$ the following holds for all sufficiently large $n$: if $\Ub_n(j,v)$ is $(\delta,\bS^{1}(\eta),\ell,k, \e)-\SST$, then $\Ub_n(j,v)$ is $(2\delta, \bS^{1}(\eta), \ell',k', 0)-\SST$ where $\ell'=\max(\frac{n}{2^{j}}C\sqrt{\e},\ell)$ and $k'=k\ell/\ell'$.
\end{lem}

\begin{proof} Observe that for every $\Ub_n(j,v)$ that is $(\delta,\bS^{1}(\eta), \ell,k, \e)-\SST$ and any $\bz \in \Ub_n(j,v)$ there exists $\bw \in \Ub_n(j,v)$ with $|\bz-\bw| \le 8\sqrt \e\frac{n}{2^j}$ and $\bw$ is $(\delta,\bS^1(\eta), \ell,k)-\SST.$ This is because the existence of  a $\bz$ for which there is no such $\bw$ contradicts the hypothesis that  $\Ub_n(j,v)$ is $(\delta, \bS^{1}(\eta), \ell,k, \e)-\SST$. The proof now follows from Lemma \ref{stab23} by for $C$ sufficiently large (and $\e$ sufficiently small). 
\end{proof}

From now on we will call a $(\delta,\bS^{1}(\eta),\ell',k', 0)-\SST$ tile as a $(\delta,\bS^{1}(\eta),\ell',k')-\SST$ tile. 
We now show that the above in fact implies  stability for all angles $\theta \in \bS^1.$
\begin{lem}
\label{l:gradtile} Let $j, m$  be as in the previous lemma. Then on $\sE,$ the following holds for all sufficiently large $n$: for a $(\delta,\bS^1(\eta), \ell,k)-\SST$   $\Ub_n(j,v)$  
for $\bz, \bz' \in \Ub_n(j,v)$ we have for all $\theta \in \bS^1,$
\begin{align}
\label{stabgrad12}
\frac{1}{1+\delta'}&\le \frac{\grad(\bz,\theta,k_1\ell)}{\grad(\bz',\theta, k_2 \ell)}< 1+\delta',
\end{align}
with $\delta'= O(\delta+\eta+ \frac{1}{2^m})$ and $1\le k_1 ,k_2 \le k.$ 
\end{lem}

\begin{proof}
The proof is quite similar to that of  Lemma \ref{ssttile}. Recalling  \eqref{tria1} if for any  $\theta \in \bS^1$,  $\hat \theta$ is the closest point in $\bS^1(\eta),$ then for any $k_1\le k,$
\begin{align*}
|\PT(\bz,\bz+(\theta, k_1\ell))- \PT(\bz,\bz+(\hat \theta, k_1\ell))|\le b\eta k_1 \ell,\\
\end{align*}
which along with the hypothesis that $\bz$ is $(\delta,\bS^1(\eta), \ell,k)-\SST$ implies that 
\begin{align*}
\frac{1}{1+O(\delta+\eta)}&\le \frac{\grad(\bz,\theta,k_1\ell)}{\grad(\bz,\theta, k_2 \ell)}< 1+O(\delta+\eta)).
\end{align*}
Now another application of triangle inequality as in \eqref{tria2}, shows that for any $\bz,\bz' $ as in the statement of the lemma,
\begin{align*}
\PT(\bz,\bz+(\theta, k \ell))&\le \PT(\bz',\bz'+(\theta,k \ell))+O(b\frac{n}{2^j}).
\end{align*}
Hence using the fact that $k\ell=\frac{2^m n}{2^j},$ it follows that,
$
\frac{1}{1+O(\delta+\eta+\frac{1}{2^m})}\le \frac{\grad(\bz,\theta,k\ell)}{\grad(\bz',\theta, k \ell)}< 1+O(\delta+\eta+\frac{1}{2^m}). 
$

\end{proof}

Thus from now on, we shall refer to a tile as $(\delta,\ell,k)-\SST$ if \eqref{stabgrad12}  is satisfied with $\delta$ in place of $\delta'$.
Now for a $(\delta,\ell,k)-\SST$ $\Ub_n(j,v)$ as above, \eqref{stabgrad12} allows us to define a gradient function not for every individual point $\bz$ but for the whole tile itself. 
\begin{defn}\label{tilegrad}
For a  $(\delta,\ell,k)-\SST$ $\Ub_n(j,v)$ define for any $\theta \in \bS^1,$
 $$ \grad_n((j,v),\theta)= \grad_n((j,v),\theta, \ell):=\grad(\bz,\theta, \ell)$$ for the center point $\bz$ of $\Ub_n(j,v).$  
\end{defn}

Observe that even though this definition implicitly depends on $\ell$, we shall drop it from our notation as the length scale $\ell$ will always be clear from the context. The reason for calling the the gradient function for $\Ub_n(j,v)$ is the following: even if we replace the centre of $\Ub_n(j,v)$ by any arbitrary $\bz\in \Ub_n(j,v)$, the value of the gradient changes only by a multiplicative factor of $(1+\delta)$; in all our applications, by proper choice of parameters $\delta$ will be made arbitrarily close to zero.  

With the above preparation we shall now go back to the setting of Proposition \ref{t:stable} and show that there exists a scale $j$ such that, conditional on $\sU^*_{\zeta}(n)$, with probability bounded below most of the scale $j$ tiles in $\Bb(\sC n)$ are  stable.

\begin{lem}\label{goodscale1} Conditional on $\sU^*_{\zeta}(n),$ (recall that this was an event on $\Lb(4\sC n)$). Then given $\eta, m,\delta,\e_1, J_1$ such that $\frac{1}{2^m}\ge \sqrt{\e_1},$ there exists a constant  $J_2$  such that for all large enough $n,$ there exists a  scale $J_1\le j\le J_2 $ (depending on $n$) such that with probability at least $\frac{1}{J_2}$, for all but $\e_1$ fraction of $v\in \llbracket 1, 2^{j}\rrbracket^2,$  $\Ub_{\sC n}(j,v)$ is $(\delta,\bS^1(\eta),\ell,k,\e_1)-\SST$ (see Definition \ref{deftile543}) where $\ell=\frac{\sC n}{2^{j+m}}$ and $k=2^{2m}.$
\end{lem}
\begin{proof} Note that from the statement of Proposition \ref{t:stable} choosing $k=2^{4m}$ and $\e=\e_1^2$ it follows that there exists  a scale $j$ such that with probability at least $\frac{1}{J_2}$ ($J_2$ appearing in the statement of Proposition \ref{t:stable}) the fraction of points $\bz$ in $\Lb(\sC n)$ which are not $(\delta,\bS^{1}(\eta),\ell,k)-\SST$ is at most $\e n^2$ where $\ell=\frac{\sC n}{2^{j+m}}$ and $k=2^{2m}.$ 
Thus the total fraction of $v\in \llbracket 1, 2^{j}\rrbracket^2$ such that  $\Ub_{\sC n}(j,v)$ is not $(\delta, \bS^1(\eta), \ell,k,\e_1)-\SST$ is at most $\e_1$ since other wise the total fraction of points $\bz\in \Lb(\sC n)$ that are not $(\delta, \bS^1(\eta),\ell,k)-\SST$ will be more than $\e=\e_1^2$ contradicting the conclusion of Proposition \ref{t:stable}.

\end{proof}

The above result along with Lemma \ref{ssttile} now implies that most of the tiles are  stable with the parameter $\e$ being set to $0$, and other parameters slightly worsened. 

\begin{lem}\label{most34} Given small enough $\delta_1,\e_1>0$ and a positive integer $m_1,$ such that $\frac{1}{2^{m_1}}\ge \e_1^{1/4},$ and $J_1 \in \N$ there exists  $J_2$ such that for all large enough $n,$ conditioned on  $\sU^*_{\zeta}(n)$  there exists $J_1\le j_1< J_2$ (depending on $n$) such that with probability at least  $\frac{1}{J_2}$ the fraction of $v\in \llbracket 1, 2^{j_1}\rrbracket^2$ such that $\Ub_{{\sC n}}(j_1,v)$ is not  $(\delta_1,\ell_1, k_1)-\SST$ is at most $\e_1$ where $\ell_1=\frac{n}{2^{j_1+m_1}}$ and $k_1=2^{2m_1}$.
\end{lem}

\begin{proof} The proof will follow by first using Lemma \ref{goodscale1} with some choice of parameters $\eta, \delta, m, \e_1, J_1$ which implies the existence of $j_1$ such that with  probability at least $\frac{1}{J_2}$, for all but $\e_1$ fraction of $v\in \llbracket 1, 2^{j_1}\rrbracket^2,$  $\Ub_{\sC n}(j_1,v)$ are $(\delta,\bS^1(\eta),\ell,k,\e_1)-\SST$ (see Definition \ref{deftile543}) where $\ell=\frac{\sC n}{2^{j_1+m}}$ and $k=2^{2m}$ for some values of $\delta$ and $m$.
We will now apply  Lemma \ref{ssttile}  to conclude from the above that all but $\e_1$ fraction of $v\in \llbracket 1, 2^{j_1}\rrbracket^2,$  $\Ub_{\sC n}(j_1,v)$ are 
$(2\delta, \bS^{1}(\eta), \ell',k', 0)-\SST$ where $\ell'=\max(\frac{n}{2^{j_1}}C\sqrt{\e},\ell)$ and $k'=k\ell/\ell'$ for some $C.$
Now applying Lemma \ref{l:gradtile} we conclude that each tile of the latter kind is in fact
$(\delta'',\ell',k')-\SST$ for some $\delta''>0.$  It can now be  verified that our initial choice of parameters can be made such  that $\delta'',\ell',k'$ matches the parameters in Lemma \ref{most34}.
\end{proof}
Throughout the article Lemma \ref{most34} will govern our choices of parameters. 
 
\section{Technical preliminaries}\label{prelim}

As mentioned in our proof strategy, we shall take a configuration from the large deviation regime at some length scale $n$, and replicate/dilate the same configuration to obtain a configuration at a higher length scale. The obvious problem one notices is that for continuous passage time distributions, each configuration has probability $0$. Hence to carry out our proof strategy, we will not be able to work with the edge weight configurations directly. We will project it to a discrete set of $e^{o(n^2)}$ many elements and pick the most likely one among them (still in the large deviation regime). We shall employ the following discretization. 
 
Note that by the upper bound on the support of the edge variables, deterministically, $\PT(\bx,\by)\le 3b|\bo{x}-\bo{y}|+2b$ for any $\bo{x},\bo{y} \in \R^2$. Now for a discretization parameter $\eta_1,$   we will discretize the normalized distances (passage time divided by Euclidean distance) to be in the set 
$\{0, \eta_1,2\eta_1,\ldots, 3b\}$
(again assuming that $\eta_1$ is chosen to avoid rounding issues) and project the distance functions $\PT(\cdot,\cdot)$ onto a discrete space accordingly.

To define things formally, first let the set of all points in $\Bb(n)\cap \frac{n}{2^j}\Z^2$ be called  $\Gr_n(j).$ We will also need the following variant. Let $\ell=\frac{n}{2^{j+m}}$ for some some $m\in \N$. By $\Gr_n(\ell;j),$ we shall denote the set of all  points in $\Gr_n(j+m)$ which intersect the line segments joining the nearest neighbors in $\Gr_n(j)$ thought of as elements of $\frac{n}{2^j}\Z^2$ (see Figure \ref{fig6}). 
\begin{figure}[h]
\centering
\includegraphics[scale=.6]{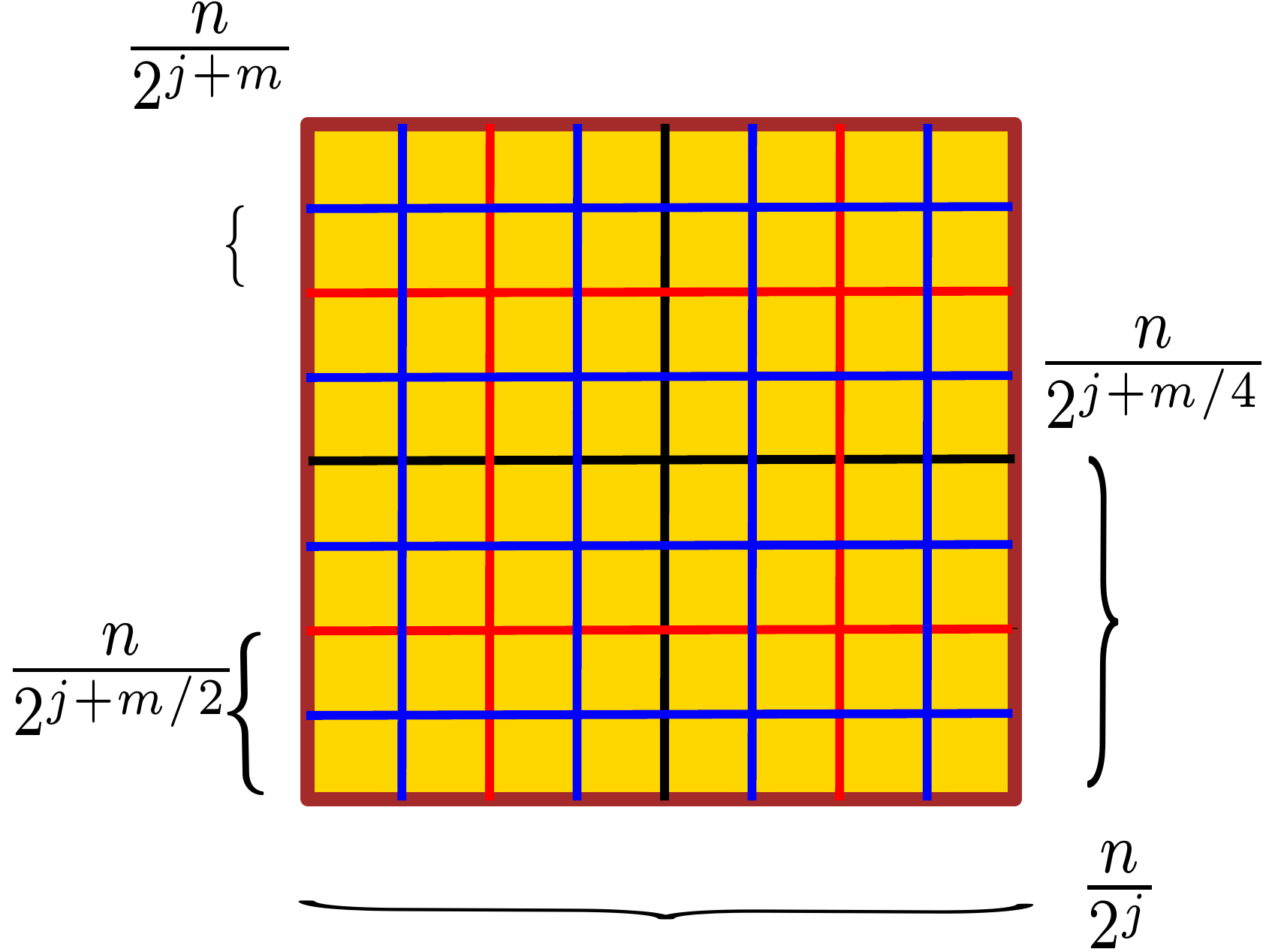}
\caption{Figure illustrating the various grid points. Intersection of the brown lines denote $\Gr_n(j)$, intersection of the black lines with the black lines as well as the brown lines denote the points in $\Gr_n(j+\frac{m}{4})$ which are not in $\Gr_n(j)$, and similarly points on red lines and blue lines denote points in $\Gr_n(j+\frac{m}{2})$, and $\Gr_n(j+m)$ respectively which are not in the previous coarser grid. 
}
\label{fig6}
\end{figure}

Now given $\eta_1,\ell_1,j_1$ with $\ell_1=\frac{\sC n}{2^{j_1+m_1}}$ let the projection map
 $\overset{\eta_1,\ell_1,j_1}{\Pro}: \Gr_{\sC n}(j_1+m_1)\times \Gr_{\sC n}(j_1+m_1) \to \R_{+}$  be defined as follows: for any $\bo{z},\bo{w} \in\Gr_{\sC n}(j_1+m_1),$
\begin{equation}\label{projectedfunction}
\overset{\eta_1,\ell_1,j_1}{\Pro}(\bo{z},\bo{w})=\eta_1\left\lfloor\frac{\PT(\bo{z},\bo{w})}{\eta_1 |\bo{z}-\bo{w}|}\right\rfloor|\bz-\bo{w}|.
\end{equation}

Observe that the function $\overset{\eta_1,\ell_1,j_1}{\Pro}$\footnote{Although the domain of $\Pro$ is determined completely by $\ell_1$ in practice we shall mostly apply this function on pairs of points in $\Gr_n(\ell_1;j_1)$, hence we chose to keep both parameters $\ell_1$ and $j_1$ while specifying $\Pro$.} is random but we choose to suppress the dependence on the underlying noise for brevity. We will also drop the dependence on $\eta_1,\ell_1,j_1$ in the notation whenever there is no scope of confusion. Observe that a very basic counting argument yields that the cardinality of the image set of  $\overset{\eta_1,\ell_1,j_1}{\Pro}$ denoted $\sP\sV_{\eta_1,\ell_1,j_1}$ satisfies

\begin{equation}\label{imageset}
|\sP\sV_{\eta_1,\ell_1,j_1}|\leq e^{O(2^{2(j_1+m_1)})\log\frac{1}{\eta_1}}=e^{o(n^2)}
\end{equation}
where $m_1$, as above, is defined by $\ell_1=\frac{\sC n}{2^{j_1+m_1}}$. Note that ${\Pro}$ induces a weighted graph with vertex set $\Gr_{\sC n}(j_1+m_1),$  and the weight on any edge $(\bo{z},\bo{w})$ being ${\Pro}(\bo{z},\bo{w})$.  It will also be useful to extend the definition of $\Pro$ to a larger set of pairs. For all pairs of points $\bo{z}, \bo{w} \in \Bb(\sC n)$ we will extend the definition, by letting $\Pro(\bo{z}, \bo{w})={\Pro}(\hat{\bo{z}},\hat{\bo{w}}),$ where $\hat{\bo{z}},\hat{\bo{w}}$ are the nearest points to $\bo{z}, \bo{w}$ respectively in $\Gr_{\sC n}(j_1+m_1),$ (as before breaking ties by picking the smallest in the lexicographic order). Note that if $\bz$ and $\bw$ get rounded to the same point, then $\Pro(\bo{z}, \bo{w})$ is zero which is not a realistic definition. However we will only be interested in pairs $\bz$ and $\bw$ that are reasonably far apart so that the above issue will not arise and hence we will not bother about this aspect of the definition.

The first thing we show now is that the error introduced by using $\Pro(\cdot, \cdot)$ instead of $\PT(\cdot, \cdot)$ can be neglected at sufficiently large length scales. For reasons that will become clear momentarily, we shall work with $\SST$ tiles, although the approximation is valid independent of that. Fix $\delta_1,\e_1$ and $m_1$  as in  Lemma \ref{most34}, which then guarantees that there exists $j_1$ with probability bounded away from zero, such that  for all but $\e_1$ fraction of $v \in \llbracket 1, 2^{j_1}\rrbracket^2$, 
$\Ub_{\sC n}(j_1,v)$ is $(\delta_1,\ell_1,k_1)-\SST$ where $\ell_1$ and $k_1$ are $\frac{\sC n}{2^{j_1+m_1}}$ and $2^{2m_1}$ respectively.
For later reference let us call  $v \in \llbracket1, 2^{j_1}\rrbracket^2$  such that  $\Ub_{\sC n}(j_1,v)$ is not $(\delta_1,\ell_1,k_1)-\SST$ as $(\delta_1,\ell_1, k_1)-\US.$

Fixing a value of $\eta_1$ (to be specified later and $\ll \delta_1$) we will now consider the projection map 
$\overset{\eta_1,\ell_2,j_1}{\Pro}$ where $\ell_2=\frac{\sC n}{2^{j_1+\frac{m_1}{2}}}$. 

\begin{lem}
\label{compare1} 
Fix $\delta_1,\eta_1, \ell_1, \ell_2, j_1$  as above and conditioned on $\sU^*_{\zeta}(n),$ consider $v\in  \llbracket 1, 2^{j_1}\rrbracket^2$ such that 
$\Ub_{\sC n}(j_1,v)$ is $(\delta_1,\ell_1,k_1)-\SST.$  Then for any $\bo{z},\bo{w}\in \Gr_{\sC n}(j_1+m_1/2)$ such that $\bo{z},\bo{w} \in \Ub_{\sC n}(j_1,v)$ we have the following:
$$1\leq \frac{\PT(\bo{z}, \bo{w})}{\Pro(\bo{z}, \bo{w})}\leq 1+O(\eta_1).$$
\end{lem} 

\begin{proof}
Observe that by definition 
$$\Pro(\bo{z}, \bo{w})\leq \PT(\bo{z}, \bo{w}) \leq  \Pro(\bo{z}, \bo{w})+ O(\eta_1) |\bo{z}-\bo{w}|.$$
The proof follows immediately by noticing that since $\bo{z}$ and $\bo{w}$ are at distance at least $\ell_2$ and on $\sU^*_{\zeta}(n)$ by definition $\PT(\bo{z}, \bo{w})\ge \alpha |\bo{z}- \bo{w}|.$   
\end{proof}

We now define a gradient function corresponding to the projected distances analogous to \eqref{grad87}. As in the above setting let $\bo{z}, \bo{w}\in \Gr_{\sC n}(j_1+m_1/2)$  and let $\theta$ and $d>0$ be such that $\bo{w}=\bo{z}+(\theta,d)$.  Then let 
\begin{equation}\label{grad88}
\grad_{\Pro}(\bz,\theta, d)=\frac{\Pro(\bo{z}, \bo{w})}{|\bo{z}-\bo{w}|}.
\end{equation}
Once we have defined the projected gradients only for pairs of points in $\Gr_{\sC n}(j_1+m_1/2)$ we define projected gradients in all directions at a slightly coarser scale, i.e. for all points in $\Gr_{\sC n}(j_1+m_1/4)$. For any $\bo{z}\in \Ub_{\sC n}(j_1,v)\cap \Gr_{\sC n}(j_1+m_1/4)$ and for any $\theta\in \bS^1$ and $\frac{n}{2^{j_1}}>d>\frac{n}{2^{j_1+m_1/4}}$ let 
\begin{equation}\label{grad89}
\grad_{\Pro}(\bz,\theta, d)=\frac{\Pro(\bo{z}, \bo{w})}{d},
\end{equation}
where $\bo{w}$ is the closest point to $\bo{z}+(\theta,d)$ in  $\Gr_{\sC n}(j_1+m_1/2)$. Note that $\Gr_{\sC n}(j_1+m_1/4)\subset \Gr_{\sC n}(j_1+m_1/2)$. Thus \eqref{grad89} is defined via 
\eqref{grad88}.
If $\Ub_{\sC n}(j_1,v)$ is $(\delta_1,\ell_1,k_1)-\SST.$
then as in  \eqref{stabgrad12} along with an application of triangle inequality as in \eqref{tria1}, the following  result about smoothness of the projected gradient field follows whose proof we omit. 
\begin{lem}\label{smoothgradproj} For any $\bz, \bz' \in \Ub_{\sC n}(j_1,v)$  and $\theta_1, \theta_2 \in \bS^1,$  such that $|\theta_1-\theta_2|\le \eta_1$ and $d_1,d_2$  such that $\grad_{\Pro}(\bz,\theta_1, d_1)$ and $\grad_{\Pro}(\bz',\theta_2, d_2)$ are defined via \eqref{grad89} then  
$$
\frac{1}{1+O(\delta_1+\eta_1+2^{-m_1/4})}\le \frac{\grad_{\Pro}(\bz,\theta_1, d_1)}{\grad_{\Pro}(\bz',\theta_2, d_2)}< 1+O(\delta_1+\eta_1+2^{-m_1/4}). 
$$
\end{lem}
Note that above we choose $|\theta_1-\theta_2|\le \eta_1$ where the latter appeared in the definition of $\Pro.$ This is done deliberately to avoid introducing new notation since for us any small enough value of $\eta_1$ would serve both the purposes.

This allows us to define a projected gradient for the entire tile as we did  in Definition \ref{tilegrad}.
\begin{defn}\label{tilegrad2proj} If $\Ub_{\sC n}(j_1,v)$ is $(\delta_1,\ell_1,k_1)-\SST,$ then let
$$\grad_{\Pro}((j_1,v),\theta):=\grad_{\Pro}(\bz,\theta, d)$$ for some arbitrary $\bz\in \Ub_n(j_1,v)\cap \Gr_{\sC n}(j_1+m_1/2),$ and $d$ such that the RHS is defined via \eqref{grad89}. Note that the definition depends on the choice of $\bz$ and $d$ but only up to a multiplicative factor of $(1+O(\delta_1+\frac{1}{2^{m_1/4}})),$ which can be made arbitrarily close to one by choosing the parameters appropriately. 
Hence for concreteness we choose $\bz$ to be the center point of $\Ub_n(j,v)$ and $d=\frac{n}{2^{j_1+m_1/8}}.$ 
\end{defn}

Essentially the fact that  $\PT$ satisfies the triangle inequality (by definition) is what leads to the convexity of the limit shape $\cB$ in \eqref{e:shape}. One might therefore hope that  ${\Pro}$ satisfies an approximate triangle inequality.  To formally state things,  it would be convenient to consider the following function on entire $\R^2$ given by the following: for any $\bo{w}=(\theta,r)$, $$\|\bo{w}\|_{(j,v)}=\|(\theta,r)\|_{(j,v)}=r\grad_{\Pro}((j,v),\theta).$$ Note that as in Definition \ref{tilegrad2proj}, this definition implicitly depends on the choice of $\bz$ and $d.$ The next lemma shows the approximate convexity of the above defined function which allows us to think of the above as roughly a norm. 

\begin{ppn} \label{conc1}If $\Ub_{\sC n}(j_1,v)$ is $(\delta_1,\ell_1,k_1)-\SST,$  then 
for any set of vectors $\bo{w}_1,\bo{w}_2,\ldots,\bo{w}_t$,  if $\bo{w}=\sum_{i=1}^t \bo{w}_i$  then 
 $$\|\bo{w}\|_{(j_1,v)}\le (1+O(\delta_1+\frac{1}{2^{\frac{m_1}{16}}}))\left(\sum_{i=1}^t\|\bo{w}_i\|_{(j_1,v)}\right).$$
\end{ppn}

The proof even though relies on an approximate triangle inequality is a little technical and is postponed to Section \ref{pconc1}. 
For the next result,  given $\delta_1, \e_1,$ and $m_1$ satisfying the hypothesis of Lemma \ref{most34}, let $j_1$ be the scale obtained from that lemma and recall the definitions of $\ell_1$ and $k_1$ from the statement of the same. Recalling  $\eta_1$, $\Pro=\overset{\eta_1,\ell_2,j_1}{\Pro}$ where $\ell_2=\frac{\sC n}{2^{j_1+\frac{m_1}{2}}}$ from Lemma \ref{compare1}. consider the set of images $\sP\sV_{\eta_1,\ell_2,j_1}$ from \eqref{imageset}.

In the sequel  to avoid introducing new notation we will in fact denote $\frac{\log\P(\sU^*_{\zeta}(n))}{n^2}$ by $\kappa$ even though it was used to define the lim sup  of $\frac{\log\P(\sU_{\zeta}(n))}{n^2}$ in \eqref{limsupinf}. We now state the following easy consequence of the pigeon-hole principle. 
 
\begin{lem}\label{proj23} Given the parameters as above and $\e_4>0$
there exists $\Im\in \sP\sV_{\eta_1,\ell_2, j_1}$ and $A \subset \llbracket 1, 2^{j_1} \rrbracket^2 $ such that $|A|= \e_1 2^{2j_1}$, such that 
$$\log\frac{\P(\sU^*_{\zeta}(n)\cap \Pro^{-1}(\Im)\cap \bigl\{\{v \in \llbracket 1, 2^{j_1}\rrbracket^2 :v \text{ is } (\delta_1, \ell_1, k_1) -\US\}\subset A\bigr\})}{n^2}\ge \kappa-\e_4,$$ for all large enough $n.$
\end{lem}

\begin{proof}Recall the trivial bound mentioned in \eqref{imageset}, $$|\sP\sV_{\eta_1,\ell_2, j_1}|=e^{O(2^{2(j_1+m_1)}\log\frac{1}{\eta_1})}=e^{O(1)}.$$ Moreover the possible subsets $A$ of $\llbracket 1,2^{j_1}\rrbracket^2$ of size at most $\e_1 2^{2j_1}$ is at most $e^{O(H(\e_1))2^{2j_1}}$ where $H(\cdot)$ is the entropy functional. 
Thus by pigeon-hole principle the result follows. 
\end{proof}
Henceforth, for $A$ as in Lemma \ref{proj23},  we will denote the above event i.e.,
\begin{equation}\label{basev23}
 \sU^*_{\zeta}(n) \cap \Pro^{-1}(\Im)\cap \bigl\{\{v \in  \llbracket 1,2^{j_1}\rrbracket^2 :v \text{ is } (\delta_1, \ell_1, k_1) -\US\}\subset A\bigr\},
 \end{equation}
which will be our building block for later constructions as $$\Bas:=\Bas(\eta_1,\delta_1,\ell_1,k_1,\e_1).$$
Note that the above definition should also contain $A$ as a parameter which we are suppressing to avoid cluttering.
\section{Constructing a Large Deviation Event at a Higher Scale}
\label{s:construct}

In this section we prove Proposition \ref{p:pathconstruct}.
With the definition and results from the previous section at our disposal, following the strategy outlined in Section \ref{outline},
for  any $n_1\gg n$, we now proceed to creating the favourable event $\Fav:=\Fav(n_1)$ which will imply $\sU_{\zeta'}(n_1)$ where $\zeta'\ge \zeta-O(\e),$ for some small $\e,$ and moreover, $$\frac{\log\P(\Fav(n_1))}{n_1^2}\ge \frac{\log\P(\sU^*_{\zeta}(n))}{n^2}-O(\e).$$

We start by defining certain key ingredients:  Fixing $\e_6>0,$ for brevity we adopt the following abbreviations 
\begin{align}\label{abbre34}
\fn_0:=\sC n_1(1+2\e_6),\,\, \fn_1:=\sC n_1(1+\e_6),\,\,\fn_2:=\sC n_1, \fn_3:= \sC n(1+\e_6), \fn_4:= \sC n. 
\end{align}
Moreover in the sequel we will denote $\Bb(\fn_i)$ as $\fb_i.$
$\Fav$ will be a function of the edges in $\mathfrak{B}_0,$
 with the property that on the event $\Fav,$  
$$\PT_{\fb_0}(\bo{0},\fb_0^c)\ge bn_1, \text{ and }
\PT_{\fb_0}(\bo{0}, \bo{n_1})\ge (\mu+\zeta-O(\e))n_1,$$ for some small $\e.$ Clearly this implies  that $\Fav \subset \sU_{\zeta'}(n_1)$ for some $\zeta'=\zeta-O(\e)$.
The basic geometry we shall be working with is the following.  Fix $j\in \N$. Tile the box $\fb_0$ by $\Ub_{\fn_0}(j,v)$ for $v\in [2^j]^2$. Now each such tile is a square of size $\frac{\fn_0}{2^j}$. For $v\in [2^j]^2$, consider the square with the same centre as $\Ub_{\fn_0}(j,v)$
and side length $\frac{\fn_1}{2^j}$. Call this square (closed) $\Ub^{*}_{\fn_1}(j,v)$; see Figure \ref{fig8}. It follows that neighbouring $\Ub^{*}_{\fn_1}(j,v)$'s are separated by vertical and horizontal strips of width at most $\e_6 \frac{\fn_1}{2^j}$. For obvious reasons, the set of all edges in $\fb_0$ that does not belong to any $\Ub^{*}_{\fn_1}(j,v)$ is called $\An^{\rm{ext}}(j,\fn_0)$. ($\rm{ext}$ stands for exterior, we will also consider corridors inside $\Ub^{*}_{\fn_1}(j,v)$) \textbf{Without loss of generality we shall assume that $\bo{0}=(0,0)$ and $\bo{n_1}=(n_1,0)$ are at the center of some (different) $\Ub^{*}_{\fn_1}(j,v)$'s}. 

\begin{figure}[h]
\centering
\includegraphics[scale=.6]{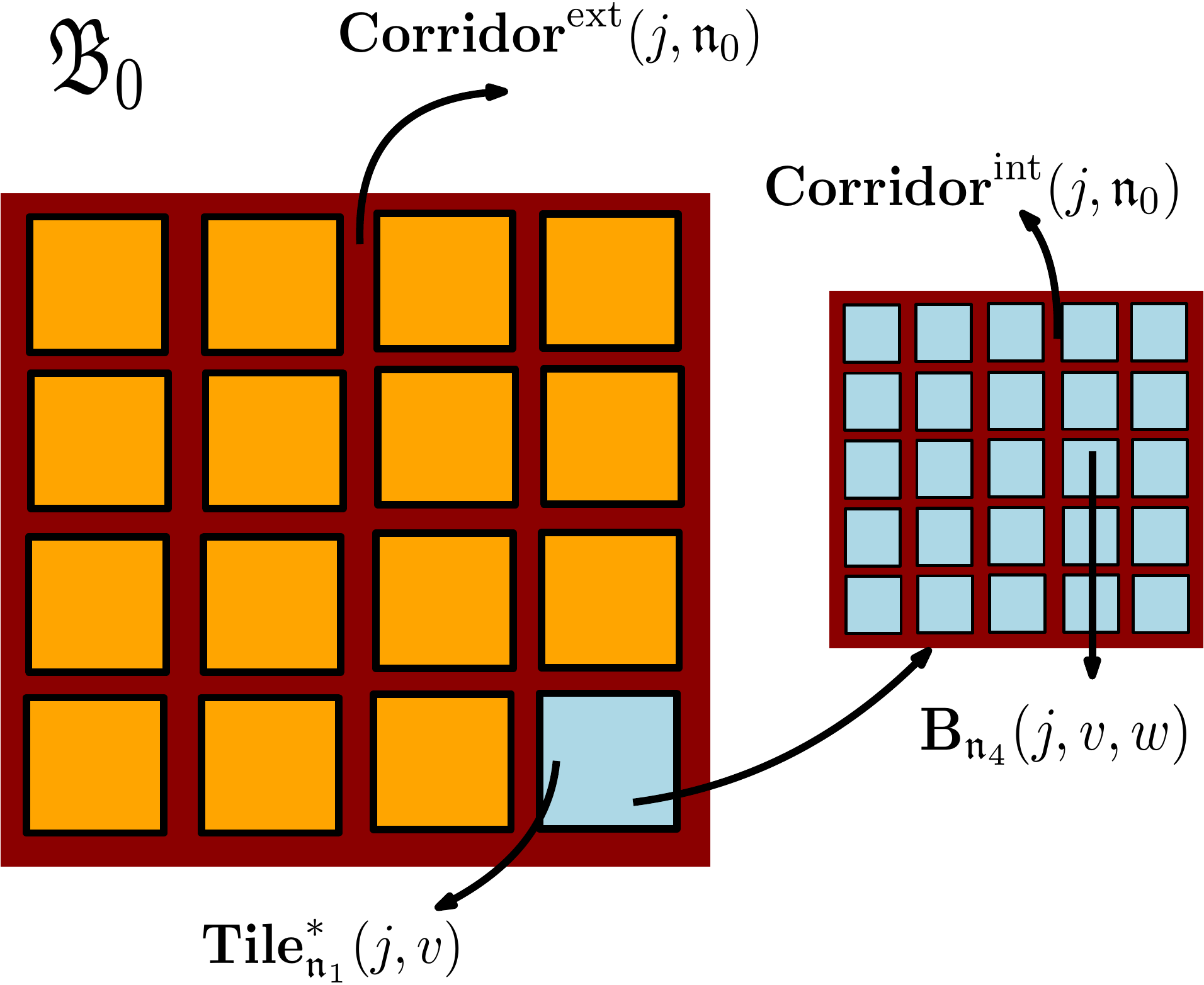}
\caption{The figure illustrates the basic structural definitions inside $\fb_0$. On the LHS the figure shows the $\An^{\rm{ext}}(j,\fn_0)$ (red region) and the tiling of the remaining area by $\Ub^{*}_{\fn_1}(j,v).$ The RHS zooms into one particular $\Ub^{*}_{\fn_1}(j,v),$ (the south-east one) and shows  $\bo{B}_{\fn_4}(j,v,w)$ and the surrounding $\bo{C}_{\fn_4}(j,v,w)$ which form a part of $\An^{\rm{int}}(j,\fn_0)$.}
\label{fig8}
\end{figure}

  Our construction of $\Fav$ will have two steps: 
  \begin{enumerate}[(i)]
\item Specifying the environment inside $\Ub^*_{\fn_1}(j,v)$ for various $v \in \llbracket1,2^j\rrbracket^2.$ 
\item Specifying the environment in $\An^{\rm{ext}}(j, \fn_0)$. 
\end{enumerate}
Part (i) involves a large deviation environment in the smaller scale $n$, whereas for the second part we just make all the edge weights close to $b$. We shall formalize part (i) later, but for now let us make part (ii) formal as follows. Let $\Ba^{\rm{ext}}(\fn_0,j)$ denote the event that the passage time on each edge in $\An^{\rm{ext}}(j, \fn_0)$ is in $[b-\e_7,b]$ for some small but fixed $\e_7$. As the total number of the edges in $\An^{\rm{ext}}(j, \fn_0)$ is $O(\e_6 \fn_0^2)$, it follows that $-\log \P (\Ba^{\rm{ext}}(\fn_0,j))= O(\e_6 n_1^2)$ (the constant in the $O(\cdot)$ notation depends on $\e_7,$ and $\e_6$ will be chosen to be much smaller than $\e_7$ depending on the edge distribution $\nu$). 

Recalling that the goal is to create an event on which the FPP distance (within the box $\fb_0$) between $\bo{0}$ and $\bo{n_1}$ is forced to be large, having constructed $\Ba^{\rm{ext}}(\fn_0,j)$ we are left to do  two more things:
\begin{enumerate}
\item Specifying the environments inside $\Ub^{*}_{\fn_1}(j,v)$ using the large deviation $\Bas$ defined in Lemma \ref{proj23}.
\item Using the above showing that any path $\gamma$ between $\mathbf{0}$ and $\bo{n_1}$ contained in $\fb_0$ has length $(\mu+\zeta-O(\e))n_1$. However to be able to use the properties of $\Bas$ (in particular the stability properties) we need some regularity properties of $\gamma.$  Hence the first step, given such an arbitrary $\gamma$ is to preprocess it to obtain another path $\cP(\gamma)$ from $\mathbf{0}$ and $\bo{n_1}$ such that the path $\cP(\gamma)$ has the desired regularity properties, and, on the event $\Ba^{\rm{ext}}(\fn_0,j)$, has length within a factor $(1+o(1))$ of the length of $\gamma$.
\end{enumerate}

To accomplish the first part for any $j$ and $v \in \brj,$ it will be convenient to think of each $\Ub^*_{\fn_1}(j,v)$ as naturally made up of $(\frac{n_1}{n})^2$ copies of $\Ub_{\fn_3}(j,v).$  
 We will denote the copy of the tile as $\bo{A}_{\fn_3}(j,v,w)$ for $w \in \llbracket 1, \frac{n_1}{n} \rrbracket^2,$
and as before each $\bo{A}_{\fn_3}(j,v,w)$ can be thought of as a copy of $\Ub_{\fn_4}(j,v)$ to be called $\bo{B}_{\fn_4}(j,v,w)$ surrounded by  an annulus $\bo{C}_{\fn_4}(j,v,w)$ of width $\frac{\e_6}{2}\frac{\fn_4}{2^j}$, (see Figure \ref{fig8}).
As before we denote the union of edges in $\bo{C}_{\fn_4}(j,v,w)$ union over $v\in \brj$ and $w \in \llbracket 1, \frac{n_1}{n} \rrbracket^2$ as $\An^{\rm{int}}(j, \fn_0).$ Now similar to $\Ba^{\rm{ext}}(\fn_0,j)$ let $\Ba^{\rm{int}}(\fn_0,j)$ denote the event that the passage time on each edge in $\An^{\rm{int}}(j, \fn_0)$ is in $[b-\e_7,b]$ and similar considerations as before show that $-\log \P (\Ba^{\rm{int}}(\fn_0,j))= O(\e_6 n_1^2)$.
 
We now prescribe the environment inside  $\bo{B}_{\fn_4}(j,v,w)$.
Since there are many parameters involved, to avoid repetition
throughout this section we will work with the choice of parameters as in Lemma \ref{proj23}. Note that this causes us from now to work with a specific scale $j_1$ and not a generic scale $j.$
Recall the set $A$ of size $\e_1 2^{2j_1}$ in the statement of Lemma \ref{proj23}. 
\subsection{Construction of $\Fav$:} At a high level the event $\Fav$ will be an intersection of three independent events i.e., 
$\Fav:=\Di \cap \Ba \cap \Boo,$ where the three events on the RHS will be independent.  We will use $\Ba$ to denote the intersection of the events $\Ba^{\rm{ext}}(\fn_0,j_1)$ and $\Ba^{\rm{int}}(\fn_0,j_1).$  $\Boo$ and $\Di$ will be used to define the edge weights in $\bo{B}_{\fn_4}(j_1,v,w)$ where $w \in \llbracket 1, \frac{n_1}{n}\rrbracket^2$ and $v\in A$ and $v\in \llbracket 1, 2^{j_1} \rrbracket^2\setminus A$ respectively. 
We define the event that the passage time on all the edges in $\bigcup_{v\in A,w}\bo{B}_{\fn_4}(j_1,v,w)$ is in $[b-\e_7,b]$ as $\Boo.$

Finally we define the event $\Di$  in the following constructive way:
Sample $\left(\frac{n_1}{n}\right)^2$ many independent realizations of $\Bas$ which yields environments $$\Pi_1,\Pi_2\ldots, \Pi_{\left(\frac{n_1}{n}\right)^2}$$
such that $\Pi_i\in \Bas,$
for all $i \in \llbracket 1, \frac{n_1}{n}\rrbracket^2.$ 
For each $v\in \llbracket 1, 2^{j_1} \rrbracket^2\setminus A$ and $w \in \llbracket 1, \frac{n_1}{n}\rrbracket^2$,  let the edge weights on the edges in $\bo{B}_{\fn_4}(j_1,v,w)$ be the same as the edge weights of $\Pi_w$ in $\Ub_{\fn_4}(j_1,v)$ where we use the natural identification between $\bo{B}_{\fn_4}(j_1,v,w)$ and $\Ub_{\fn_4}(j_1,v).$ 

Note that the choice of the term 
 $\Di$ to denote the above event is natural, as  by using   $(\frac{n_1}{n})^2$ copies of $\Bas$ we ensure that for any $v \in \llbracket 1, 2^{j_1} \rrbracket^2$, the environments in different $\bo{B}_{\fn_4}(j_1,v,w)$ are essentially the same.
Now the event  $\Ba$ along with $\Di$  describe the projection of the event $\Fav$ on all the edges except the edges in $\bigcup_{v\in A, w\in \llbracket 1, \frac{n_1}{n}\rrbracket^2}\bo{B}_{\fn_4}(j_1,v,w)$ whereas $\Boo$ defines those in the latter.  
Hence 
\begin{equation}\label{lb7896}
\P(\Fav)=[\P(\Bas)]^{\frac{n_1^2}{n^2}}\nu([b-\e_7,b])^{O((\e_1+\e_6) n_1^2)},\end{equation} 
where $\nu$ is the passage time distribution satisfying the hypothesis in Theorem \ref{t:ldp}.

The proof of Proposition  \ref{p:pathconstruct} will now be complete from the following lemma.
\begin{lem}\label{verify} Given $\e_8$ and $\e_9$ there exists choice of parameters in the definition of $\Bas$ in Lemma \ref{proj23},  and $\e_6,\e_7$ in the definition of $\Fav$,  such that $$\frac{\log(\P(\Fav))}{n_1^2}\ge \kappa-\e_8, \text{ and, } \Fav\subset \sU_{\zeta'}(n_1),$$
where $\zeta'>\zeta-\e_9.$
\end{lem}

Note that the lower bound on the probability of $\Fav$ is a straightforward consequence of \eqref{lb7896} and Lemma \ref{proj23}.
The rest of the discussion is devoted to the proof of the second part which will follow from a series of lemmas.
Before stating the lemmas we roughly describe our strategy. 
The proof involves broadly showing that on the event $\Fav$  two things occur:

\begin{align}
\label{part1}
\PT_{\fb_0}(\bo{0},\bo{n_1})& \ge (\mu+\zeta')n_1, \\
\label{part2} 
\PT_{\fb_0}(\bo{0},\fb_0^c)& \ge b n_1.
\end{align}
Now the proof of both the above bounds is obtained by the same strategy. 
Keep in mind the two random fields given by $\Fav$ and $\Bas$  on $\fb_0$ and $\fb_4$ respectively. Recall that the former is a `dilation' of the latter by a factor of $\frac{\fn_2}{\fn_4},$ with some additional changes including the setting up of the barriers  and the boosting on the unstable tiles.

As outlined in Section \ref{outline}, given the above,  the strategy is to show that for any path $\gamma$ (joining $\bo{0}$ and $\bo{n_1}$) in $\fb_0,$ there exists a scaled version  $\gamma_{\Sc}$ (joining $\bo{0}$ and $\bo{n}$) in $\fb_4$ such that 
\begin{equation}\label{scale879}
|\gamma|\ge \frac{n_1}{n}(1-o(1))|\gamma_{\Sc}|,
\end{equation}
 where the LHS is computed on $\Fav$ and the RHS is computed on $\Bas$.
Thus $\gamma$ can be thought of as a path obtained by dilating the path $\gamma_{\Sc}$.

Since $\gamma_{\Sc}$ is a path in the random field given by $\Bas,$ it follows by definition of the latter that $|\gamma_{\Sc}|\ge (\mu+\zeta)n$ and this yields the sought lower bound of $|\gamma|.$ 
To make \eqref{scale879} formal  we need some regularity properties of the path $\gamma$ which will be obtained by some preprocessing. This is done in the next section.\subsection{Preprocessing of Paths}
\label{s:process}
 Observe that given any path $\gamma$ contained in $\fb_0$  it admits a unique decomposition as a concatenation of a number of paths i.e., $\gamma= \alpha_0\chi_0\alpha_1\chi_1\alpha_2\chi_2\ldots \alpha_{L}\chi_{L}\alpha_{L+1}$ with the following properties:
\begin{enumerate}[i.]
\item Each $\alpha_i$ is contained in some $\Ub^*_{\fn_1}(j_1,v_i)$ for some $v_i \in \llbracket 1, 2^{j_1} \rrbracket ^2$; $\alpha_0$ and $\alpha_{L+1}$ could be empty. 
\item Each $\chi_i$ is contained in $\An^{\rm{ext}}(j_1, \fn_0)$.
\end{enumerate}
\begin{figure}[h]
\centering
\includegraphics[scale=.7]{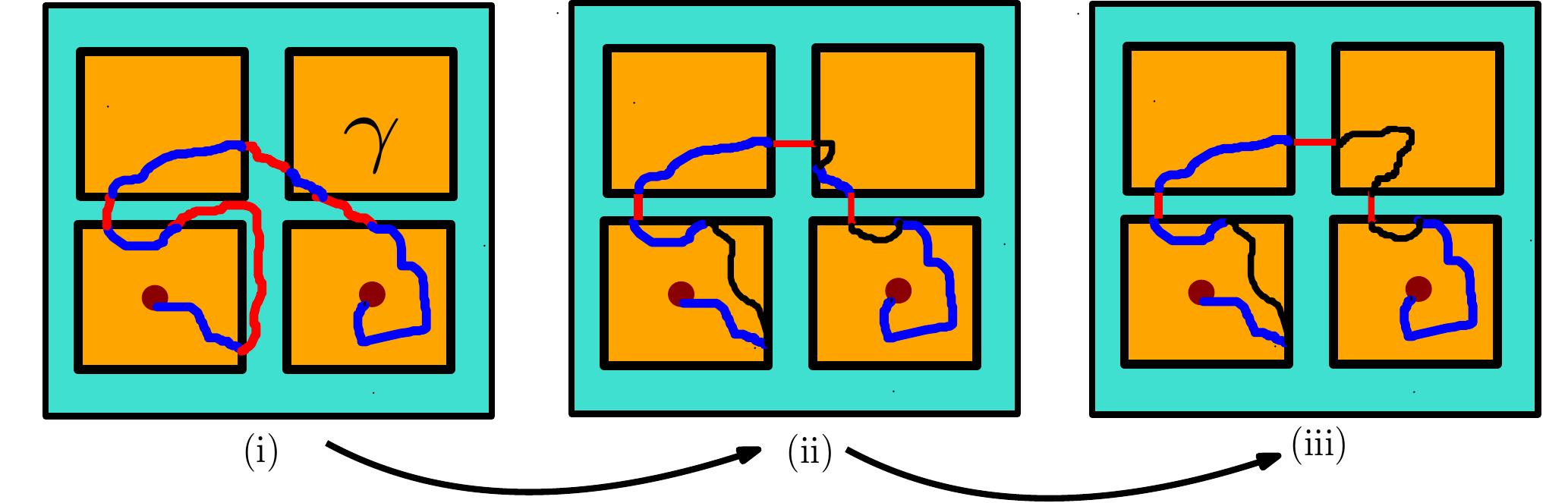}
\caption{A schematic diagram describing the preprocessing. (i)  illustrates the decomposition of the path $\gamma$, into $\alpha_i$ (blue segments) and $\chi_i$ (red segments). (ii) describes the content of Lemma \ref{l:outside} where each red segment is replaced by a regular path. (iii) describes the content of Lemma \ref{l:excursion} where if an excursion is small (the part in the north-east tile) then we replace it by a larger excursion without changing the length too much. }
\label{fig9}
\end{figure}
Given $\e_6$ as in \eqref{abbre34} let us call the paths $\alpha_i$ for $i\in \{1,2,\ldots, L\}$ as \textbf{excursions} of $\gamma$ and let us call the above decomposition of $\gamma$ its decomposition into excursions. Let $\mathbf{x}_i$ (resp.\ $\mathbf{y}_i$) denote the starting (resp.\ ending) vertex of $\alpha_i$. Let us call the excursion $\alpha_i$ \textbf{large} if there exists a vertex $\bz_i$ on $\alpha_i$ such that $\min\{|\bx_i-\bz_i|, |\by_i-\bz_i|\} \geq \e_6^2\frac{\fn_0}{2^{j_1}}$. Observe that $\alpha_i$ is large if $|\mathbf{x}_i-\mathbf{y}_i|\geq 2\e_6^2\frac{\fn_0}{2^{j_1}}$.
 
We shall need to define one more property of a path. Consider a path $\gamma$ with the decomposition into excursions as above. Observe that each $\chi_i$ must start at $\Ub^{*}_{\fn_1}(j_1,v)$ and end at some $\Ub^{*}_{\fn_1}(j_1,v')$ for some $v=v(\chi_i),v'=v'(\chi_i)\in [2^{j_1}]^2$. We call the path $\gamma$ \textbf{regular} if for each $\chi_i$ we have $\|v(\chi_i)-v'(\chi_i)\|_1=1$ (i.e., they are neighbouring vertices) and the starting point of $\chi_i$ lies in the same vertical (if $v$ and $v'$ are on the same vertical line) or horizontal (if $v$ and $v'$ are on the same horizontal line) line as its endpoint.
Recall the parameters $\e_6$ and $\e_7$ in the definition of the event $\Fav.$
\begin{lem} \label{decomp120} 
For any path $\gamma$ starting at $\mathbf{0}$ and ending at $\bo{n_1}$  and contained in $\fb_0$, there exists a regular path $\cP$ from $\mathbf{0}$ to $(n_1,0)$ such that 
\begin{enumerate}[i.]
\item All the excursions of $\cP$ are large. 
\item On $\Ba^{\rm{ext}}(\fn_0,{j_1})$, we have $|\gamma|\ge (1-O(\e_7+\e_6))|\cP|.$
\end{enumerate}
\end{lem}
The proof of the above lemma is done in two steps (see Figure \ref{fig9} for an illustration). Let $\gamma$ be fixed as in the lemma. Consider its decomposition into excursions: $\gamma=\alpha_0\chi_0\alpha_1\chi_1\alpha_2\chi_2\cdots$. Observe that if we can replace each $\chi_i$ by a regular path with the same endpoints, the resulting path will be regular. The following lemma shows that this can be done without increasing the length of the path by more than a factor of $(1-O(\e_7))^{-1}$. 

\begin{lem}
\label{l:outside}
Consider a path $\chi$ completely contained in $\An^{\rm{ext}}(j_1, \fn_0)$ whose starting and ending points are located at the boundary of $\Ub^{*}_{\fn_1}(j_1,v)$ and $\Ub^{*}_{\fn_1}(j_1,v')$ respectively. Then there exists a regular path $\cP_{\chi}$ with the same starting and ending point such that on $\Ba^{\rm{ext}}(\fn_0,{j_1})$, we have $|\chi|\geq (1-O(\e_7))|\cP_{\chi}|$. 
\end{lem}

\begin{proof}
Let $\bx$ and $\by$ be the starting and ending point of $\chi$ respectively. Consider the $1-$norm minimizing path from $\bx$ to $\by$ that constitutes of a horizontal path followed by a vertical path (this choice is arbitrary): i.e., for $\bx=(x_1,x_2)$ and $\by=(y_1,y_2)$ consider the piecewise linear curve $\mathbb{L}$ obtained by concatenating the straight line segment obtained by joining $\bx$ to $(y_1,x_2)$ followed by the straight line segment obtained by joining $(y_1,x_2)$ to $\by$. Consider $\L$ as a path on the nearest neighbour graph of $\Z^2$. Observe that there exists points $u_0=\bx, u_1, \ldots, u_{\ell}=\by$ on $\mathbb{L}$ all on boundaries of $\Ub^{*}_{\fn_1}(j_1,u)$'s  such that the $\mathbb{L}$ restricted between $u_{i}$ and $u_{i+1}$ (called $\mathbb{L}_i$) is either (a) contained in $\Ub^{*}_{\fn_1}(j_1,u)$ for some $u$ (type A, say) or (b) is entirely contained in $\An^{\rm{ext}}(j_1, \fn_0)$, and further $u_i\in \Ub^{*}_{\fn_1}(j_1,u), u_{i+1}\in \Ub^{*}_{\fn_1}(j_1, u')$ for some $u,u'$ that have $\ell_1$ distance one (type $B$). Observe again that such a decomposition is unique. Now if $\mathbb{L}_i$ is type $A$ let us set $\cP_i$ to be the shortest path between $u_i$ and $u_{i+1}$ contained in $\Ub^{*}_{\fn_1}({j_1},u)$, and if $\mathbb{L}_i$ is type $B$ we set $\cP_i=\mathbb{L}_i$. Consider the path $\cP_{\chi}=\cP_0\cP_1\cdots \cP_{\ell-1}$ obtained by concatenating $\cP_{i}$'s. It is clear that the path $\cP_{\chi}$ obtained as above is regular, (see Figure \ref{fig9} for an illustration) and hence it only remains to show that on $\Ba^{\rm{ext}}(\fn_0,{j_1})$, we have $|\chi|\geq (1-O(\e_7))|\cP_{\chi}|$. Observe first that, on $\Ba^{\rm{ext}}(\fn_0,{j_1})$, we have $|\chi|\geq (b-\e_7)\|\bx-\by\|_1$. It also follows from definitions that $|\cP_i|\leq b\|u_i-u_{i+1}\|_1$. The lemma now follows from observing that $\sum_{i} \|u_i-u_{i+1}\|_1= \|\bx-\by\|_1$. 
\end{proof}

Lemma \ref{l:outside} tells us that for any $\gamma$ as in the statement of Lemma \ref{decomp120} one can replace the paths $\chi_i$ in its decomposition by the paths $\cP_{\chi_i}$ as constructed in Lemma \ref{l:outside} to end up with a regular path $\cP_{*}$ with the same endpoints such that $|\gamma|\geq (1-O(\e_7))|\cP_*|$. The following lemma ensuring the largeness of the excursions, therefore will suffice to complete the proof of Lemma \ref{decomp120}.

\begin{lem}
\label{l:excursion}
For any regular path $\gamma$ starting at $\mathbf{0}$ and ending at $\bo{n_1}$ and contained in $\fb_0$, there exists a regular path $\cP$ from $\mathbf{0}$ to $\bo{n_1}$ such that 
\begin{enumerate}[i.]
\item Each excursion of $\cP$ is large. 
\item On $\Ba^{\rm{ext}}(\fn_0,j_1)$, we have $|\gamma|\ge (1-O(\e_6))|\cP|.$
\end{enumerate}
\end{lem}

\begin{proof}
Let $\gamma$ be as in the statement of the lemma. Consider its decomposition into excursions $\gamma=\alpha_0\chi_0\alpha_1\chi_1\alpha_2\chi_2\cdots \alpha_{L}\chi_{L}\alpha_{L+1}$. The proof, again will be a step by step procedure, we shall inspect the short excursions one by one, and remove them by modifying the path locally without increasing the lengths too much. Let $\alpha_i$ be contained in 
$\Ub^{*}_{\fn_1}(j_1,v_i)$. We shall establish the following: for any excursion $\alpha_{i}$ that is not large, there exists a path $\alpha'_{i}$ contained in $\Ub^{*}_{\fn_1}(j_1,v_i)$ with the same starting and ending point as $\alpha_{i}$  such that: (i) $\alpha'_i$  is a large excursion and (ii) on $\Ba^{\rm{ext}}(\fn_0,j_1)$, we have $|\alpha_i \chi_{i}|\geq (1-O(\e_{6}))|\alpha'_i \chi_i|$. Before proving this, let us observe that this clearly suffices. Consider the path $\cP=\alpha_0\chi_0\alpha'_1\chi_1\cdots \alpha'_{L}\chi_{L}\alpha_{L+1}$ where $\alpha'_{i}$ is as above if $\alpha_{i}$ is not a large excursion and $\alpha'_{i}=\alpha_{i}$ otherwise. Clearly the above exhibits a decomposition of $\cP$ into excursions which ensures that $\cP$ is regular. The second assertion of the lemma is immediate from the bound on $|\alpha'_i \chi_i|$. It remains to prove the claim. 

Consider any excursion $\alpha_i$ that is not large. Let $\bx_i$ and $\by_i$ be its starting and ending points respectively. Fix a vertex $\bz_i$ in $\Ub^{*}_{\fn_1}(j_1,v_i)$ such that $|\bx_i-\bz_i|, |\by_i-\bz_i|\in (\e_6^2\frac{\fn_0}{2^{j_1}}, 2\e_6^2\frac{\fn_0}{2^{j_1}})$; clearly such a vertex exists. now consider the path $\alpha'_i=\alpha^{(1)}_i\alpha^{(2)}_i$ where $\alpha^{(1)}_i$ (resp.\ $\alpha^{(2)}_i$) is the shortest path between $\bx_i$ and $\bz_i$ (resp.\ $\bz_i$ and $\by_i$) contained in $ \Ub^{*}_{\fn_1}({j_1},v_i)$. Clearly $\alpha'_i$ is a large excursion. To get an upper bound on $|\alpha'_i\chi_i|$, observe that $|\alpha'_i|\leq 8b\e_6^2 \frac{\fn_0}{2^{j_1}}$ and on $\Ba^{\rm{ext}}(\fn_0,{j_1})$,  by taking $\e_6,\e_7$ sufficiently small we have $|\chi_i|\geq (b-\e_7)\e_6 \frac{\fn_0}{2^{j_1}}$, and $|\alpha_i \chi_{i}|\geq (1-O(\e_6))|\alpha'_i \chi_i|$. This completes the proof of the lemma. 
\end{proof}

Given the regular path $\cP=\alpha_0\chi_0\alpha_1\chi_1\alpha_2\chi_2\ldots \alpha_{L}\chi_{L}\alpha_{L+1}$ from Lemma \ref{decomp120}, we use essentially the same arguments on each of the excursions $\alpha_i$ as in the proof of Lemma \ref{decomp120} to obtain a further decomposition in to excursions, i.e. if $\alpha_i$ is contained in $\Ub^*_{\fn_1}(j_1,v_i)$ then the further excursions would be contained in $\bo{B}(\fn_4,v_i,w)$ for some $w\in \llbracket1, \frac{n_1}{n} \rrbracket^2.$ We now work with the obvious adaptations of the terms $\textbf{regular}$ (replacing $\Ub^*_{\fn_1}(j_1,v)$ and $\Ub^*_{\fn_1}(j_1,v')$ by $\bo{B}(\fn_4,v_i,w)$ and $\bo{B}(\fn_4,v_i,w')$ respectively) and $\textbf{large}$ (replacing $\fn_0$ by $\fn_4$).

Using the above altered definitions along with the same argument as before we obtain the following {whose proof we omit}.
\begin{figure}[h]
\centering
\includegraphics[scale=.5]{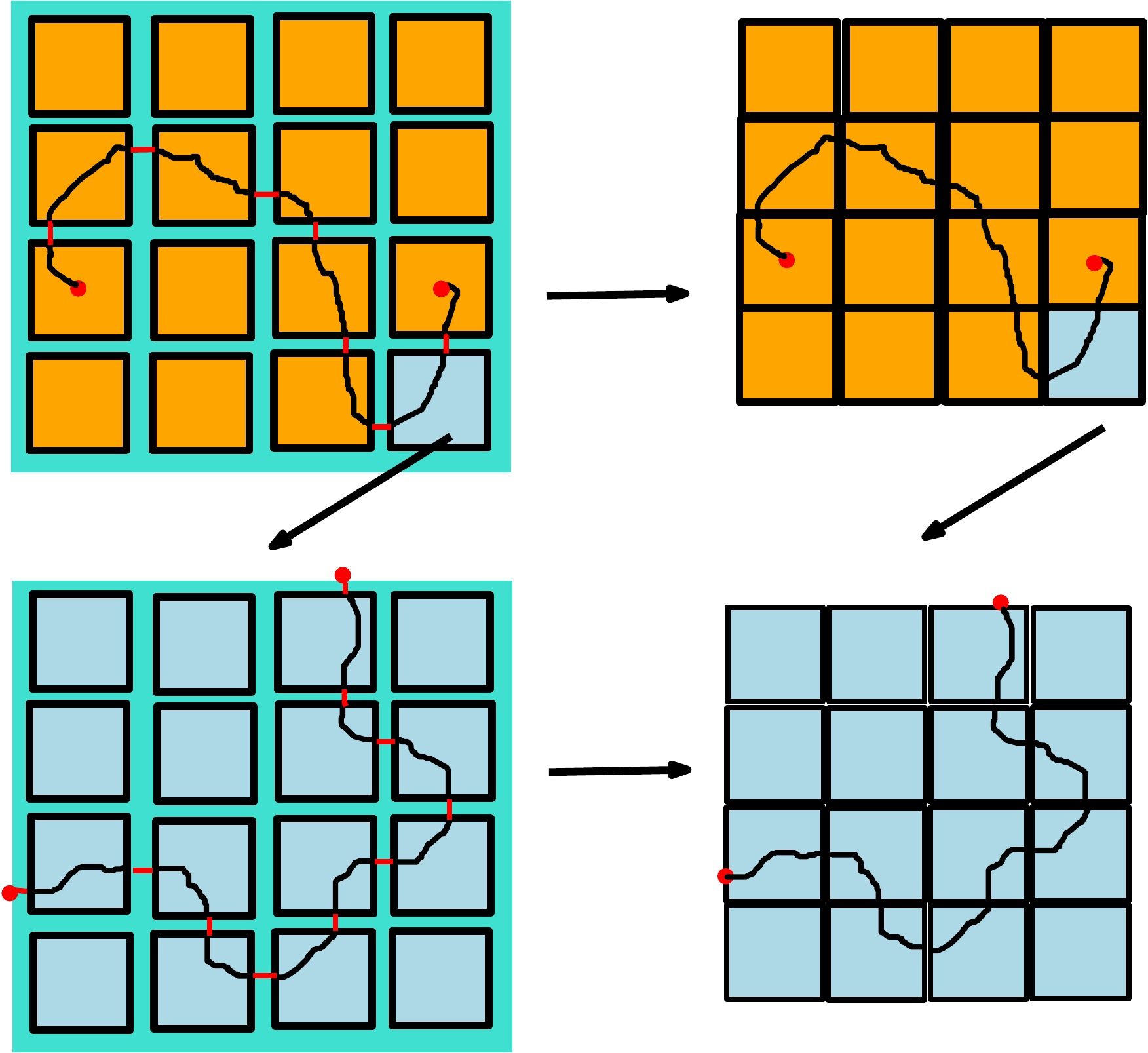}
\caption{The figure illustrates a natural identification between $\fb_{0}$ and $\Bb(\sC n_1).$ The top two figures show the effect of ignoring $\An^{\rm{ext}}(j_1, \fn_0)$. The bottom two figures zoom in the on the south-east tile and shows the effect locally of ignoring $\An^{\rm{int}}(j_1, \fn_0)$.}
\label{fig10}
\end{figure}

\begin{lem} \label{decomp121} For any $\cP=\alpha_0\chi_0\alpha_1\chi_1\alpha_2\chi_2\ldots$  satisfying the properties listed in Lemma \ref{decomp120}, for each $\alpha_i, i\ge 1$  there is a regular path  $\beta_i$ with the same starting and ending points as $\alpha_i$ and a decomposition into excursions 
$\beta_i=\beta_{i,1}\chi_{i,1}\beta_{i,2}\chi_{i,2}\ldots$
such that 
\begin{enumerate}[i.]
\item All the excursions of $\beta_i$ are large. 
\item On $\Ba^{\rm{int}}(\fn_0,j_1)$, we have $|\alpha_i|\ge (1-O(\e_7+\e_6))|\beta_i|.$
\end{enumerate}
\end{lem}

Equipped with these results we are now ready to prove \eqref{scale879}.
In the next few lemmas we create a scaled version $\cP^{\Sc}$ of path $\cP$ mentioned in the  above lemma. What we do is rather simple and natural.  Consider the 
$\cP=\alpha_0\chi_0\alpha_1\chi_1\alpha_2\chi_2\ldots,$ and then using Lemma \ref{decomp121} let $\cP'=\beta_1\chi_1\beta_2\chi_2\ldots $
and moreover consider the decomposition of each $\beta_i$ as provided by the last lemma. 

Now by the regularity of the path $\cP$ and all the $\beta_i$'s there is a natural path one can form in $\Bb(\sC n_1)$ by squishing all the corridors. 
Formally one can naturally identify $$\fb_0 \setminus (\An^{\rm{ext}}(j_1, \fn_0) \cup \An^{\rm{int}}(j_1, \fn_0))$$ with $\Bb(\sC n_1)$.  This allows one to identify with the path $\cP',$ a path $\cP^*$ in $\Bb(\sC n_1)$ formed by ignoring all the bridges $\chi_{i,j}$ and $\chi_i$. This is possible since in the above identification of $\fb_0 \setminus \left\{\An^{\rm{ext}}(j_1, \fn_0) \cup \An^{\rm{int}}(j_1, \fn_0)\right\}$ with $\Bb(\sC n_1)$ the endpoints of $\chi_{i,j}$ or $\chi$ map to adjacent points in $\Bb(\sC n_1)$ (see Figure \ref{fig10}). 

Under the above operation $\cP^*$ admits a decomposition $\cP^*=\cP_1\cP_2\ldots$ where  $\cP_i$ belongs to  $\Ub_{\sC n_1}(j_1,v_i)$ where $v_i \in \llbracket 1, 2^{j_1}\rrbracket^2$ is such that $\alpha_i\subset \Ub^*_{\fn_1}(j_1,v_i).$ 
Let the starting and ending points of $\cP_i$ be $x_i$ and $y_i$.
Recalling $\ell_2, j_1$ from the definition of $\Bas,$ let 
  $x^{\Sc}_i$ be the closest point in $\Gr_{\fn_4}(\ell_2;j_1)$ to $\frac{n}{n_{1}}x_{i}$ and similarly let $y^{\Sc}_i$ be the closest point in $\Gr_{\fn_4}(\ell_2;j_1)$ to $\frac{n}{n_{1}}y_{i}$ (see Figure \ref{fig11}).

Since $x_{i+1}$ and $y_i$ are adjacent it follows that $x^{\Sc}_{i+1}$ and $y^{\Sc}_i$ are adjacent points in $\Gr_{\fn_4}(\ell_2;j_1)$. 
Now a natural candidate for $\alpha^{\Sc}$ in \eqref{scale879} would be to take the shortest path passing through the points $x^{\Sc}_{1},y^{\Sc}_{1},x^{\Sc}_{2},y^{\Sc}_{2},\ldots$. However  it is a little inconvenient notationally since  $x^{\Sc}_{i+1}$ and $y^{\Sc}_i$ are not necessarily adjacent in $\fb_{4}.$ Since $\ell_2$ by our choice of parameters will be much smaller than the distance between $x^{\Sc}_{i}$ and $x^{\Sc}_{i+1}$  we will in fact ignore $y^{\Sc}_{i}$ and define $\tilde x^{\Sc}_{i}$ to be the point adjacent to $x^{\Sc}_{i+1}$ contained in the tile containing $\tilde x^{\Sc}_{i}$ (namely $\Ub_{\fn_4}(j_1,v_i)$).

\begin{figure}[h]
\centering
\includegraphics[scale=.9]{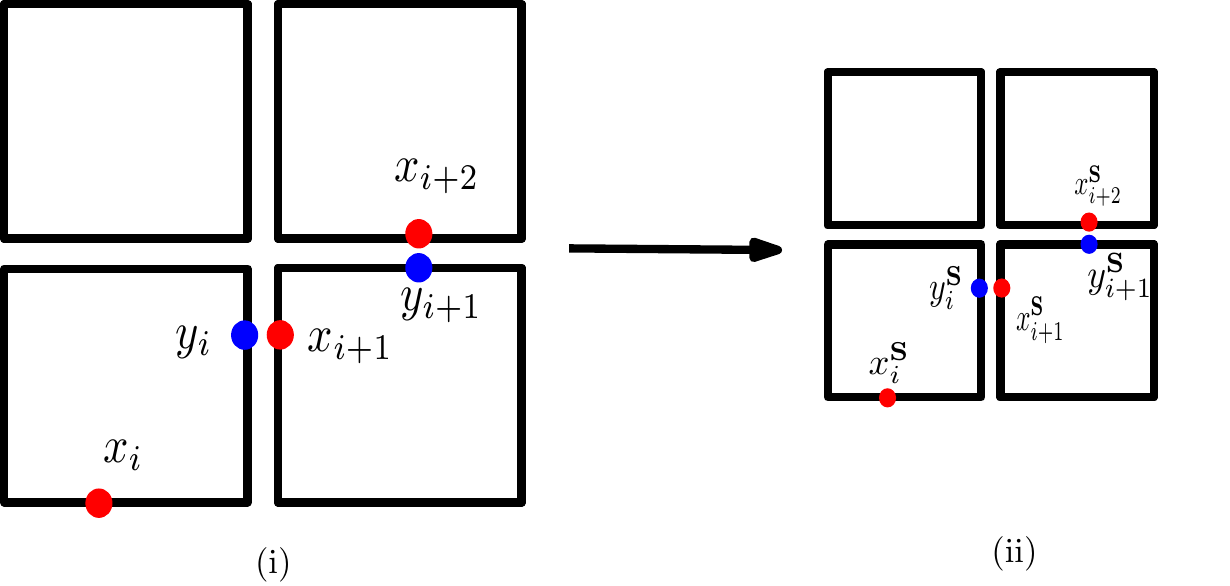}
\caption{(i) illustrates the  $x_i$ and $y_i$ denoted by red and blue colors respectively and (ii)  the corresponding scaled picture.}
\label{fig11}
\end{figure}

Given the above we now let $\alpha^{\Sc}_i$ to be the shortest path between $ x^{\Sc}_{i}$ and $\tilde x^{\Sc}_{i}$
 and define $\alpha^{\Sc}$ to be the concatenation of $\alpha^{\Sc}_i$ thought of as a sequence of vertices.  This is a valid construction as by the above discussion the endpoint of $\alpha^{\Sc}_i$ ($\tilde x^{\Sc}_{i}$) is adjacent to the starting point of $\alpha^{\Sc}_{i+1}$ ($x^{\Sc}_{i+1}$).
  
As a consequence of the above lemma and the approximate convexity statement in Proposition \ref{conc1} we have the following key result. 
\begin{lem}\label{lb123} Given any $\e_{11}>0,$  there exists a choice of the parameters in the definition of $\Bas$ and $\Fav$ such that,  deterministically for any $i$. 
$$|\alpha_i|\ge (1-\e_{11})\frac{n_1}{n}|\alpha^{\Sc}_i|,$$
where the LHS is computed on the event $\Fav$ and the RHS is  computed on any environment in $\Bas.$ \footnote{Note that the points $x_{i}^{\Sc}$ are points in $\Gr_{\fn_4}(\ell_2;j_1)$ and hence as $\Pi$ varies over $\Bas$, the length of the path $\alpha^{\Sc}_i$ can at worst change by a multiplicative factor $(1+\eta_1)$ where $\eta_1$ appears in the definition of $\Bas$.}
\end{lem}
Before proving the above lemma we finish the proof of Lemma \ref{verify} using the above results.
\begin{proof}[Proof of Lemma \ref{verify}]
The proof will clearly follow by showing \eqref{part1} and \eqref{part2}. 
We will show only the former and the latter has an identical proof. 
Fix any path $\gamma$ from $\bo{0}$ to $\bo{n_1}$ in $\fb_{0}.$
Now applying Lemma \ref{decomp120}, we obtain a path $\cP=\alpha_1\chi_1\alpha_2\chi_2\ldots$
and by the previous result

\begin{align}\label{inequality}
|\gamma|\ge (1-O(\e_7+\e_6))|\cP|&\ge (1-O(\e_7+\e_6))\sum_{i}|\alpha_i|\ge (1-O(\e_7+\e_6))(1-\e_{11})\frac{n_1}{n}\sum_i |\alpha^{\Sc}_i|,\\
&=(1-O(\e_7+\e_6))(1-\e_{11})\frac{n_1}{n}|\alpha^{\Sc}|,\end{align}
where $|\cP|$ is  computed on $\Fav$ and $|\alpha^{\Sc}|$\footnote{Note that to be completely precise there is an edge joining $\alpha_i^{\Sc}$ and  $\alpha_{i+1}^{\Sc}$ which we are ignoring in \eqref{inequality} for brevity since it is easily seen that such edges only have a negligible contribution.} is computed on any environment in $\Bas.$
Now by definition, $\alpha^{\Sc}$ is a path joining $\bo{0}$ and $\bo{n}(1-O(\e_6))$ in $\fb_4$ and hence on $\Bas$ we have $|\alpha^{\Sc}|\ge (\mu+\zeta)n(1-O(\e_6))$ and thus we are done by choosing $\e_6,\e_7$ and $\e_{11}$ small enough depending on $\e_9$. 
\end{proof}
 
We now prove  Lemma \ref{lb123} using Lemmas \ref{decomp120} and \ref{decomp121} and Proposition \ref{conc1}.

\begin{proof}[Proof of Lemma \ref{lb123}] Recall the set of unstable tiles $A$ in the definition of $\Bas.$
For the proof let us consider an environment $\Pi_1 \in \Bas.$ Let us obtain an altered environment $\Pi_*$ which agrees with $\Pi_1$ on $\Ub_{\fn_4}(j_1,v)$ for any $v\in \llbracket 1, 2^{j_1}\rrbracket^2\setminus A$ and is $b-\e_7$ ($\e_7$ appearing in the definition of the barrier and boosting events.) on the edges in the tiles corresponding to $v\in A.$
Clearly the length of the shortest path between $\bo{0}$ and $\bo{n}$ in $\Pi_*$ is at least $(1-O(\e_7))$ times the length of shortest path between $\bo{0}$ and $\bo{n}$ in $\Pi_1$ since pointwise for any edge $e$, $\Pi_*(e)\ge (1-\frac{\e_7}{b})\Pi_1(e).$  

Fix any $i.$ Note that by Lemma \ref{decomp121} all the $\beta_{i,j}$ are large which means that there is a point (say $a_2$) which is far apart from the end points $a_0$ and $a_1$.  Let $\overset{\rightarrow}{w}^{1}_{i,j}$ and $\overset{\rightarrow}{w}^{2}_{i,j}$ be the vectors obtained by taking the difference of  
$a_2-a_0$ and $a_1-a_2$ respectively.
Also let $\overset{\rightarrow}{w}'_{i,j}$ be the vector obtained by taking the difference of  of the starting and ending points of $\chi_{i,j}$.
Note that by the properties listed in Lemma \ref{decomp121}, it follows that for all $i,j,$ 
\begin{equation}\label{lbnorm}
\min (\|\overset{\rightarrow}{w}^{1}_{i,j}\|_2, \|\overset{\rightarrow}{w}^{2}_{i,j}\|_2, \|\overset{\rightarrow}{w}'_{i,j}\|_2)\ge \e^2_6 \frac{\fn_4}{2^{j_1}}.
\end{equation}
where $\|\cdot\|_2$ denotes the usual Euclidean norm. 
Note that the bound on $ \|\overset{\rightarrow}{w}'_{i,j}\|_2$ follows since $\chi_{i,j}$ is a bridge across the barriers of width $\e_6 \frac{\fn_4}{2^{j_1}}.$ 
Now recall that for any $\Pi \in \Bas,$ and all $v\in \llbracket 1, 2^{j_1}\rrbracket^2\setminus A,$  $\Ub_{\fn_4}(j_1,v)$ is  stable with a certain choice of parameters,
 and moreover recall the approximate norm $\|\cdot\|_{({j_1},v_i)}$ from the statement in Proposition \ref{conc1} where $\alpha_i\in \Ub_{\fn_0}({j_1},v_i)$ in case $v_i\notin A.$ 
Now the following argument is split into two cases:

\nin
$\bo{(1)}$ $v_i \notin A.$ In this case the proof is now complete by  the following string of inequalities ($\e'$ below changes from line to line). 
\begin{align*}
|\alpha_i| &\ge (1-O(\e_7+\e_6))\sum_{j}\left(|\beta_{i,j}|+|\chi_{i,j}|\right) \ge (1-\e')\sum_{j}\left(\norm{\overset{\rightarrow}{w}^{1}_{i,j}}_{({j_1},v_i)}+\norm{\overset{\rightarrow}{w}^{2}_{i,j}}_{({j_1},v_i)}+\norm{\overset{\rightarrow}{w}'_{i,j}}_{({j_1},v_i)}\right)\\
& \ge (1-\e')\norm{\sum_{j}\left(\overset{\rightarrow}{w}^1_{i,j}+\overset{\rightarrow}{w}^2_{i,j}+\overset{\rightarrow}{w}'_{i,j}\right)}_{({j_1},v_i)}\ge (1-\e')\norm{\overset{\rightarrow}{\alpha}_i}_{({j_1},v_i)}= (1-\e')\frac{n_1}{n}\norm{\overset{\rightarrow}{\alpha}^{\Sc}_i}_{({j_1},v_i)} \\
& \ge (1-\e_{10})\frac{n_1}{n}|\alpha^{\Sc}|
\end{align*}
where $\overset{\rightarrow}{\alpha}_i$  (resp. $\overset{\rightarrow}{\alpha}^{\Sc}_i$) is the vector obtained by joining the end points of the path $\alpha_i$ (resp. $\alpha^{\Sc}_i$)  and $\e_{10}$ can be made arbitrarily small by choosing the parameters appropriately. 
The first inequality follows from Lemma \ref{decomp121}. Now note that the second inequality follows from  the definition of the approximate norm $\|\cdot\|_{(j_1,v_i)},$ along with the fact that $\Ub_{\fn_4}(j_1,v_i)$ is  stable and most importantly the lower bound on the euclidean norms of the vectors in \eqref{lbnorm}. The last fact is needed crucially since recall that $(\delta_1,\ell_1,k_1)-\SST$  allows us to relate the passage time to the norm only for pairs of points which are at a distance $\ell_1$ or more apart (see for e.g. \eqref{grad89} and Definition \ref{tilegrad2proj}).  The third inequality is the content of Proposition \ref{conc1} and the fourth inequality again follows from the definition of  $\|\cdot\|_{(j_1,v_i)}.$ Again as above the last inequality relating the passage time to $\|\cdot\|_{(j_1,v_i)}$ follows since $|\overset{\rightarrow}{\alpha}^{\Sc}_i|$ is large enough by choice. 

\nin
$\bo{(2)}$ $v_i \in A.$ In this case clearly $$|\alpha_i|\ge (b-\e_7)\|\overset{\rightarrow}{\alpha}_i\|_1 \ge (b-\e_7)\frac{n_1}{n}\|\overset{\rightarrow}{\alpha}^{\Sc}_i\|_1\ge \frac{n_1}{n} (1-O(\e_7))|\alpha^{\Sc}_i|$$

where the last inequality follows from the discussion at the beginning of the proof and $\|\overset{\rightarrow}{\alpha}_i\|_1$ denotes the $1-$norm of the vector $\overset{\rightarrow}{\alpha}_i.$
\end{proof}

The next three sections prove the three key technical results, Proposition \ref{t:stable}, \ref{conc1}  and \ref{p:cont} regarding stability, approximate convexity of the distance function as well as continuity of the rate function. We start with the continuity result.

\section{Continuity of the rate function} 
\label{s:cont} In this section we prove Proposition \ref{p:cont}. Recall that the statement says 
that for each $\e>0$, there exists $\e'>0$ such that for all $n$ sufficiently large we have 
$$ \frac{\log \P(\sU_{\zeta-\e'}(n))}{n^2} \leq \frac{\log \P(\sU_{\zeta}(n))}{n^2}+\e.$$
This is where the assumption of continuous density of the edge distribution will simplify the proof significantly.
Moreover,
to avoid introducing new notation, we will use several letters in this section which has been used earlier to denote different quantities. However this section will be completely self contained and hence this should not create any confusion or conflict.

The basic approach is simply to start with an environment $\Pi \in \sU^*_{\zeta-\e'}$ and then increase the weight of `all' the edges slightly to construct an environment $\Pi' \in \sU^*_{\zeta}.$
However a technical issue arises since we have assumed the variables are bounded by a constant $b>0$. Hence the variables  in $\Pi$ which are very close to $b$ cannot be increased. 
Thus the first step is to localize the set of such really high valued edges. In fact we will also localize the set of edges which takes values where the density $f_{\nu}$ is close to zero. To carry this out, for any $\e_1,$ let $\e_2$ be such that $\P(X_{e}\in[b-\e_2,b])\le \e_1$ and moreover we will choose $\e_2$ such that  there exists $\e_3>0$ such that $$\inf\{f_{\nu}(x): x\in[b-\e_2,b-\e_2+\e_3]\}\ge \e_3.$$
Now let $\bo{B}=\{x \in [0,b-\e_2]: f_{\nu}(x)\le \frac{\e^3_3}{b}\}.$ Thus by definition $\nu(\bo{B}) \leq \e_3^3.$ 
Now for any $n,$ recall the notation $\fb_4=\Bb(\sC n)$ from \eqref{abbre34}. We will work with the event $\sU^*_{\zeta-\e'}(n)$  which is a function of the edges on $\Bb(4\sC n)$. However recall that on  $\sU^*_{\zeta-\e'}(n)$ any path from $\bo{0}$ which exited $\fb_4$ has length bigger than $bn$ thus it would suffice to increase the value of the  edges only inside $\fb_4.$%

Let $\bo{H}_1=\{e \in \fb_4 :X_e \in [b-\e_2,b]\}.$ 
Now  by a straightforward union bound over all possible choices  of $\bo{H}_1$ (at most $2^{O(n^2)}$), for any $\e_4>0,$
\begin{align}\label{localize23}
\P(|\bo{H}_1|\ge \e_4n^2)\le 2^{O(n^2)}\e_1^{\e_4 n^2}&= e^{O(n^2)+\e_4\log(\frac{1}{\e_1})n^2} \text{and hence, }\\
\nonumber
&=o\left(\P(\sU^*_{\zeta})\right) \text{ for all small enough } \e_1,\\
\nonumber
&= o\left(\P(\sU^*_{\zeta-\e'})\right) \text{ for all }\e'>0. 
\end{align}
Similarly letting $\bo{H}_2=\{e \in \fb_4 :X_e \in \bo{B}\}$ we get \begin{align}\label{localize24}
\P(|\bo{H}_2|\ge \e_4n^2)\le 2^{O(n^2)}\e_3^{3 \e_4 n^2}&= e^{O(n^2)+3 \e_4\log(\frac{1}{\e_3})n^2} \text{and hence, }\\
\nonumber
&=o\left(\P(\sU^*_{\zeta})\right) \text{ for all small enough } \e_3,\\
\nonumber
&= o\left(\P(\sU^*_{\zeta-\e'})\right) \text{ for all }\e'>0. 
\end{align}

The above allows us to localize $\bo{H}_1$ and $\bo{H}_2,$ without paying too much in the probability. 
Formally fix some $\e'>0$ (whose value would be specified later).
Observe that the total number of subsets of $\fb_4$ of size at most $\e_4n^2$ is at most $2^{O(\e_5)n^2}$ where $\e_5=-\left[\e_4\log(\e_4)+(1-\e_4)\log(1-\e_4)\right]$ goes to zero  as $\e_4$ goes to zero. 
From the above discussion it follows that for any $\e_4,$ by choosing $\e_1,\e_3$ small enough we have $$\P(\{|\bo{H}_1|\le \e_4 n^2\}\cap \{|\bo{H}_2|\le \e_4 n^2\} \cap \sU^*_{\zeta-\e'})\ge \P(\sU_{\zeta-\e'})(1-o(1)).$$  Thus by pigeon-hole principle it follows that there exists subset  $A_1,A_2$ of $\fb_4$ each of size at most $\e_4 n^2$ such that
\begin{equation}\label{probbound6754}
\P(\{\bo{H}_1=A_1\} \cap \{\bo{H}_2=A_2\} \cap \sU^*_{\zeta-\e'})\ge \P({\sU_{\zeta-\e'}})e^{-O(\e_5n^2)}.
\end{equation}

For easy referencing let us call the event $\{\bo{H}_1=A_1\} \cap \{\bo{H}_2=A_2\} \cap \sU^*_{\zeta-\e'}$ as $\bo{C}.$
We will also use the following consequence of uniform continuity of $f_{\nu}$ on $[0,b]$, and the fact that $\overline{\bo{D}},$ where $\bo{D}=[0,b]\setminus \{\bo{B}\cup [b-\e_2,b]\}$ is compact and more importantly $f_{\nu}$ is uniformly away from zero (at least $\frac{\e^3_3}{b}$) on the former: Given any $\e_6$ there exists $\e_7$ such that for any $x\in \overline{\bo{D}}$ such that 
\begin{equation}\label{RN}
\frac{1}{1+\e_6}\le \frac{f_{\nu}(x+\e_7)}{f_{\nu}(x)}.
\end{equation}
Now let us modify the event $\bo{C}$ to get an event $\bo{C}_{1}$ which will posses the property that $\log(\P(\bo{C}_1))-\log(\P(\bo{C}))=o(n^2)$   and most importantly $\bo{C}_1\subset \sU_{\zeta}.$
Formally for any $\Pi \in \bo{C}$ noting that by definition $A_1$ and $A_2$ are disjoint,
\begin{align*}
\bo{C}_1(\Pi)=&\{ \Pi': \Pi'(e)=\Pi(e)\,\, \forall e \in A_1,\\
& \Pi'(e)\in [b-\e_2+\frac{\e_3}{2},b-\e_2+\e_3]\,\, \forall e \in A_2,\\ 
& \Pi'(e)=\Pi(e)+\e_7 \,\, \forall e\in \fb_4\setminus A_1\cup A_2 
\}.
\end{align*}
Let 
$\displaystyle{\bo{C}_1=\bigcup_{\Pi\in \bo{C}} \bo{C}_1(\Pi)}.$
We now compute $\P(\bo{C}_1).$ For any $\Pi,$ and subset $B$ of edges in $\fb_{4},$ it would be convenient to let $\Pi|_{B}$ be the restriction of $\Pi$ on the edges in $B$; for any  event $\bo{E}$ let $\bo{E}(B)=\{\Pi|_{B}:\Pi \in \bo{E}\}$; 
 and let $\displaystyle{f_{\nu}(\Pi|_B)):=\prod_{e\in B}}f_{\nu}(\Pi(e))$ (in case $B=\fb_4$ we would omit the above notations.). Thus  
 \begin{equation}\label{prob98534}
 \P(\bo{C})=\int_{\bo{C}} f_{\nu}(\Pi) \,\,\mathrm{d}\Pi \le  \e_3^{3|A_2|} \int_{\bo{C}(\fb_4\setminus A_2)} f_{\nu}(\Pi|_{\fb_4\setminus A_2})\,\,\mathrm{d}\Pi|_{\fb_4\setminus A_2},
 \end{equation}
  where the second inequality follows from the definition of $A_2.$
Note that by definition $\bo{C}_1(\fb_4\setminus A_2)=\bo{C}(\fb_4\setminus A_2)+\bo{v},$
where
$$
\bo{v}(e)=\left\{\begin{array}{cc}
0 & \text{ if } e \in A_1 \\
\e_7 & \text{ if } e \in \fb_4\setminus \{A_1\cup A_2\}
\end{array}\right..
$$ 
Now observe that
\begin{align}
\label{lbc1}
\P(\bo{C}_1)&\ge {\left(\frac{\e_3}{2}\right)^{2|A_2|}} \int_{\bo{C}_1(\fb_4\setminus A_2)} f_{\nu}(\Pi|_{\fb_4\setminus A_2}),\\
\nonumber
&\ge  \left(\frac{\e_3}{2}\right)^{2|A_2|} \left(\frac{1}{1+\e_6}\right)^{|\fb_4|} \int_{\bo{C}(\fb_4\setminus A_2)} f_{\nu}(\Pi|_{\fb_4\setminus A_2})\\
&\ge e^{-O(\e_6)n^2}\P(\bo{C}),
\end{align}
where the first inequality follows from the definition of $\bo{C}_1$, the second inequality is by \eqref{RN} and the final equality is by \eqref{prob98534} by choosing  $\e_3$ and $\e_7$ small enough.  Thus by \eqref{probbound6754},
the proof will now be complete once we show that $\bo{C}_1 \subset \sU_{\zeta}.$
To do this note that for any $\Pi'\in \bo{C}_1$ there exists $\Pi \in \bo{C}$ such that $\Pi'(e)\ge \Pi(e)+\min (\frac{\e_3}{2},\e_7)$ for all $\e \in \fb_4 \setminus A_1$ and $\Pi'(e)=\Pi(e)$ for $e\in A_1.$ 

Note that since $\Pi \in \sU^{*}_{\zeta-\e'}$,  any path $\sG$ starting from the origin, which exits $\fb_4$ has weight at least $bn$ in $\Pi$ and hence by the above discussion also in $\Pi'$.

Thus to prove the lemma we only consider the path 
 $\sG$ which is the shortest path between $\bo{0}$ and $\mathbf{n}$ lying inside $\fb_4$, in the environment $\Pi'.$ We want to show $|\sG|_{\Pi'}\ge (\mu+\zeta)n,$ where $|\sG|_{\Pi},|\sG|_{\Pi'}$ denote the weights of $\sG$ in the environments $\Pi$ and $\Pi'$ respectively.

Now since trivially $|\sG|_{\Pi'}\ge |\sG|_{\Pi},$ there is nothing to show if $|\sG|_{\Pi}>(\mu+\zeta)n.$ Assuming otherwise, it follows that $$|\sG\cap A_1|\le \frac{(\mu+\zeta)n}{b-\e_2}=cn$$ 
for some $c<1$ for all $\e_2$ small enough since $\mu+\zeta <b$ (by $\sG\cap A_1$ we denote the set of edges in $A_1$ that $\sG$ passes through). Indeed, this is true since each edge in $A_1$ has weight at least $b-\e_2.$ However note that since $\sG$ connects $\bo{0}$ and $\bo{n},$ trivially $\sG$ passes through at least $n$ edges.  Thus $|\sG\cap A_1^c|\ge (1-c)n$ and hence $|\sG|_{\Pi'}-|\sG|_{\Pi}\ge {(1-c)}n\min (\frac{\e_3}{2},\e_7)$.  By definition $|\sG|_{\Pi} \ge (\mu+\zeta-\e')n$ and hence taking $\e'=\frac{\min(\frac{\e_3}{2},\e_7){(1-c)}}{2}$  implies the sought bound $|\sG|_{\Pi'}\ge (\mu+\zeta)n.$
Thus to finish the proof phrased in terms of the parameters in the statement of Proposition \ref{p:cont}, given $\e$ we must choose $\e_4$ small enough so that in \eqref{probbound6754} the $O(\e_5)$ term is at most  $\e.$ This dictates the choice of $\e_1$ and $\e_3,$ which in turn dictates the choice of $\e_2$.  The choice of $\e_6$ and hence $\e_7$ is governed by \eqref{lbc1} which then fixes the value of $\e'$. 
\qed
\section{Approximate convexity properties}\label{pconc1}
In this section we will prove Proposition \ref{conc1}, i.e.  given $\delta_1,m_1$ and $j_1$, for any $\Ub_{\sC n}({j_1},v)$ which is $(\delta_1,\ell_1,k_1)-\SST$ where $\ell_1=\frac{n}{2^{j_1+m_1}}$ and $k=2^{2m_1}$,  and any set of vectors $\bo{w}_1,\bo{w}_2,\ldots,\bo{w}_t$,  if $\bo{w}=\sum_{i=1}^t \bo{w}_i$ then   
\begin{equation}\label{dis9876}
\|\bo{w}\|_{(j_1,v)}\le (1+\delta)\sum_{i=1}^t\|\bo{w}_i\|_{(j_1,v)}
\end{equation}
 where $\delta= O(\delta_1+{2^{-\frac{m_1}{16}}}).$
The proof essentially follows by noticing that any set of vectors as above can be scaled down to get a sum of vectors inside $\Ub_{\sC n}(j_1,v)$ followed by application of  stability and triangle inequality.
To formalize this, we need some notation:  For every $\phi\in \mathbb{S}^1(\eta_1),$ (value of $\eta_1$ will be specified later and sufficiently small)
$$
\cC(\phi)=\{\bo{b} \in \{\bo{w}_1, \bo{w}_2,\ldots, \bo{w}_t\} : {\rm{arg}}(\bo{b}) \in [\phi,\phi+\eta_1) \}
$$
where $\bo{w}_1,\bo{w}_2,\ldots \bo{w}_t$ are as in the statement of the proposition, i.e. $\cC(\phi)$ denotes the collection of vectors among $\{\bo{w}_1, \bo{w}_2,\ldots, \bo{w}_t\}$ whose angle with the $x-$axis falls in the interval $[\phi,\phi+\eta_1).$
Let $\bo{w}_{\phi}= \sum_{i=1}^t \bo{w}_i \mathbf{1}(\bo{w}_i \in \cC(\phi))$ be the sum of the vectors in $\cC(\phi).$ Thus by definition $$\sum_{\phi \in \mathbb{S}^{1}(\eta_1)}\bo{w}_{\phi}=\bo{w}.$$
Also note that for every $\bo{w}_i \in \cC(\phi),$  by Lemma \ref{smoothgradproj}

\begin{align}
\label{dis90678}
\|\bo{w}_i\|_{(j_1,v)}&=\|\bo{w}_i\|_2\left(1+O(\delta_1+\eta_1+2^{-m_1/4})\right)\grad_{\Pro}((j_1,v),\phi) \text{ and hence,}\\
\label{app453}
\sum_{\bo{w}_i\in \cC(\phi)}\|\bo{w}_i\|_{(j_1,v)}&=A_{\phi}\left(1+O(\delta_1+\eta_1+2^{-m_1/4})\right) \grad_{\Pro}((j_1,v),\phi)
\end{align}
where $A_{\phi}=\sum_{i=1}^t\|\bo{w}_i\|_2\mathbf{1}(\bo{w}_i \in \cC(\phi)).$
In the sequel for brevity we will denote the term 
\begin{equation}\label{app4512}
\left(1+O(\delta_1+\eta_1+2^{-m_1/4})\right) \grad_{\Pro}((j_1,v),\phi)
\end{equation}
 in \eqref{app453}  by $\tilde \grad_{\Pro}((j_1,v),\phi).$
Now since the angle made by each $\bo{w}_i\in \cC(\phi)$ lies in the interval $[\phi, \phi+\eta_1),$ for all small $\eta_1$ it also follows that $(1-\eta_1^2)A_{\phi}<\|\bo{w}_{\phi}\|_2\le A_{\phi}.$
Thus from the above two expressions it follows that 

\begin{equation}\label{normbd564}
(1-\eta_1^2)A_{\phi}\tilde \grad_{\Pro}((j_1,v),\phi)\le \|\bo{w}_{\phi}\|_{(j_1,v)}\le A_{\phi}\tilde \grad_{\Pro}((j_1,v),\phi).
\end{equation}

For each $\phi$ let us consider the value ${b}_{\phi}=\frac{A_{\phi}}{\|\bw\|_2}$. 
Now without loss of generality we can assume that there exists a universal constant $C$ such that ${b}_{\phi}\le C$ for all $\phi$ since otherwise we would be done using \eqref{dis90678} and the fact that $ \grad_{\Pro}((j_1,v),\phi)$ is bounded away from zero and infinity for any $\phi.$
 We now define the set $
\cB=\{\phi\in \mathbb{S}^1(\eta_1): b_{\phi}\le \e_1\}
$
( the value of $\e_1$ is specified later)
and hence 
\begin{equation}\label{ignoresmall}
\|\sum_{\phi \in \cB}\bo{w}_{\phi}\|_2 \le \frac{2\pi \e_1}{\eta_1} \|\bo{w}\|_2.
\end{equation}
Now let $$\bo{c}_{\phi}=\frac{\bo{w}_{\phi}}{\|\bo{w}\|_2}\frac{\sC n}{100\times 2^{j_1}}, \text{ and let } \bo{c}=\frac{\bo{w}}{\|\bo{w}\|_2}\frac{\sC n}{100\times  2^{j_1}}.$$
Thus we have rescaled $\bo{w}$ to get a vector $\bo{c}$ of length $\frac{\sC n}{100 \times 2^{j_1}}$ and scaled all the $\bo{w}_{\phi}$'s by the same factor to obtain the $\bo{c}_{\phi}$'s.
For convenience let $\mathbb{S}^{1}(\eta_1)\setminus \cB=\{\phi_1,\phi_2,\ldots,\}$
We now consider the sequence of points $\bo{v}_0, \bo{v}_1,\bo{v}_2,\ldots$ such that $\bo{v}_{i}-\bo{v}_{i-1}=\bo{c}_{\phi_i}$ and let $\bo{v}_0$ be for concreteness the center point of $\Ub_{\sC n}(j_1,v).$

Now consider the path $\cP$  (recall from \eqref{abbre34}) obtained by concatenation of paths $\cP_1,\cP_2,\cP_3,\ldots$ where $\cP_i$ is the shortest path  between $\bo{v}_{i-1}$ and $\bo{v}_i$. 
 
At this point we make another assumption that none of the points $\bo{v}_i$  is outside  $\Ub_{\sC n}(j_1,v)$ \footnote{This is  not essential for the proof but is done for convenience. Note that if this assumption is not satisfied this can always be achieved by chopping our vectors $\bo{c}_{\phi_j}$ in to smaller vectors and rearranging the order of the sum $\sum_{j=1}^{i-1}\bo{c}_{\phi_j}$.}. 
Now assuming  that $\Ub_{\sC n}(j_1,v)$ is $(\delta_1,\ell_1,k_1)-\SST$ as in the hypothesis of the proposition, it follows from Lemma \ref{smoothgradproj} that $$|\cP_i|\le  \|\bo{c}_{\phi_i}\|_2 \tilde \grad_{\Pro}((j_1,v),\phi_i),$$ provided that $\e_1 \gtrsim 2^{-\frac{m_1}{4}}$ since by hypothesis as every $\phi_{i} \notin \cB$, $\|\bo{c}_{\phi_i}\|_2\ge \frac{\e_1\sC n}{100 \times 2^{j_1}}$.
Thus 
\begin{equation}\label{eq221}
|\cP|=\sum_{i}|\cP_i|\le (1+O(\delta_1+\eta_1+2^{-\frac{m_1}{4}}))\sum_{i}\|\bo{c}_{\phi_i}\|_{(j_1,v)}.
\end{equation}
Now as $\cP$ is  a path (not necessarily the shortest) joining $\bo{v}_0$ and $\bo{v}_0+\bo{c}_*$ where $\bo{c}_*=\sum_{i}\bo{c}_{\phi_i}$, by  stability we have 
\begin{equation}\label{eq222}
|\cP|\ge (1-O(\delta_1+\eta_1+2^{-\frac{m_1}{4}}))\|\bo{c}_*\|_{(j_1,v)}.
\end{equation}
However note that by \eqref{ignoresmall}, it follows that $\|\bo{c}-\bo{c}_*\|_2 \le O(\frac{\e_1}{\eta_1}) \|\bo{c}\|_2 $ and hence by Lemma \ref{smoothgradproj} 
\begin{equation}\label{eq223}
\|\bo{c}_*\|_{(j_1,v)}\ge \left(1-O(\delta_1+\frac{\e_1}{\eta_1}+2^{-\frac{m_1}{4}})\right) \|\bo{c}\|_{(j_1,v)}.
\end{equation}

Putting the above together (letting $1+B=1+O(\delta_1+\eta_1+2^{-m_1/4})$ appearing in \eqref{app4512}) it follows that
\begin{align*}
\sum_{i=1}^t \| \bo{w}_i\|_{(j_1,v)}&\overset{\eqref{normbd564}}{\ge} (1-\eta_1^2) (1-B) \sum_{\phi \in \bS^{1}(\eta_1)}\| \bo{w}_{\phi}\|_{(j_1,v)} \ge (1-\eta_1^2)(1-B)\sum_{\phi \in \bS^{1}(\eta_1)\setminus \cB}\| \bo{w}_{\phi}\|_{(j_1,v)},\\
&=(1-\eta_1^2)(1-B) \sum_{\phi \in \bS^1(\eta_1)\setminus \cB}\| \bo{c}_{\phi}\|_{(j_1,v)}\frac{100 \times 2^{j_1}\|\bo{w}\|_2}{{\sC n}},\\ 
&\ge (1-\eta_1^2)(1-B)^2(1-B-O(\frac{\e_1}{\eta_1})) \|\bo{c}\|_{(j_1,v)}\frac{100\times 2^{j_1}\|w\|_2}{{\sC n}} \ge (1-O(\delta_1+2^{-\frac{m_1}{16}}))\|\bo{w}\|_{(j_1,v)}
\end{align*}
where the second to last inequality follows from \eqref{eq221}, \eqref{eq222} and \eqref{eq223} and the final inequality follows by choosing $\eta_1=2^{-m_1/8}$ and $\e_1\gtrsim 2^{-m_1/4}$ ensuring $\frac{\e_1}{\eta_1}=O(2^{-\frac{m_1}{16}}).$

\section{Stability of the gradient}
\label{s:sproof}

This section is devoted to proving Proposition \ref{t:stable}. It turns out that this property has little to do with the specific details of the first passage percolation metric, rather it is a property of general distance functions on $\R^2$ that are comparable to the Euclidean metric  
i.e., it satisfies triangle inequality and  that for all $x,y\in \R^2$ such that $|x-y|$  is large enough (possibly $n$ dependent)
\begin{equation}
\label{e:comp2}
\alpha |\bx-\by|\leq \PT(\bx,\by) \leq 3b |\bx-\by|,
\end{equation}
for some $\alpha>0$. For the ease of reading we recall the statement of the proposition and state it as a theorem to highlight the fact that its generality makes it potentially applicable in other problems of metric geometry.  Recall our terminology that $\bz\in \R^2$ is $(\delta, \bS^1(\eta),\ell,k)-\SST$  if $\bz$ is $(\delta,\theta,\ell,k)-\SST$ for each $\theta \in \bS^1(\eta)$. 

\begin{thm}
\label{t:stablegen}
Fix $\delta,\e,\eta >0,$ and $k\in \N$ and $J_1\in \N$. There exists $J_2\in \N$  such that for all large enough $n$: for $\Pi \in \sU^*_{\zeta}(n)$ (note that \eqref{e:comp2} holds for all $\bx,\by \in \Bb(10n)$ such that $|\bx-\by|\ge \sqrt n$) there exists $J_1\leq j\leq J_2$   such that 
$$\#\{\bz\in \Lb(n): \bz~\text{is not}~(\delta, \bS^1(\eta),\frac{ n}{2^j},k)-\SST\} \leq \e n^2 .$$
\end{thm}

Above we have replaced $\sC n$ in the statement of the proposition by $n$ for notational brevity since as the reader will notice the arguments do not depend on the exact value in any way. Moreover from now on without explicitly stating it,  we will assume that  \eqref{e:comp2} holds for all pairs of points $\bx, \by \in \Bb(10 n)$ where $|\bx-\by|\ge \sqrt n$ even though we will not explicitly mention the last qualification every time since it will be trivially satisfied in our applications.

\subsection{A roadmap of the proof}
As we are not shooting for optimal bounds the proofs will often rely on several  crude averaging arguments and applications of the pigeon hole principle along with the bi-Lipschitz nature of the FPP metric. However, there are many technical  steps involved and for the sake of exposition we give a brief overview of the argument at this point. Our argument relies on the following observations. 

\nin
(1) Fix $\bz\in \Bb(n)$ and $\theta\in \bS^1(\eta)$. Observe that for all $J_2>J_1,$
\begin{align}\label{tria786}
 \PT(\bz,\theta,\frac{n}{2^{J_2}},2^{J_2})-\PT(\bz,\theta, \frac{n}{2^{J_1}},2^{J_1})&= \sum_{j=J_1}^{J_2-1}\left[\PT(\bz,\theta,\frac{n}{2^{j+1}},2^{j+1})-\PT(\bz,\theta,\frac{n}{2^j},2^{j})\right].
\end{align}
The LHS in \eqref{tria786} is bounded by $3bn$ and all the terms in the RHS in \eqref{tria786} are positive by triangle inequality (as in Lemma \ref{monotone}). 

\nin
(2)
Thus if $J_2-J_1\ge \frac{1}{\e},$ then by the pigeon-hole principle  there must exist one $J_1\le j\le J_2$ such that $\left[\PT(\bz,\theta,\frac{n}{2^{j+1}},2^{j+1})-\PT(\bz,\theta,\frac{n}{2^j},2^{j})\right]\le O(\e) n$. {As a matter of fact we should find consecutive many such $j$ if $J_2-J_1\gg \frac{1}{\e}$.}

\nin
(3) Now for $j$ as in (2) consider the discrete segments 
\begin{align*}
\sS(\bz,\theta,\frac{n}{2^{j}},2^{j})=[\bz_0,\bz_1,\ldots \bz_{2^{j}}] \text{ and, }\sS(\bz,\theta,\frac{n}{2^{j+1}},2^{j+1})=[\bz_0, \bz_{0,1},\bz_1, \bz_{1,2},\bz_2,\ldots ,\bz_{2^{j}}], 
\end{align*}
where $\bz_{i,i+1}$ is the mid-point of the line segment joining $\bz_i$ and $\bz_{i+1}.$ Thus the above observation together with the lower bound in \eqref{e:comp2} suggests that for most $i$, 
$$\PT(\bz_i,\bz_{i,i+1})+\PT(\bz_{i,i+1},\bz_{i+1})\le (1+O(\e))\PT(\bz_i,\bz_{i+1}).$$ 
However this is not quite enough to establish  stability and in fact we need something along the lines of the following stronger fact (see Lemma \ref{crudelem43}):   for most $i,$
$$\PT(\bz_i,\bz_{i,i+1})\approx\PT(\bz_{i,i+1},\bz_{i+1})\approx \frac{1}{2}\PT(\bz_i,\bz_{i+1}).$$ 

\nin 
(4)
Suppose the contrary and without loss of generality assume that 
$$\PT(\bz_i,\bz_{i,i+1})\ge  (\frac{1}{2}+\delta) \PT(\bz_i,\bz_{i+1}).$$ 
The contradiction will come from the fact that the above cannot be true for many consecutive scales.  Indeed, if it was true for $j'$ many consecutive scales, then recursively picking one half of an interval at each scale in which the above inequality holds leads to an interval $[\bw_1, \bw_2]$ such that $\|\bw_1-\bw_2\|_2=\frac{\|\bz_i-\bz_{i+1}\|_2}{2^{j'}}$ but 
$$\PT(\bw_1,\bw_2)\ge (1+2\delta)^{j'}\frac{\PT(\bz_i,\bz_{i+1})}{2^{j'}}.$$
Clearly for $j'$ large enough (depending on $\delta$) this contradicts the upper bound in \eqref{e:comp2}. 
We now move towards making the above formal.

Recalling the notion of  stability from \eqref{stabnot32}, the following crude lemma will be useful to show the latter.

\begin{lem}\label{crudelem43} Given $\delta>0$, $ \theta \in \bS^1$ and $\ell, k \in \N$.  Recalling that $\sS(\bz,\theta,\ell,k)=[\bz=\bz_0,\bz_1,\ldots,\bz_k],$ suppose  
\begin{align*}
\sup_{0 \le i,j \le k-1} \frac{\PT(\bz_i,\bz_{i+1})}{\PT(\bz_j,\bz_{j+1})} \le 1+\delta,  ~~~\text{and }\quad 
\frac{k\PT(\bz_0,\bz_{1})}{1+\delta}\le {\PT(\bz_0,\bz_{k})}\le {k\PT(\bz_0,\bz_{1})}(1+\delta). 
\end{align*} 
Then for each $i\leq k$, and $k'\le k-i,$ $\bz_i$ is $(\delta',\theta,\ell,i+k')-\SST$ where $\delta'=O(\delta k)$.
\end{lem}
\begin{proof} By hypothesis
\begin{align*}
\frac{1}{(1+\delta)}k\PT(\bz_{0},\bz_1)\le  \PT(\bz_0,\bz_k)\le {(1+\delta)}k\PT(\bz_{0},\bz_1).
\end{align*}
Now as for any $i\le k$ and $k'\le k-i$ we have  $\PT(\bz_0,\bz_k)\le \PT(\bz_i,\bz_{i+k'})+(k-k')(1+\delta)\PT(\bz_{0},\bz_1).$ Thus it follows that  
\begin{align*}
 \PT(\bz_i,\bz_{i+k'})&\ge\frac{1}{(1+\delta)}k\PT(\bz_{0},\bz_1)-(k-k')(1+\delta)\PT(\bz_{0},\bz_1)\ge  k'\PT(\bz_{0},\bz_1)(1-O(\delta k)), \\
&\ge k'\PT(\bz_{i},\bz_{i+1})(1-O(\delta k)).
\end{align*}
Moreover note that by triangle inequality and the hypothesis,
 $\PT(\bz_i,\bz_{i+k'})\le (1+\delta)k'\PT(\bz_i,\bz_{i+1}).$
\end{proof}
Thus in the sequel to prove   stability we will only prove that the hypothesis of Lemma \ref{crudelem43} is satisfied. 
Going back to the proof of Theorem \ref{t:stablegen} following 
 the line of argument in the roadmap above, one can deduce the existence of many  stable points along a fixed line in a given direction. Further arguments are then necessary to strengthen this to get the full result. We shall first state and prove the weaker version.

\subsection{Stability on a fixed line}

For the weaker version let us consider the discrete segment $\sS(\bz,\theta,\frac{n}{2^j},2^j)$ for some $\bz\in \Bb(n)$ and $\theta\in \bS^1(\eta)$. We shall show most points on this segments are stable for $k$ consecutive intervals.

\begin{lem}
\label{l:weak}
Let $\bz\in \Bb(n), \theta \in \bS^1(\eta), k\in \N,$ and let $\delta_2>0$ and $J_1\in \N$ be fixed. Then there exists $\fm$ such that for all small enough $\delta_3$ the following holds:  there exists $j\in \N$ with $  J_1\le j \le (J_1+ \frac{1}{\delta_3^2})$ for which all but $O(\delta_2)$ fraction of the points $\bz_i$ in the discrete segment $\sS(\bz,\theta,\frac{n}{2^{j\fm }},2^{j\fm})$  are $(\delta_2,\theta,\frac{n}{2^{j\fm}},k)-\SST$.
\end{lem}
The quantification in the above statement might be a little hard to parse, but it will create some simplification in the notational choices later. 

For the moment let us fix a value of $\fm$ to be specified later. To make formal the outline described in the subsection it will be convenient to associate trees to the the intervals in $\sS(\bz,\theta,\frac{n}{2^{\fm J_1}},2^{\fm J_1})=[\bz=\bz_0, \bz_1, \ldots, \bz_{2^{\fm J_1}}]$. Let 
\begin{equation}\label{treecon12}
\cT_1,\cT_2.\ldots, \cT_{2^{\fm J_1}}
\end{equation}
be complete $2^{\fm}-$ary trees of depth $J_2-J_1$ where the value of $J_2$ will be specified to be a  large enough number later (for convenience we shall index the levels of these trees by $j=J_1, J_1+1\ldots , J_2$).  Let $L^{(i)}_j$ denote the vertices at the $j^{th}$ level of $\cT_{i}$ and let $L_j=\cup_{i}L^{(i)}_j,$ denote the union of the vertices at the $j^{th}$ level. 
We will identify $\cT_{i}$ with the interval $[\bz_{i-1},\bz_{i}]$. Now  for any $J_1\le j\le J_2,$ consider the discrete segment $$\sS(\bz,\theta, \frac{n}{2^{j\fm}},2^{j\fm})=[\bz^*_{0},\ldots \bz^*_{2^{(j-J_1)\fm}},\bz^*_{2^{(j-J_1)\fm}+1},\ldots, \bz^*_{2^{(j+1-J_1)\fm}},\ldots \bz^*_{2^{(j-1)\fm}}\ldots \bz^*_{2^{j\fm}}].$$

Naturally $[\bz^*_{0},\ldots ,\bz^*_{2^{(j-J_1)\fm}}]$ is a discretization of the interval $[\bz_0,\bz_1]$ and hence can be associated to $L^{(1)}_{j}$ where each vertex in $L^{(1)}_{j}$ corresponds to $[\bz^*_h,\bz^*_{h+1}]$ in the natural order. (e.g.\ the root of $\cT_1$ corresponds to the interval $[\bz_{0},\bz_1]$). The same correspondence holds for the other intervals and trees. See Figure \ref{fig12}.

\begin{figure}[h]
\centering
\includegraphics[scale=.7]{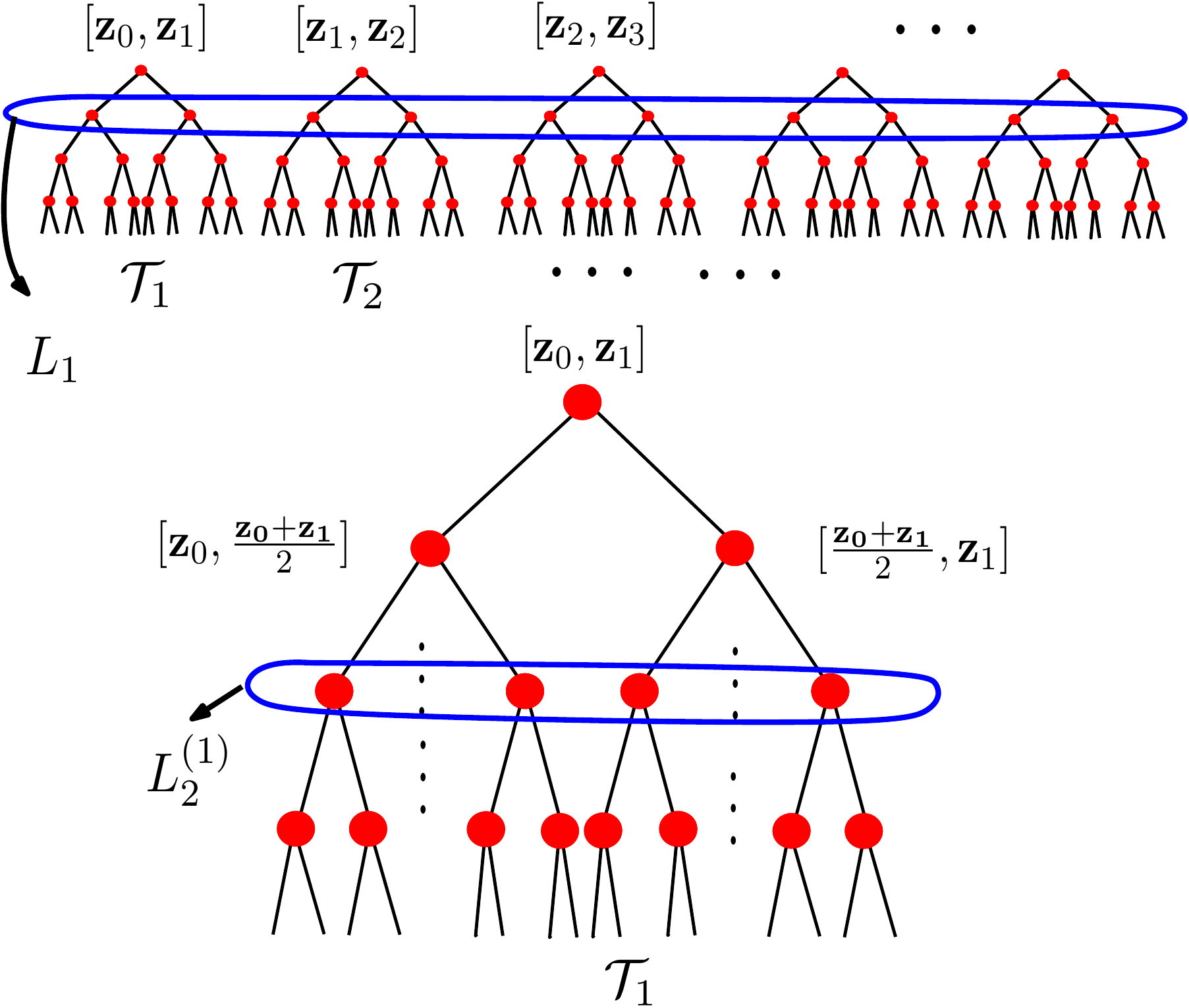}
\caption{This figure illustrates the various definitions introduced in this section related to the trees in \eqref{treecon12} in the toy case $\fm=1$ where the trees are binary trees.}
\label{fig12}
\end{figure}
Now for any  vertex $v$ in any of the trees let $Y_v:=\PT(\bz',\bz'')$ where $[\bz',\bz'']$ is the discrete segment associated to the vertex $v$. 
We need some further notation: let $U_{i,j}=\sum_{v\in L^{(i)}_j}Y_v$
and let $U_j=\sum_{i} U_{i,j}$. It will be convenient to frame our arguments using pigeon hole principle as applications of `the probabilistic method', 
and hence we define a set of random variables. For any $j$ pick uniformly any edge $e_{j+1}$ at the $(j+1)^{th}$ level across all the trees, i.e., connecting $L_{j}$ and $L_{j+1}$ and let 
\begin{equation}\label{edge23}
X_{e_{j+1}}:=\frac{2^{\fm}Y_{{w}}}{Y_{v}}
\end{equation}
 where $e_{j+1}=(v,w)$ and $v$ is closer to the root. For brevity we will identify the set of such edges with the set $L_{j+1}$ using the natural correspondence. 
Now by triangle inequality (again, as in Lemma \ref{monotone})
\begin{equation}\label{triangle}
\E(X_{e_{j+1}}\mid Y_{v})\ge 1.
\end{equation} 
 However the distributions of $X_{e_j}$ across various $j$ will not be independent and the joint distribution can be defined in the following way: pick uniformly a vertex among all the  leaf vertices across all the trees (note that it is naturally and uniquely associated with a uniformly chosen simple path from the root to the leaf in an uniformly chosen tree)  and label the edges on the path as 
$(e_{J_1+1},e_{J_1+2},\ldots , e_{J_2})$ where $e_i$ denotes the intersection of the path with the $i^{th}$ level.  
It is clear that $e_j$ is uniformly distributed among all edges connecting $L_{j-1}$ and $L_j$. 
 Notice that 
$$\prod_{j=J_1+1}^{J_2}X_{e_j}\overset{d}{=}\frac{2^{(J_2-J_1)\fm}Y_{w}}{Y_{v}}$$ where $Y_v$  and $Y_w$ are the variables attached to the root of a randomly chosen tree $\cT_i$ and a randomly chosen leaf of $L^{(i)}_{J_2-J_1}$ respectively.
We now bound the expectation of $X_{e_j}$ for all $J_1< j\le J_2.$
To do this consider the ratio $\frac{U_{j+1}}{U_j}.$ By definition, we have the following:
\begin{align}\label{expreps}
\frac{U_{j+1}}{U_j}=\frac{\sum_{w\in L_{j+1}}Y_w}{\sum_{v\in L_{j}}Y_v}=\frac{\sum_{v \in L_j}Y_v\E(X_{e_{j+1}}\mid Y_v)}{\sum_{v\in L_j}Y_v}\ge 1
\end{align}
where the second equality follows from \eqref{edge23} and the fact that the trees $\cT_i$ are $2^{\fm}-$ary and the final inequality follows from \eqref{triangle}. The following lemma completes the proof of Lemma \ref{l:weak} under a further assumption that on a sufficiently large interval contained in $[J_1, J_1+\frac{1}{\delta_3^2}]$ the LHS above is also upper bounded by $1+\delta_3$ for some small enough $\delta_3$. 

\begin{lem}
\label{l:conditional}
Fix $c>0$. In the setting of Lemma \ref{l:weak}, suppose there exists an interval $I\subseteq [J_1, J_1+\frac{1}{\delta_3^2}]$ with $|I|\geq \frac{c}{\delta_3}$ such that for all $j\in I$ for some small enough $\delta_3$ depending on $\fm$ and $\delta_2$
\begin{align}\label{smallgap}
0\le \frac{U_{j+1}}{U_j}-1 \le \delta_3.
\end{align}
Then the conclusion of Lemma \ref{l:weak} holds.
\end{lem}

\begin{proof}
Without loss of generality for this proof we shall write $I=[J_1,J_2]$ where $I$ is given by the hypothesis. It is a consequence of \eqref{e:comp2}  that for $j\le k$ and $v\in L_j$  and $w \in L_k,$ 
\begin{equation}\label{comparable}
\frac{1}{C}\le \frac{Y_v}{2^{(k-j)\fm}Y_w}\le C
\end{equation}
for some universal constant $C=C(b,\alpha)>1$. Define now the probability measure $\mu_j$ on the $j^{th}$ level vertices given by $\mu_j(v)=\frac{Y_v}{\sum_{v\in L_j} Y_v}$. In particular, \eqref{comparable} implies the Radon-Nikodym derivative of $\mu_j$ with respect to the uniform measure $\mathfrak{u}$ is bounded above and below by $C$ and $C^{-1}$ respectively. Now, \eqref{smallgap}, along with \eqref{expreps}, implies $\E_{{\mu_j}}(\E(X_{e_{j+1}}\mid v)-1)\leq \delta_3$. This, together with the above observation, and the fact $\E(X_{e_{j+1}}\mid v)-1)>0$ implies that 

\begin{align}\label{uniexp}
\E_{\mathfrak{u}}(\E(X_{e_{j+1}}\mid v)-1)&\le C \delta_3.
\end{align} 

By \eqref{e:comp2}, $C$ in \eqref{comparable} can be chosen such that such that deterministically $\frac{1}{C}\le X_{e_j}\le C$ and moreover,
\begin{align}\label{log}
\frac{1}{C}\le \prod_{j=J_1+1}^{J_2}X_{e_j}\le C, \text{ which implies, } \left|\sum_{j=J_1+1}^{J_2}\E_\mathfrak{u}(\log X_{e_j})\right|& \leq \log C.
\end{align}
Thus  it follows that there exists $J_1+1\le j\le J_2$ such that $\E_{\mathfrak{u}}(\log X_{e_j})\ge -c^{-1}(\log C)\delta_3$. Hence  we have found a $J_1+1\le j\le J_2$ with the following two properties:
\begin{align*}
1\le \E_{\mathfrak{u}}(X_{e_j})\overset{\eqref{uniexp}}{\le} 1+ C\delta_3 \text{ and }\E_{\mathfrak{u}}(\log(X_{e_j}))\ge-c^{-1}(\log C)\delta_3.
\end{align*}
Now for any edge $e,$ denoting $X_e-1=y_e,$  the above can be restated as 
\begin{align*}
0\le \frac{1}{2^{j\fm}}\sum_{e\in L_j}y_e \le C\delta_3 \text{ and } \frac{1}{2^{j\fm}}\sum_{e\in L_j}\log (1+y_e) \ge -c^{-1}(\log C)\delta_3.
\end{align*}

Now note that by \eqref{comparable}, $y_e\le C$  and hence using Taylor expansion,
$\log (1+y_e)\le y_e -C'y_e^2$ for some universal constant $C'$.
Using the above inequalities it follows that 
$$\E_{\mathfrak{u}}(X_{e_j}-1)^2=\frac{1}{2^{j\fm}}\sum_{e \in L_j } y^2_{e}=O(\delta_3).$$
Thus by Chebyshev inequality, for at least $1-O(\sqrt{\delta_3})$ fraction of $e\in L_j$, we have $|X_{e}-1|\le  \delta_3^{1/4}$. Let us call such an edge $e$,  a good edge. 
Now let us consider all $v\in L_{j-1}$ such that all the children of $v$ are good (let us call such $v$ good).  A naive bound shows that the fraction of good  $v$ is at least $1-O(2^{\fm}\sqrt \delta_3).$
Now for any good $v$ corresponding to an interval $[\bw_1,\bw_2]$ say, if the discrete segment $[\bw_1=\bw^*_0, \bw^*_1, \ldots, \bw^*_{2^{\fm}}=\bw_2]$ corresponds to the $2^{\fm}$ children then Lemma \ref{crudelem43} implies the following: each $\bw^*_i$ for $i\in \llbracket 0,2^{\fm}-k\rrbracket$ is $(\delta',\theta,\frac{n}{2^{j\fm}},k)-\SST,$ where $\delta'=O(2^\fm \delta_3^{1/4}).$ Thus the total fraction of points on $\sS[\bz,\theta,\frac{n}{2^{j\fm}},2^{j\fm}]$ that are not $(\delta',\theta,\frac{n}{2^{j\fm}},k)-\SST$ is at most $O(\frac{k}{2^{\fm}}+2^{\fm}\sqrt \delta_3).$
Now choose $\fm$ large enough and then $\delta_3$ small enough  such that $\max (\frac{k}{2^{\fm}}+2^{\fm}\sqrt \delta_3, \delta')\le \delta_2$.
\end{proof}

It remains to prove that \eqref{smallgap} holds for a number of consecutive scales. This is ensured by the following lemma using another pigeon hole argument. 

\begin{lem}
\label{l:consec}
In the setting of Lemma \ref{l:conditional}, there exists $c>0$, and $I\subseteq [J_1, J_1+\frac{1}{\delta_3^2}]$ with $|I|\geq \frac{c}{\delta_3}$ such that for all $j\in I$
$$0\le \frac{U_{j+1}}{U_j}-1 \le \delta_3. $$
\end{lem} 

\begin{proof}
For $c>0$ to be specified later, we divide the $\frac{1}{\delta_3^2}$ many scales into consecutive blocks of $\frac{c}{\delta_3}$ many scales each. For $i\in \llbracket 1, \frac{1}{c\delta_3}\rrbracket$, 
Let $a_i=U_{J_1+\frac{ic}{\delta_3}}-U_{J_1+\frac{(i-1)c}{\delta_3}}$. By the triangle inequality, 
$a_{i}\ge 0$ for all $i$, and by \eqref{comparable}, there exists a universal constant $C$ such that $U_{J_1+\frac{1}{\delta_3^2}}\le C U_{J_1}.$ 
As a consequence, $\sum_{i} a_i \leq CU_{J_1}$
and by choosing $c$ sufficiently small it follows there exists some $i\in \llbracket 1,\frac{1}{c\delta_3} \rrbracket$ such that $a_i \le U_{J_1}\delta_3.$ 
Now this implies that for any $J_1+\frac{(i-1)c}{\delta_3}\le j\le J_1+\frac{ic}{\delta_3}$ we have $\frac{U_{j+1}-U_j}{U_j}\le \frac{a_i}{U_{J_1}}\leq \delta_3$; completing the proof. 
\end{proof}

\subsection{Strengthening Lemma \ref{l:weak} to Theorem \ref{t:stablegen}}
We now provide the extra ingredients needed to  extend the argument of the previous subsection to establish the stronger statement of Theorem \ref{t:stablegen}. To avoid repetition, often instead of providing the full formal proof we shall describe the main ideas and present an elaborate sketch. Observe that to establish Theorem \ref{t:stablegen}, one needs to extend Lemma \ref{l:weak} in the following two directions:
 
\begin{enumerate}
\item[(a)] Get the stability at a point simultaneously at all directions in $\bS^1(\eta)$ at the same scale $j$. 
\item[(b)] Deducing stability of most lattice points from stability of points on a discrete segment (which are not necessarily lattice points).
\end{enumerate} 
 
We describe below how to take care of these two items. 
To address the issue in (a) note that one cannot naively apply the above argument separately for all $\theta\in \mathbb{S}^1(\eta)$ since a priori one might not end up with the same scale $j$ for all $\theta\in \mathbb{S}^1(\eta)$. 
\begin{figure}[h]
\centering
\includegraphics[scale=.5]{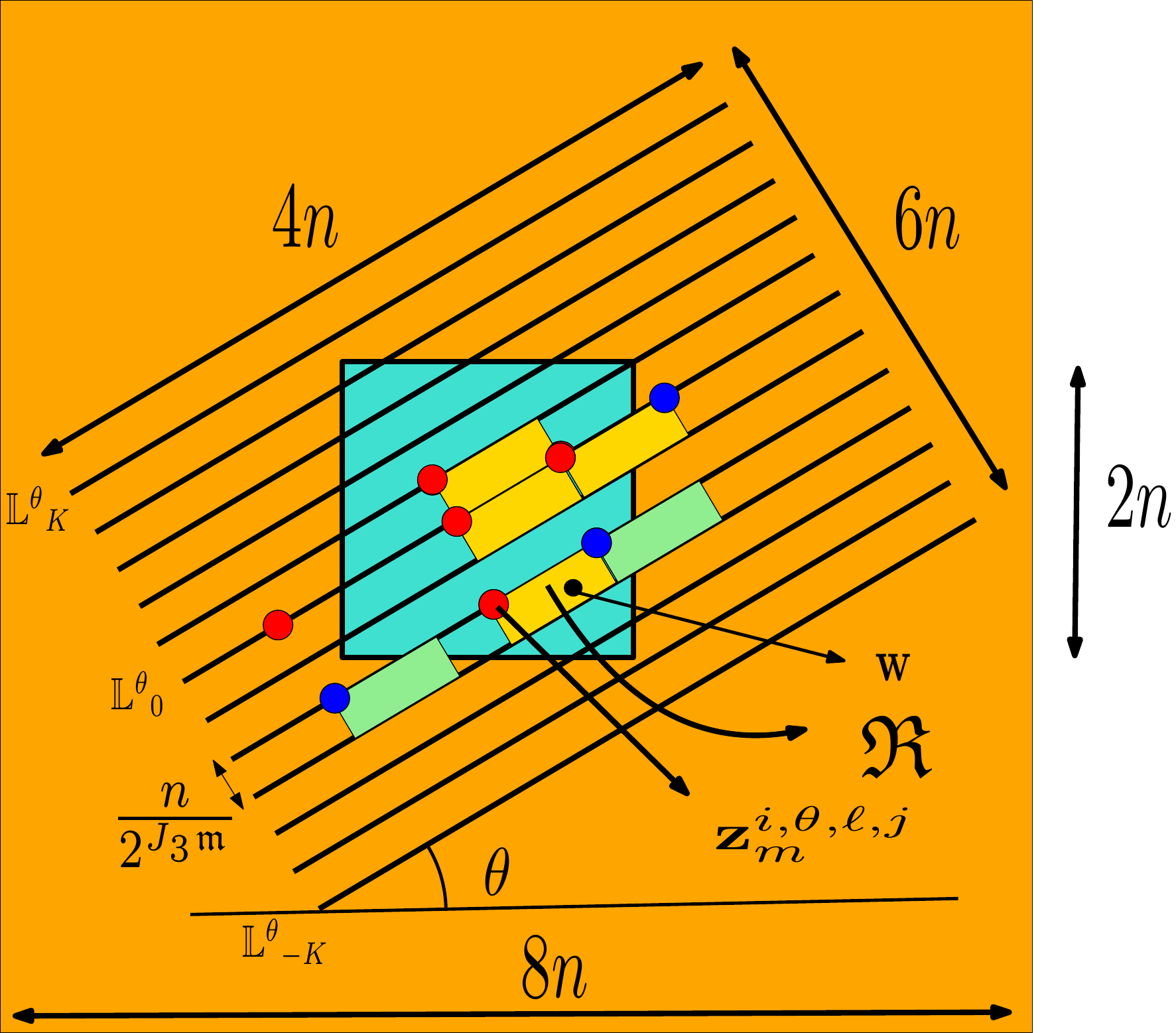}
\caption{This figure illustrates the set of parallel lines $\mathfrak{L}_{\theta}:=\{\mathbb{L}^{\theta}_{-K}, \ldots, \mathbb{L}^{\theta}_{-1}, \mathbb{L}^{\theta}_0,\mathbb{L}^{\theta}_1, \ldots, \mathbb{L}^{\theta}_K\}$. 
The red and blue dots denote the points $\bz^{i,\theta,\ell, j}_{h}$ for $i \in \llbracket-K,K \rrbracket , \theta \in \mathbb{S}^{1}(\eta), \ell \in \llbracket 0,M-1\rrbracket ,$ and $h \in \llbracket 1,2^{(j-J_1)\fm}\rrbracket$. The red and blue colors denote whether the point is  $(\delta_2,\theta,\frac{n}{2^{j\fm}},k_1)-\SST$ or not respectively. For each such point we associate a rectangular box with one of the sides parallel to $\mathbb{L}^{\theta}_{0}$, where the point is at the north-west corner of the associated rectangle. A particular example of a point $\bz^{i,\theta,\ell, j}_{h}$  and the associated rectangle $\mathfrak{R}$  and a lattice point $\bf$ inside $\mathfrak{R}$ are marked in the figure. The green boxes are associated to the blue points and the yellow boxes are associated to the red points. 
}
\label{fig13}
\end{figure}
Instead we do the following:  for each $\theta\in \bS^1(\eta),$ consider the set of parallel lines $$\mathfrak{L}_{\theta}:=\{\mathbb{L}^{\theta}_{-K}, \ldots, \mathbb{L}^{\theta}_{-1}, \mathbb{L}^{\theta}_0,\mathbb{L}^{\theta}_1, \ldots, \mathbb{L}^{\theta}_K\}$$  where for any $i\in \llbracket-K,K\rrbracket,$  $\mathbb{L}^{\theta}_i$ is a line segment of length $4n$, making angle $\theta$ with the $x-$axis; $\mathbb{L}^{\theta}_0$ is centered at the origin; and  $\mathbb{L}^{\theta}_i$ is obtained by translating $\mathbb{L}^{\theta}_{0}$ in the orthogonal direction by  $\frac{i n}{2^{J_3\fm}}$ where $J_3=J_1+\frac{1}{\delta_3^4}$ and $K=3 \times 2^{J_3 \fm}$, (see Figure \ref{fig13}). 
For each $\theta \in \mathbb{S}^{1}(\eta),$ and each $i\in \llbracket-K,K \rrbracket$
let $\sS_{i,\theta}$ be the discrete line segment formed by the points on $\mathbb{L}^{\theta}_i$ at spacing $\frac{n}{2^{J_1\fm}}$ (without loss of generality we assume that the starting and ending points of $\mathbb{L}^{\theta}_i$ and $\sS_{i,\theta}$ are the same to avoid rounding issues). 
Thus $\sS_{i,\theta}=[\bz^{i,\theta}_0,\bz^{i,\theta}_1,\ldots,\bz^{i,\theta}_{M}]$ where $M=4 \times 2^{J_1\fm}.$ 
We now create a tree $\cT_{i,\theta,\ell}$ for each $i\in \llbracket-K,K\rrbracket, \theta \in \mathbb{S}^{1}(\eta), \ell \in \llbracket 0,M-1 \rrbracket$ corresponding to the interval $[\bz^{i,\theta}_{\ell},\bz^{i,\theta}_{\ell+1}]$ as in \eqref{treecon12}. 
As before for any $j\ge J_1$, let $L^{i,\theta,\ell}_j$ denote the $j^{th}$ level of the tree $\cT_{i,\theta,\ell}$ and $L_j=\bigcup_{i,\theta,\ell} L^{i,\theta,\ell}_j$.

Running the same argument as before with these trees in place of the ones in \eqref{treecon12} now gives us  $J_1\le j\le J_1+\frac{1}{\delta_3^2}$ with the following property.
If $$\sS_{i,\theta,\ell,j}=[\bz^{i,\theta,\ell, j}_{0},\bz^{i,\theta,\ell, j}_{1},\ldots, \bz^{i,\theta,\ell, j}_{2^{(j-J_1)\fm}}]$$ denotes the discrete segment corresponding to $L^{i,\theta,\ell}_j,$ i.e., the $2^{(j-J_1)\fm}$ vertices in the latter correspond to the intervals $[\bz^{i,\theta,\ell, j}_{h},\bz^{i,\theta,\ell, j}_{h+1}]$ for $i \in \llbracket-K,K \rrbracket , \theta \in \mathbb{S}^{1}(\eta), \ell \in \llbracket 0,M-1\rrbracket ,$ and $h \in \llbracket 1,2^{(j-J_1)\fm}\rrbracket$. Then  then for any  $k_1$ (to be specified soon and small enough compared to $J_1$) for any $\theta \in \bS^1(\eta)$ except for at most $O(\frac{k_1}{2^{\fm}}+\frac{2^{\fm}\sqrt \delta_3}{\eta})$ fraction, all the remaining  $\bz^{i,\theta,\ell, j}_{h}$, are $(\delta',\theta,\frac{n}{2^{j\fm}}, k_1)-\SST$  where $\delta'=O(2^\fm \delta_3^{1/4}).$

Thus  by choosing $\fm$ large enough followed by $\delta_3$ small enough, provides  
for any $\theta \in \mathbb{S}^1(\eta)$ a dense set of points at spacing $\frac{n}{2^{j\fm}}$ which are  $(\delta_2,\theta,\frac{n}{2^{j\fm}},k_1)-\SST$
and hence addresses the issue in (a).

To address the issue in (b) we will use the above along with Lemma \ref{stab23} to imply  stability for most points in $\Lb(n)$ with slightly worse parameters. 
Fixing $\theta \in \bS^1(\eta),$ for any $(\delta_2,\theta,\frac{n}{2^{j\fm}},k_1)-\SST$ $\bz^{i,\theta,\ell, j}_{h}$, consider any lattice point $\bw$ in the associated rectangular box $\mathfrak{R}$ as illustrated in Figure \ref{fig13}.  Thus $|\bw-\bz^{i,\theta,\ell, j}_{h}|\le 2\frac{n}{2^{j\fm}}$. Hence applying Lemma  \ref{stab23} (by taking $\ell=\frac{n}{2^{j\fm}}$, $m=2$, $k=k_1$ and $C=\sqrt{k_1}$)  now implies that:
$$\#\{\bz\in \Lb(n): \bz~\text{is not}~(\delta', \theta,\frac{n\sqrt k_1}{2^{j\fm}},\sqrt k_1)-\SST\} \leq O(\frac{k_1}{2^{\fm}}+\frac{2^{\fm}\sqrt \delta_3}{\eta})n^2 ,$$
where $\delta'=\delta_2+O(\frac{1}{\sqrt k_1}).$
By a simple union bound over $\theta \in \mathbb{S}^{1}(\eta)$ it follows that  
\begin{equation}\label{lb123456}\#\{\bz\in \Lb(n): \bz~\text{is not}~(\delta',\mathbb{S}^1(\eta), \frac{n\sqrt k_1}{2^{j\fm}},\sqrt k_1)-\SST\} \leq O\left(\frac{1}{\eta}(\frac{k_1}{2^{\fm}}+\frac{2^{\fm}\sqrt \delta_3}{\eta})\right) n^2 .
\end{equation}

The statement of Theorem \ref{t:stablegen} now follows from choosing $\sqrt{k_1}\gtrsim \max(\frac{1}{\delta},k)$ followed by $\delta_{2}$ small enough to ensure $\delta'\le \delta.$ 
and then $\fm$ large enough followed by $\delta_3$ small enough to ensure that $O\left(\frac{1}{\eta}(\frac{k_1}{2^{\fm}}+\frac{2^{\fm}\sqrt \delta_3}{\eta})\right)$ is less than $\e.$ 
Moreover we take the value of $J_2$ to be  $\mathfrak{m}(J_1+\frac{1}{\delta_3^2})$. Note that the value of $j$ in Theorem \ref{t:stablegen} is the value  $j\fm -\frac{\log k_1}{2}$ appearing in \eqref{lb123456}.

\bibliography{delocalization}

\begin{thebibliography}{10}

\bibitem{AD14}
Antonio Auffinger, Michael Damron, et~al.
\newblock A simplified proof of the relation between scaling exponents in
  first-passage percolation.
\newblock {\em The Annals of Probability}, 42(3):1197--1211, 2014.

\bibitem{ADH15}
Antonio Auffinger, Michael Damron, and Jack Hanson.
\newblock 50 years of first passage percolation.
\newblock {\em arXiv preprint arXiv:1511.03262}, 2015.

\bibitem{BDJ99}
Jinho Baik, Percy Deift, and Kurt Johansson.
\newblock On the distribution of the length of the longest increasing
  subsequence of random permutations.
\newblock {\em J. Amer. Math. Soc}, 12:1119--1178, 1999.

\bibitem{BGS17A}
Riddhipratim Basu, Shirshendu Ganguly, and Allan Sly.
\newblock Delocalization of geodesics in the lower tail large deviation in last
  passage percolation.
\newblock In preparation.

\bibitem{BKS04}
Itai Benjamini, Gil Kalai, and Oded Schramm.
\newblock First passage percolation has sublinear distance variance.
\newblock {\em Ann. Probab.}, 31(4):1970--1978, 10 2003.

\bibitem{Cha11}
Sourav Chatterjee.
\newblock The universal relation between scaling exponents in first-passage
  percolation.
\newblock {\em Ann. Math. (2)}, 177(2):663--697, 2013.

\bibitem{CZ}
Yunshyong Chow and Yu~Zhang.
\newblock Large deviations in first-passage percolation.
\newblock {\em Annals of Applied Probability}, pages 1601--1614, 2003.

\bibitem{CD81}
J.~Theodore Cox and Richard Durrett.
\newblock Some limit theorems for percolation processes with necessary and
  sufficient conditions.
\newblock {\em Ann. Probab.}, 9(4):583--603, 08 1981.

\bibitem{cranston}
M~Cranston, D~Gauthier, and TS~Mountford.
\newblock On large deviations for the parabolic anderson model.
\newblock {\em Probability theory and related fields}, 147(1):349--378, 2010.

\bibitem{DH14}
Michael Damron and Jack Hanson.
\newblock Busemann functions and infinite geodesics in two-dimensional
  first-passage percolation.
\newblock {\em Communications in Mathematical Physics}, 325(3):917--963, 2014.

\bibitem{DH17}
Michael Damron and Jack Hanson.
\newblock Bigeodesics in first-passage percolation.
\newblock {\em Communications in Mathematical Physics}, 349(2):753--776, 2017.

\bibitem{DZ1}
Jean-Dominique Deuschel and Ofer Zeitouni.
\newblock On increasing subsequences of iid samples.
\newblock {\em Combinatorics, Probability and Computing}, 8(03):247--263, 1999.

\bibitem{HW65}
J.~M. Hammersley and D.~J.~A. Welsh.
\newblock {\em First-Passage Percolation, Subadditive Processes, Stochastic
  Networks, and Generalized Renewal Theory}, pages 61--110.
\newblock 1965.

\bibitem{H08}
Christopher Hoffman.
\newblock Geodesics in first passage percolation.
\newblock {\em The Annals of Applied Probability}, 18(5):1944--1969, 2008.

\bibitem{Jensen00}
L~Jensen.
\newblock {\em The asymmetric exclusion process in one dimension}.
\newblock PhD thesis, Ph. D. dissertation, New York Univ., New York, 2000.

\bibitem{Jo99}
Kurt Johansson.
\newblock Shape fluctuations and random matrices.
\newblock {\em Communications in Mathematical Physics}, 209(2):437--476, 2000.

\bibitem{KPZ86}
Mehran Kardar, Giorgio Parisi, and Yi-Cheng Zhang.
\newblock Dynamic scaling of growing interfaces.
\newblock {\em Phys. Rev. Lett.}, 56:889--892, 1986.

\bibitem{Kes86}
Harry Kesten.
\newblock {\em {\'E}cole d'{\'E}t{\'e} de Probabilit{\'e}s de Saint Flour XIV -
  1984}, chapter Aspects of first passage percolation, pages 125--264.
\newblock 1986.

\bibitem{Kes87}
Harry Kesten.
\newblock Percolation theory and first-passage percolation.
\newblock {\em Ann. Probab.}, 15(4):1231--1271, 10 1987.

\bibitem{Kes93}
Harry Kesten.
\newblock On the speed of convergence in first-passage percolation.
\newblock {\em Ann. Appl. Probab.}, 3(2):296--338, 1993.

\bibitem{K73}
J.~F.~C. Kingman.
\newblock Subadditive ergodic theory.
\newblock {\em Ann. Probab.}, (6):883--899, 12 1973.

\bibitem{New95}
Charles~M Newman.
\newblock A surface view of first-passage percolation.
\newblock In {\em Proceedings of the International Congress of Mathematicians},
  pages 1017--1023. Springer, 1995.

\bibitem{OT17}
Stefano Olla and Li-Cheng Tsai.
\newblock Exceedingly large deviations of the totally asymmetric exclusion
  process.
\newblock {\em arXiv preprint arXiv:1708.07052}, 2017.

\bibitem{R73}
Daniel Richardson.
\newblock Random growth in a tessellation.
\newblock In {\em Mathematical Proceedings of the Cambridge Philosophical
  Society}, volume~74, pages 515--528, 1973.

\bibitem{Ro81}
H.~Rost.
\newblock Nonequilibrium behaviour of a many particle process: Density profile
  and local equi- libria.
\newblock {\em Zeitschrift f. Warsch. Verw. Gebiete}, 58(1):41--53, 1981.

\bibitem{Tal94}
Michel Talagrand.
\newblock Concentration of measure and isoperimetric inequalities in product
  spaces.
\newblock {\em Publications Mathematiques de l'IHES}, 81(1):73--205, 1995.

\bibitem{varadhan1}
Srinivasa~RS Varadhan.
\newblock Large deviations for the asymmetric simple exclusion process.
\newblock {\em Stochastic analysis on large scale interacting systems},
  39:1--27, 2004.

\end{thebibliography}
\bibliographystyle{plain}

\end{document}